\documentclass[12pt,letter,reqno]{amsart}

\usepackage[text={420pt,660pt},centering]{geometry}

\usepackage{amssymb,amsfonts,amsthm,mathrsfs}
\usepackage{marginnote}
\usepackage{comment}
\usepackage{enumitem}
\usepackage{graphicx}
\usepackage{bm}
\usepackage{xfrac}

\usepackage[dvipsnames]{color}

\usepackage[colorlinks=true, pdfstartview=FitV, linkcolor=black, citecolor=blue, urlcolor=blue]{hyperref}

\graphicspath{{figures1D/}}

\def\Xint#1{\mathchoice
{\XXint\displaystyle\textstyle{#1}}%
{\XXint\textstyle\scriptstyle{#1}}%
{\XXint\scriptstyle\scriptscriptstyle{#1}}%
{\XXint\scriptscriptstyle%
\scriptscriptstyle{#1}}%
\!\int}
\def\XXint#1#2#3{{\setbox0=\hbox{$#1{#2#3}{%
\int}$ }
\vcenter{\hbox{$#2#3$ }}\kern-.6\wd0}}
\def\barint{\, \Xint -} 
\def\bariint{\barint_{} \kern-.4em \barint}
\def\bariiint{\bariint_{} \kern-.4em \barint}
\renewcommand{\iint}{\int_{}\kern-.34em \int} 
\renewcommand{\iiint}{\iint_{}\kern-.34em \int} 

\DeclareMathAlphabet{\mathcal}{OMS}{cmsy}{m}{n}

\theoremstyle{plain}

\newtheorem{theorem}{Theorem}[section]

\newtheorem{lemma}[theorem]{Lemma}

\newtheorem{corollary}[theorem]{Corollary}
\newtheorem{proposition}[theorem]{Proposition}

\theoremstyle{definition}
\newtheorem{remark}[theorem]{Remark}

\newcommand{\R}{\mathbb{R}}

\newcommand{\N}{\mathbb{N}}
\newcommand{\Z}{\mathbb{Z}}
\newcommand{\T}{\mathbb{T}}

\newcommand{\bI}{\mathbf{I}}

\newcommand{\p}{\partial}

\newcommand{\norm}[1]{\lVert #1 \rVert}

\let\div\relax
\DeclareMathOperator{\div}{div}

\DeclareMathOperator{\sgn}{sgn}
\let\tilde\relax
\newcommand{\tilde}[1]{\widetilde{#1}}
\let\hat\relax
\newcommand{\hat}[1]{\widehat{#1}}
\renewcommand{\bar}[1]{\overline{#1}}
\newcommand{\abs}[1]{\left\lvert #1 \right\rvert}

\newcommand{\mc}[1]{\mathcal{#1}}

\renewcommand{\T}{\mathbb{T}}

\numberwithin{equation}{section}
\setlist[enumerate]{leftmargin=*}

\title[]{Boundary layers and vanishing diffusivity in run-and-tumble models}

\author[Albritton]{Dallas Albritton} 
\address[Dallas Albritton]{University of Wisconsin-Madison, Department of Mathematics, 480 Lincoln Dr, Madison, WI 53706, USA}
\email{dalbritton@wisc.edu}

\author[Ohm]{Laurel Ohm}
\address[Laurel Ohm]{University of Wisconsin-Madison, Department of Mathematics, 480 Lincoln Dr, Madison, WI 53706, USA}
\email{lohm2@wisc.edu}

\author[Yastrzhembskiy]{Timur Yastrzhembskiy}\address[Timur Yastrzhembskiy]{Academia Sinica, Taiwan}
\email{yastr@as.edu.tw}

\date\today

\begin{document}

\begin{abstract}
A notable feature of confined active matter systems is the tendency for motile particles to accumulate near solid boundaries.
In various linear models with no-flux boundary conditions, this accumulation is realized through the development of sharp boundary layers at small particle diffusivity $\kappa$. In this paper, we present the first rigorous investigation of \emph{nonlinear} boundary layers in the context of confined active matter. Specifically, we consider a family of 1D run-and-tumble models with nonlinear advection and tumbling on the half-line $\R_+$. We rigorously prove the vanishing diffusivity limit with quantitative convergence rates. In the limiting system, the boundary mass enters as a new variable which solves a nonlinear ODE, coupled to the PDE through a dynamic boundary condition. Interestingly, the nonlinearity on the boundary at $\kappa = 0$ cannot be obtained without reference to the boundary layer analysis at $\kappa \ll 1$. Numerically, these models exhibit rich behavior, including phase transition and hysteresis in the boundary layer. 
\end{abstract}

\maketitle

\setcounter{tocdepth}{1}
\tableofcontents

\parskip   2pt plus 0.5pt minus 0.5pt


\section{Introduction}
\label{sec:introduction}

We consider a family of 1D run-and-tumble models, the simplest example of which is
\begin{equation}
    \label{eq:telegrapher}
\begin{aligned}
        \p_t c_+ + \p_x c_+ &= \kappa \p_x^2 c_+ - c_+ + c_- \\
        \p_t c_- - \p_x c_- &= \kappa \p_x^2 c_- + c_+ - c_- \, .
\end{aligned}
\end{equation}
The variables $c_{\pm}$ represent concentrations of left- and right-moving (running) agents, which may spontaneously change direction (tumbling). The concentrations are moreover subject to weak translational diffusion with diffusivity $0 < \kappa \ll 1$.

Models of this type arise naturally as simple descriptions of various active matter systems, including myxobacteria swarms \cite{ScheelStevensWavenumberSelection,KangScheelStevens,flynn2020self}, suspensions of swimming bacteria \cite{tailleur2008statistical,cates2015motility}, and active Brownian particles \cite{bressloff2025stochastic, bruna2022phase,bechinger2016active,evans2018run}. The system~\eqref{eq:telegrapher} in particular is sometimes known as the generalized \emph{telegrapher's equation} with diffusion \cite{malakar2018steady, angelani2015run, weiss2002some}, originating in electromagnetic theory, or as a version of the Goldstein-Taylor model \cite{bruna2022phase} used as a prototypical simplified kinetic equation.

Our focus will be on boundary effects, and we therefore consider~\eqref{eq:telegrapher} on the half-line $\R_+ := \{ x > 0 \}$ with a no-flux boundary condition at $x=0$:
\begin{equation}
    \label{eq:telegrapherBC}
    (1-\kappa \p_x) c_+\big|_{x=0} = 0 \, , \quad (1+\kappa \p_x) c_-\big|_{x=0} = 0 \, .
\end{equation}
When $0 < \kappa \ll 1$, left-moving agents are observed to accumulate at the ``wall" in a boundary layer of width $O(\kappa)$, where they wait to reverse direction (see Figure~\ref{fig:compare}). When $\kappa \to 0^+$, the mass in the boundary layer becomes a new variable, $b_-(t)$, and the system~\eqref{eq:telegrapher}-\eqref{eq:telegrapherBC} becomes\footnote{We often abuse terminology and refer to this system as ``inviscid", although it is perhaps more accurately ``non-diffusive".}
\begin{equation}
    \label{eq:inviscidtelegrapher}
    \p_t c_{\pm} \pm \p_x c_{\pm} = \mp c_+ \pm c_-
    \end{equation}
    with the so-called ``sticky" boundary conditions 
    \begin{equation}
        \label{eq:inviscidtelegrapherBC}
    c_+\big|_{x=0} = b_- \, , \quad \dot b_- = -b_- + c_-\big|_{x=0} \, .
    \end{equation}
The model~\eqref{eq:inviscidtelegrapher}-\eqref{eq:inviscidtelegrapherBC} essentially appears in~\cite{angelani2017confined, angelani2023one, bressloff2023encounter, bressloff2025run, bressloff2025stochastic}. These boundary conditions are already interesting in that \emph{the boundary layer enters as a new variable in the inviscid problem}. This is in contrast to various singular perturbation problems, including run-and-tumble processes with Dirichlet (absorbing) boundary conditions, for which the boundary layer disappears in the limit. 
\begin{figure}[!ht]
    \centering
    \includegraphics[scale=0.4]{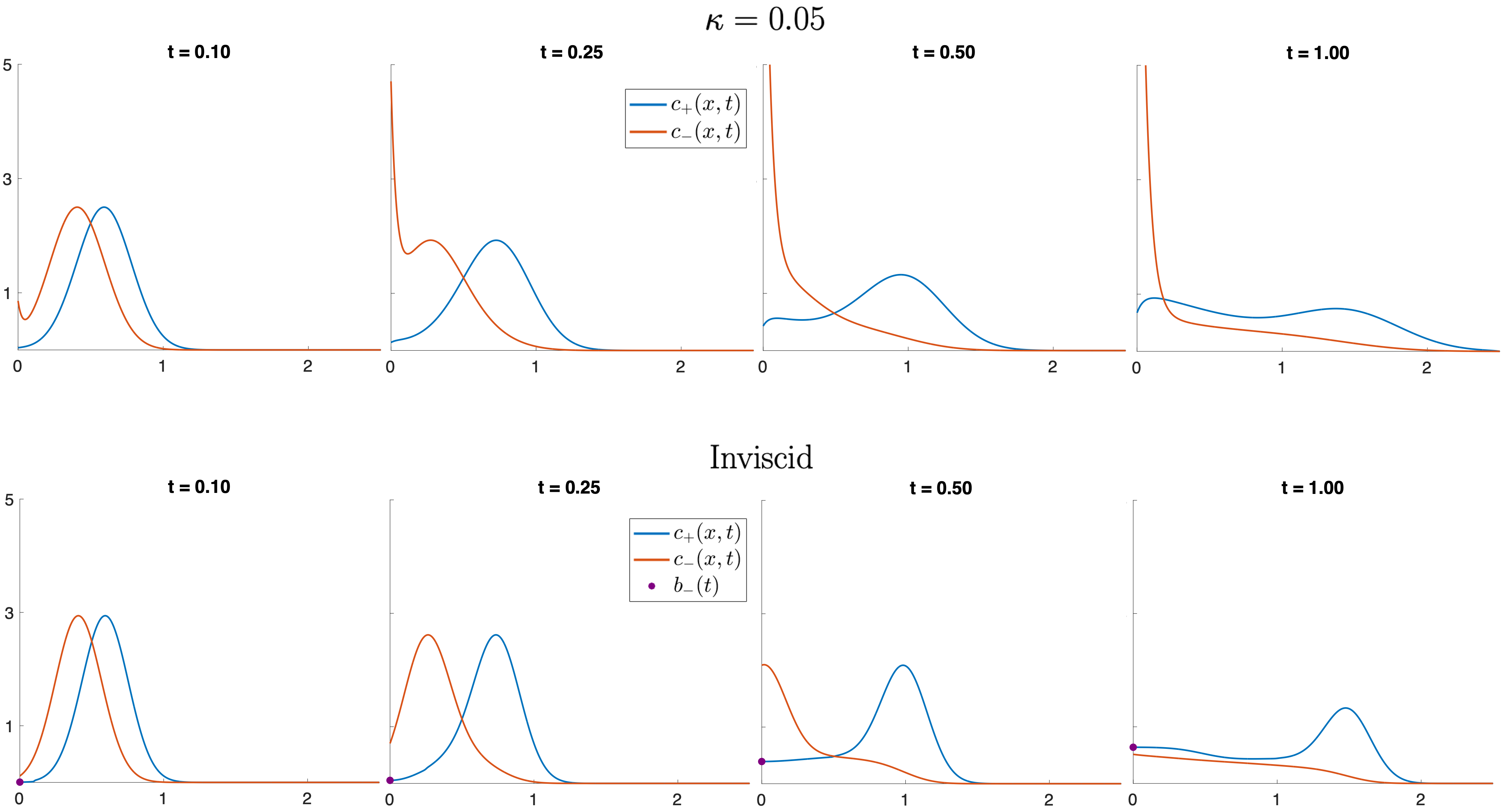}
    \caption{Comparison of the viscous dynamics (top) of \eqref{eq:telegrapher}-\eqref{eq:telegrapherBC} using $\kappa=0.05$ with the inviscid dynamics (bottom) of \eqref{eq:inviscidtelegrapher}-\eqref{eq:inviscidtelegrapherBC}. Here the initial condition is $c_+=c_-=3e^{-20(x-\frac{1}{2})^2}$. In the viscous dynamics, we note the immediate appearance of an $O(\kappa)$ boundary layer for $c_-(x,t)$ about $x=0$. In both cases, the particles all eventually leave the system at $x=+\infty$, resulting in only the trivial steady state $c_+=c_-=0$.  }
    \label{fig:compare}
\end{figure}

More realistically, agents are subject to nonlinear effects, a selection of which we consider here. First, agents interact non-locally by generating a macroscopic velocity field $V$ which captures, for example, the tendency for certain species to aggregate. Second, the tumbling rate may be nonlinear and depend, e.g., upon the probability of encountering an oppositely-oriented agent.

To incorporate these nonlinear effects, we introduce the family of equations
\begin{equation}
    \label{eq:viscousequation}
\begin{aligned}
    \p_t c_+ + \p_x ((V+\beta_+) c_+) &= \kappa \p_x^2 c_+ + f(c_+,c_-) \\
    \p_t c_- + \p_x ((V-\beta_-) c_-) &= \kappa \p_x^2 c_- - f(c_+,c_-)
\end{aligned} 
\end{equation}
with no-flux boundary conditions
\begin{equation}\label{eq:nofluxBCs}
\begin{aligned}
    (V+\beta_+) c_+ - \kappa \p_x c_+ &= 0 \\
    (V-\beta_-) c_- - \kappa \p_x c_- &= 0
\end{aligned} \quad  \text{ at } x=0 \, .
\end{equation}
Here $\beta_+,\beta_- > 0$ are the possibly different propulsion speeds. Different speeds may arise naturally in the presence of a non-zero mean flow; in this case, the boundary condition at $x=0$ acts as a ``filter" through which the background medium may flow but which the agents cannot penetrate.

\emph{Non-local effects}. Let $\rho := c_+ + c_-$ represent the total density. We consider a velocity operator $V[\rho]$ induced by an interaction kernel $K(x,y)$:\footnote{
Sometimes a local term $W(\rho)$ is included in the velocity~\cite{mogilner1999non,topaz2004swarming}; however, our methods used to obtain convergence in Section~\ref{sec:vanishingdiff} do not seem well-adapted to this term.
}
\begin{equation}\label{eq:Vexample}
    V[\rho](x) := \int_{\R_+} K(x,y) \rho(y) \, dy \, .
\end{equation}
Here we suppose
\begin{equation}
    \label{eq:smoothnessofK}
K(x,y) \in C^\infty([0,+\infty)^2)
\end{equation}
and further assume decay conditions on the kernel, namely,
\begin{equation}
    \label{eq:decayonkernel}
\exists \,\eta > 0 \text{ s.t. } \p_x^i \p_y^j K(x,y) \lesssim_{i,j} \langle x-y \rangle^{-\eta}  \quad \forall i, j \in \N_0 \, .
\end{equation}
A typical example would be
\begin{equation}
    \label{eq:convolutiontype}
    K(x,y) = H(x-y) \, ,
\end{equation}
where $H : \R \to \R$ is a smooth function decaying algebraically at infinity; commonly $H$ is odd and attractive, i.e., $H\big|_{x \geq 0} \leq 0$. (See, e.g., \cite{mogilner1999non,topaz2004swarming,topaz2006nonlocal} for examples.) We are sometimes interested in velocities satisfying the boundary condition $V\big|_{x=0} = 0$, in which case we should impose $K(0,y) = 0$. \emph{Supposing that $H$ is odd}, i.e., $H(z) =- H(-z)$, a typical example is obtained via reflection:
\begin{equation}
    \label{eq:reflectedkernel}
    K(x,y) = H(x-y) + H(x+y) \, .
\end{equation}
One way to view the boundary condition $V\big|_{x=0} = 0$ is as a one-dimensional caricature of the no-slip condition~\eqref{eq:noslipcond} for a particle-generated fluid flow $\bm{u}$.

\emph{Nonlinear tumbling}. We study reaction terms of the following form:
\begin{equation}
    \label{eq:fintermsofr}
    f(c_+,c_-) = - c_+ r(c_-) + c_- r(c_+) \, ,
\end{equation}
where $r : [0,+\infty) \to (0,+\infty)$ is a positive smooth function. In particular, the rate of orientation-reversal is dependent on encounters with oppositely-oriented agents, which is relevant to the dynamics of myxobacteria (see \cite{LuSte02}). Such models are common, see, e.g., \cite{LuSte02,ScheelStevensWavenumberSelection,KangScheelStevens,flynn2020self} and references therein. We further assume
\begin{equation}
    \label{eq:tumblingassumptiononr}
\sup_{u \geq 0} |\p_u^k r(u)| < +\infty  \quad \forall k \in \N_0 \,.     
\end{equation}
A typical example is
\begin{equation}\label{eq:example_r}
    r(u) = r_0 + \frac{r_1 u^2}{1+\zeta u^2} \, ,
\end{equation}
where $r_0 > 0$, $r_1 \in \R$, and $\zeta > 0$. We think of such nonlinearities as \emph{saturated}.\footnote{An alternative class of tumbling operators would be $f(c_+,c_-) = (-c_+ + c_-) r(c_++c_-)$, so that the rate depends on the total density. This would lead to further complications in the asymptotic analysis related to the determination of $C_{\pm}^{(0)}$, $C_{\pm}^{(1)}$ (see Section~\ref{subsec:innersol}).}

The inviscid system corresponding to~\eqref{eq:viscousequation}-\eqref{eq:nofluxBCs}, with $V$ and $f$ given by~\eqref{eq:Vexample} and~\eqref{eq:fintermsofr}, respectively, is
\begin{equation}
    \label{eq:inviscidequation}
\begin{aligned}
    \p_t c_+ + \p_x ((V+\beta_+) c_+) &=  f(c_+,c_-) \\
    \p_t c_- + \p_x ((V-\beta_-) c_-) &= - f(c_+,c_-)
\end{aligned}
\end{equation}
\begin{equation}\label{eq:Vinviscid}
    V[\rho,b](x) = \int_{\R_+} K(x,y) \rho(y) \, dy + K(x,0) b_- \, .
\end{equation}
In analogy to~\eqref{eq:inviscidtelegrapherBC}, this system should be supplemented with equations for the variable $b_-(t)$ representing the mass of the left-moving agents at the boundary. We are primarily concerned with two cases, for which we can rigorously prove the vanishing diffusivity limit:

\emph{Case 1. $r \equiv r_0$ is constant.} In this case, the relevant boundary conditions are directly related to the telegrapher conditions~\eqref{eq:inviscidtelegrapherBC}:\footnote{
Notice that the conditions~\eqref{eq:bminuseqn}-\eqref{eq:cinv_plus_bval} are consistent with the dimension counting $[c] = M/L$ (mass/length), $[b] = M$, $[\beta_{\pm}] = [V] = L/T$, and $[r_0] = 1/T$. When $V=0$, the telegrapher's equation~\eqref{eq:telegrapher} is obtained by non-dimensionalizing time by $1/r_0$ and length$/$time by $\beta$.}
\begin{equation}\label{eq:bminuseqn}
    \dot b_- = (\beta_- - V\big|_{x=0}) c_-\big|_{x=0} - r_0 b_-
\end{equation}
\begin{equation}\label{eq:cinv_plus_bval}
    c_+\big|_{x=0} = \frac{r_0 b_-}{(V\big|_{x=0}+\beta_+)} \, .
\end{equation}
Equation~\eqref{eq:bminuseqn} reflects that the boundary gains mass from left-moving agents and loses mass due to tumbling and subsequent movement to the right, as captured by~\eqref{eq:cinv_plus_bval}. In view of this, we only consider the inviscid problem while the conditions
\begin{equation}\label{eq:V0condition}
    V\big|_{x=0}-\beta_-<0 \, , \quad V\big|_{x=0}+\beta_+ > 0
\end{equation}
remain satisfied.

\emph{Case 2. $V\big|_{x=0} = 0$.} Here, due to the nonlinear tumbling, the relevant boundary conditions now involve a \emph{nonlinear} ODE for the boundary mass:
\begin{equation}\label{eq:bminuseqncase2}
    \dot b_- = - \beta_+ c_\infty(\beta_- b_-,\beta_+,\beta_-) + \beta_- c_-\big|_{x=0}
\end{equation}
\begin{equation}\label{eq:cinv_plus_bvalcase2}
    c_+\big|_{x=0} = c_\infty(\beta_- b_-,\beta_+,\beta_-)
\end{equation}
where
\begin{equation} \label{eq:cinftydef}
c_\infty(q,\gamma,\nu) := \lim_{X \to +\infty} A(q,\gamma,\nu)(X) 
\end{equation}
is the value at $X=+\infty$ of a certain bounded solution $A(q,\gamma,\nu)(X)$ to the inner problem
\begin{equation}
    \label{eq:Aequationintro}
\begin{aligned}
    \gamma \p_X A - \p_X^2 A &= q e^{-\nu X} r(A) \\
    (\gamma - \p_X) A\big|_{X=0} &= 0 \, ,
\end{aligned}
\end{equation}
which arises in the determination of the structure of $c_+$ in the boundary layer. We term this equation~\eqref{eq:Aequationintro} the \emph{incoming ODE}, for agents reentering the domain.

In either case, the dynamic boundary conditions can be written in a unified way:
\begin{equation}
\begin{aligned}\label{eq:cplusunifiedbc}
  c_+\big|_{x=0} &= c_\infty \big((\beta_- - V\big|_{x=0}) b_-,V\big|_{x=0}+\beta_+,\beta_--V\big|_{x=0}\big), \\
  \dot b_- &= - (V\big|_{x=0} + \beta_+) c_\infty\left((\beta_- - V\big|_{x=0}) b_-,V\big|_{x=0}+\beta_+,\beta_--V\big|_{x=0}\right)\\
    &\qquad + (\beta_- - V\big|_{x=0}) c_-\big|_{x=0} \, ,
\end{aligned}
\end{equation}
 where $c_\infty$ is explicit in Case~1; see~\eqref{eq:Aincase1}.

Not only does the nonlinear tumbling enter as a nonlinearity in the ODE for $b_-$, but moreover, \emph{it is not possible to determine the nonlinearity $c_\infty$ without appealing to the structure of the diffusive boundary layer}. This is in contrast to~\eqref{eq:bminuseqn}-\eqref{eq:cinv_plus_bval}, which in principle can be correctly ``guessed" directly at the inviscid level (see \cite{angelani2017confined}).

When $q$ is sufficiently small, the equation~\eqref{eq:Aequationintro} for $A$ is uniquely solvable. However, for certain choices of $r$, we observe numerically that equation~\eqref{eq:Aequationintro} has non-unique solutions -- see Figure \ref{fig:ODEforA}. In such cases, one expects that the ``correct" choice should be a dynamically stable solution to~\eqref{eq:Aequationintro}. 

A valuable check on the consistency of the inviscid equations is that the total mass is conserved. For example, in Case~1, we have 
\begin{equation}
    \frac{d}{dt} (b_- + \int_{\R_+} \rho) = \dot b_- + (V+\beta_+) c_+\big|_{x=0} + (V-\beta_-) c_-\big|_{x=0} = 0 \, ,
\end{equation}
and similarly in Case 2. \\

To quantify the convergence, we define the \emph{Fortet-Mourier} (bounded Lipschitz) \emph{norm}
\begin{equation} \label{eq:fmnorm}
   \| \mu \|_{\rm FM (\overline{\R}_{+})} := \sup_\phi  \int \phi \, d\mu  \, , \quad \forall \mu \in \mathcal{M}_b(\overline{\R}_+) \, ,
\end{equation}
where $\mathcal{M}_b(\overline{\R}_+)$ are finite (signed) Radon measures, and the supremum is over $\phi \in W^{1,\infty}(\R_+)$ with $\| \phi \|_{W^{1,\infty}(\R_+)} \leq 1$. It is well known that the convergence in the norm \eqref{eq:fmnorm} is equivalent to the weak (narrow) convergence of nonnegative measures $\mu_n \Rightarrow \mu$ (see~\cite[Theorem 8.3.2]{BogachevMeasureTheory}): 
\begin{equation}
 \mu_n \Rightarrow \mu \quad \text{ means } \quad \lim_{n \to \infty}  \int f \, d\mu_n  = \int f \, d\mu \, , \;  \forall f \in C_b (\bar \R_{+})\,,
\end{equation}
which is fundamental in probability theory.

\begin{theorem}
    \label{thm:mainthm}
Suppose that $K$ satisfies~\eqref{eq:smoothnessofK}-\eqref{eq:decayonkernel} and $f$ satisfies~\eqref{eq:fintermsofr}-\eqref{eq:tumblingassumptiononr}. Suppose that Case~1 ($r \equiv r_0 > 0$ const.) or Case~2 ($K(0,y) = 0$) holds.

Let $c^{\rm in}_{\pm} \in W^{2,1}(\R_+)$ satisfy the compatibility conditions~\eqref{eq:compat1}-\eqref{eq:compat3} with $b_-^{\rm in}=0$, and let $(c^{\rm inv}_{\pm},b_-)$ be the solution to the inviscid system~\eqref{eq:inviscidequation}-\eqref{eq:Vinviscid} with initial data $(c^{\rm in}_{\pm},0)$ on its maximal interval of existence $[0,T^*)$. Let $c^\kappa_{\pm}$, $\kappa > 0$, be the solution to the diffusive system~\eqref{eq:viscousequation}-\eqref{eq:nofluxBCs} with the same initial data. Then, for each $t \in [0, T^*)$,
\begin{equation}
\label{eq:muinvdef}
\left\lbrace
\begin{aligned}
& c^\kappa_+ \Rightarrow c_+^{\rm inv} =: \mu^{\rm inv}_+ \\
    & c^\kappa_- \Rightarrow c_-^{\rm inv} + \delta_0(x) b_-(t) =: \mu^{\rm inv}_-
\end{aligned}
\right.
     \quad \text{ as } \kappa \to 0^+ \, ,
\end{equation}
where $c^\kappa_{\pm}$, $c^{\rm inv}_{\pm}$ are identified with the measures $c^\kappa_{\pm} \, dx$, $c^{\rm inv}_{\pm} \, dx$ on $\overline{\R}_+$, and $\delta_0$ is the Dirac mass at the origin. Moreover, we have the convergence rate
\begin{equation}\label{eq:cappdiffthm}
    \sup_{t \in [0,T]} \| c^{\kappa}_{\pm} - \mu^{\rm inv}_{\pm} \|_{\rm FM (\overline{\R}_{+})} \lesssim_{\theta,T} \kappa^{1-\theta} \, , \quad \forall \theta \in (0,1] \, , \; T \in (0,T^*) \, .
\end{equation}
\end{theorem}

For convenience, we denote $\bm{c} = (c_+,c_-)$. In fact, our analysis produces quantitative estimates on the difference between $\bm{c}^\kappa$ and an approximate solution $\bm{c}^{\rm app}$ constructed via matched asymptotics. Let $T \in (0,T^*)$ and $\theta \in (0,1]$. More precisely, in Case~2, we construct an approximate solution satisfying
\begin{equation}\label{eq:cappdiff0}
    \sup_{t \in [0,T]} \| \bm{c}^\kappa - \bm{c}^{\rm app} \|_{L^1(\R_+)} \lesssim \kappa^{1-\theta} \, .
\end{equation}
In Case~1, our approximate solution satisfies
\begin{equation}\label{eq:cappdiff1}
    \sup_{t \in [0,T]} \| \bm{c}^\kappa - \bm{c}^{\rm app} \|_{\rm FM(\overline{\R}_+)} \lesssim \kappa^{1-\theta} \, .
\end{equation}
In either case, the approximate solutions additionally satisfy
\begin{equation}\label{eq:cappdiff2}
    \sup_{t \in [0,T]} \| c^{\rm app}_{\pm} - \mu^{\rm inv}_{\pm} \|_{\rm FM(\overline{\R}_+)} \lesssim \kappa^{1-\theta} \, ,
\end{equation}
and the convergence in Theorem~\ref{thm:mainthm} follows from the triangle inequality.

To analyze the boundary layer, we introduce the stretched variable $X = \frac{x}{\kappa}$. In this variable, the time derivative $\p_t \bm{c}$ and tumbling $\pm f(c_+,c_-)$ become next-order effects.\footnote{The boundary layer relaxes to a quasi-steady state on a fast timescale $O(\kappa)$, which, if necessary, e.g., for ill-prepared initial data, one could capture by introducing a fast time $T = t/\kappa$.} Our effective ansatz in the boundary layer is
\begin{equation}
\begin{aligned}
    C_-(X,t) &= \frac{1}{\kappa} C^{(0)}_- \,+ && \hspace{-.2cm} C^{(1)}_- \\
    C_+(X,t) &= &&\hspace{-.2cm} C^{(1)}_+ \, ,
\end{aligned}
\end{equation}
where $C^{(0)}_- = q(t) e^{(V|_{x=0}-\beta_-)X}$. This term accounts for nearly all of the boundary layer mass, which becomes $b_- := q/(\beta_- - V\big|_{x=0})$ in the limit. The $C^{(1)}$ terms are responsible for matching to the outer (inviscid) solution, which yields an ODE for $b_-$ (equivalently, $q$). The matched asymptotics predict the correct inviscid system and produce a rigorous approximate solution $\bm{c}^{\rm app}$, predicated on  solvability for the incoming ODE~\eqref{eq:Aequationintro} and the inviscid system. We review both solvability theories in Sections~\ref{sec:Aeqnearly}-\ref{sec:invicsidproblemearly} but delay the (technical) proofs until the end of the paper.

Once ${\bm c}^{\rm app}$ has been constructed, a key difficulty is to establish the \emph{stability} of the construction in the $\kappa \to 0^+$ limit. Roughly speaking, the equation for the difference $\bm{c}^{\kappa} - \bm{c}^{\rm app}$ contains terms like, for example, $(V^{\kappa} - V^{\rm app}) \p_x c^{\rm app}_\pm$, where $\p_x c^{\rm app}_\pm$ acts like an $O(1/\kappa)$ coefficient with no advantageous sign, which can destroy the estimates.\footnote{A further difficulty is that the boundary conditions themselves depend on $V$.} In Case~2, this difficulty is ameliorated by the requirement that $V$ vanishes at the boundary. In Case~1, we instead exploit the \emph{mass variable}
\begin{equation}
    \label{eq:massvariabledefintro}
    m_{\pm}(x) = \int_0^x c_{\pm}(x') \, dx' \, ,
\end{equation}
for which the nonlinear equation becomes
\begin{equation}\label{eq:MAIN_massvars}
\begin{aligned}
    \p_t m_+ + (V+\beta_+) \p_x m_+ = \kappa \p_x^2 m_+ - r_0 (m_+ - m_-) \\
    \p_t m_- + (V-\beta_-) \p_x m_- = \kappa \p_x^2 m_- - r_0 (m_- - m_+)
\end{aligned}
\end{equation}
with Dirichlet conditions
\begin{equation}
    m_{\pm}\big|_{x=0} = 0 \, .
\end{equation}
The change of variables $c_{\pm} \to m_{\pm}$ is very natural from the point of view of probability theory, as $m_{+}+m_{-}$ can be interpreted as a probability distribution function of particles on a half-line. Cumulative variables analogous to~\eqref{eq:massvariabledefintro} have appeared in the context of viscous shocks, see, e.g.,~\cite{GoodmanNonlinearAsymptotic}. Heuristically, the $L^1$ norm of $\bm{m}$ is insensitive to small changes in the boundary layer, compared to the $L^1$ norm of $\bm{c}$. Practically, the bad term mentioned above becomes~$(V^{\kappa} - V^{\rm app}) \p_x m^{\rm app}_\pm$, and $\| \p_x m^{\rm app}_\pm \|_{L^1} \lesssim 1$.

\subsection{Discussion and further questions}

\emph{(i) Pinning and depinning:} The non-diffusive limit and inviscid system are valid provided that the conditions $V\big|_{x=0} + \beta_+ > 0$ and $V\big|_{x=0} - \beta_- < 0$ in~\eqref{eq:V0condition} hold. However, it is not difficult to imagine situations in which $V$ decreases enough such that $V\big|_{x=0} + \beta_+$ becomes \emph{negative};\footnote{This could even be imposed extrinsically, i.e., one simply considers a time-dependent background field $V = V(t)$.} in this case, one would require a variable $b_+(t)$ measuring the mass of $+$ agents pinned to the boundary, and $(b_-,b_+)$ would solve an ODE system. The diffusive problem would involve boundary layers in both $c_+$ and $c_-$. In the opposite scenario, $V$ increases so that $V\big|_{x=0} - \beta_-$ becomes positive; in this case, the $b_-$ mass would move into the interior of the domain (see Figure \ref{fig:bdrymassleaves}), and it would therefore be necessary to deal with \emph{measure-valued solutions}. From this perspective, it would be interesting to establish the vanishing diffusivity limit in a class of solutions which admits this behavior. For relevant work, see \cite{fetecau2022,Zhang2018}. It may be possible to analyze the $\kappa \to 0^+$ asymptotics in the presence of such ``switching", i.e., when the velocity at the boundary changes sign, particularly when the ``switch" occurs with non-zero speed $\dot V\big|_{x=0} \neq 0$ at the switching time. The stochastic interpretation of the PDE may also be useful in this scenario. 

\begin{figure}[!ht]
    \centering
    \includegraphics[scale=0.42]{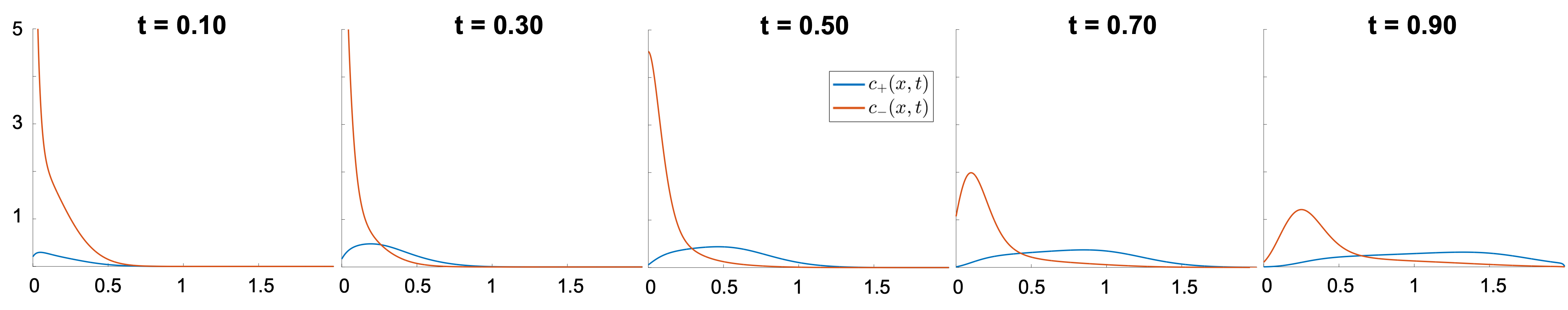}
    \caption{An example of velocity ``switching'' in the diffusive case ($\kappa=0.025$) with linear tumbling ($r_0=1$), leading to the boundary mass entering the bulk. Here the background field $V=2t$ is imposed, so that $V-\beta=2t-1$ switches from negative to positive at $t=0.5$. The width of the ``peak" leaving the boundary after $t=0.5$ is controlled by the value of $\kappa$. }
    \label{fig:bdrymassleaves}
\end{figure}

\emph{(ii) Hysteretic boundary layers:}
So far, we consider vanishing diffusivity only while the incoming ODE for $A$~\eqref{eq:Aequationintro} has unique solutions. However, our numerical simulations suggest that the incoming ODE has non-unique solutions for quite reasonable parameter values. We considered the model nonlinearity~\eqref{eq:example_r}, which was studied also in~\cite{ScheelStevensWavenumberSelection}, with parameter values $r_0 = r_1 = 1$, $\zeta = 0.05$, $\alpha=\gamma=1$:\footnote{The nonlinearity becomes stronger as $\zeta$ is decreased. It is worth mentioning that, when $r' \leq 0$, no bifurcation is possible -- see Corollary~\ref{cor:rprimeneg}.}
\begin{equation}
    \label{eq:Aequationspecialintro}
\begin{aligned}
    \p_X A - \p_X^2 A &= q e^{-X} \left(1 + \frac{A^2}{1+\zeta A^2} \right) \\
    (1 - \p_X) A\big|_{X=0} &= 0 \, .
\end{aligned}
\end{equation}
As $q \geq 0$ is varied, the solutions to~\eqref{eq:Aequationspecialintro} undergo two saddle-node bifurcations, resulting in an intermediate parameter regime for which two solutions are stable, and for which there is the possibility of phase transition and hysteresis (see Figure \ref{fig:ODEforA}). We suspect that the hysteresis leads to an inviscid problem for which the dynamic boundary condition has \emph{memory}. This is an interesting topic for further investigation.

\begin{figure}[!ht] 
    \centering
    \includegraphics[scale=0.42]{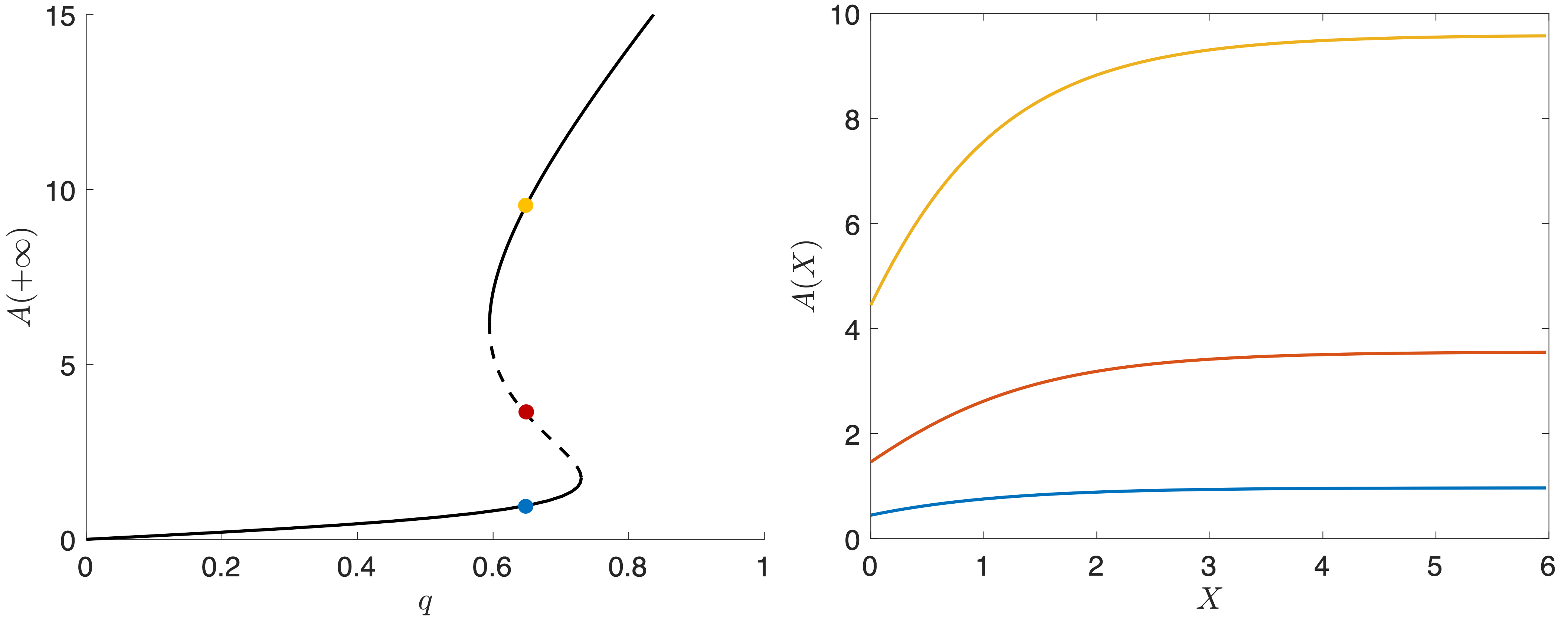}
    \caption{Bifurcation diagram displaying non-uniqueness of solutions to the ODE \eqref{eq:Aequationintro} for $A$ as the parameter $q$ is varied. Plotted is the right-end value $A(+\infty)$ as a function of $q$. Here $\zeta=0.05$, so the nonlinearity is quite strong. The right figure displays the three different solutions $A(X)$ when $q=0.65$. }
    \label{fig:ODEforA}
\end{figure}

The asymptotic structure of the boundary layer may change if the nonlinear tumbling is not saturated, e.g., nonlinearity may enter the ODEs at $O(1/\kappa)$.

\emph{(iii) Further generalizations:}
The energy methods used in Section~\ref{sec:vanishingdiff} to establish estimates on $\bm{c}^{\kappa} - \bm{c}^{\rm app}$ lead us to consider separate Cases~1 and~2. However, it is likely that convergence holds under more general assumptions along the lines of \emph{stability of the boundary layer}, even when $V\big|_{x=0}$ is non-zero and the tumbling is nonlinear. That energy methods are not sharp for this type of problem is well known from the literature on shocks and boundary layers in systems of viscous conservation laws, see, e.g.,~\cite{GoodmanXinViscous,YuZeroDissipation,GrenierRoussetBoundaryLayer,RoussetViscousShock,OlivierMetieverWilliamsZumbrunARMA}. There, energy methods are enough to prove the inviscid limit with small shocks~\cite{GoodmanXinViscous}, whereas the inviscid limit holds more generally under the assumption of spectral stability, which can be phrased in terms of the Evans function. Due to the particular structure of our problem, our energy methods are not specific to small solutions.

\emph{(iv) Active suspensions and higher dimensions:}
The 1D models~\eqref{eq:viscousequation} may be considered as simplified models for active suspensions in 2D and 3D. Our underlying motivation for studying the models~\eqref{eq:viscousequation} is to understand boundary layers in the more complicated Doi-Saintillan-Shelley (DSS) model~\cite{saintillan2008instabilities1,saintillan2008instabilities2,ezhilan2015distribution,SS_15,SS_notes}, which we intend to handle in future work. 
The DSS model describes the evolution of a number density $\psi(\bm{x},\bm{p},t)$ of rod-like swimmers with positions $\bm{x}\in \Omega\subset \R^d$, $d=2,3$, and a continuum of possible orientations $\bm{p}$ belonging to the unit sphere $S^{d-1}$. The swimmers are immersed in a Stokes fluid with velocity field $\bm{u}(\bm{x},t)$, and $\psi$ satisfies the Smoluchowski equation
\begin{equation}\label{eq:smoluchowski}
\p_t \psi + \div_x\big((\beta \bm{p}+\bm{u})\psi\big)+\div_p\big((\bI - \bm{p} \otimes \bm{p}) \nabla \bm{u} \bm{p}\,\psi\big)= \kappa\Delta_x\psi + \nu\Delta_p\psi \,.
\end{equation}
Here tumbling is replaced by orientational diffusion $\nu\Delta_p\psi$ with coefficient $\nu>0$, where $\Delta_p$ is the Laplacian on $S^{d-1}$. The nonlinear term involving $\div_p$, the divergence on $S^{d-1}$, describes the reorientation of elongated particles due to the fluid \cite{jeffery1922motion}. The swimmers interact hydrodynamically by exerting an active, alignment-dependent stress $\bm{\Sigma}$ on the surrounding Stokes fluid:
\begin{equation}\label{eq:activestress}
\begin{aligned}
- \Delta \bm{u} + \nabla q &= \div \bm{\Sigma}, \quad \div\bm{u} =0 \, , \\
\bm{\Sigma}(\bm{x},t)&=\pm \int_{S^{d-1}} \psi(\bm{x},\bm{p},t)\, \bm{p}\otimes\bm{p} \, d\bm{p} \, .  
\end{aligned}
\end{equation}
The no-flux and no-slip boundary conditions
\begin{equation}
    \label{eq:noslipcond}
    \big(\beta \bm{p}\psi -\kappa \nabla_x\psi\big)\cdot\bm{n}\big|_{\p\Omega} =0\,, \quad \bm{u}\big|_{\p\Omega}=0
\end{equation}
are commonly used, and an important question is to determine the effective behavior of this system in the $\kappa\to 0$ limit.

In~\cite[{\sc Conjecture}, p. 5]{berlyand2020kinetic}, Beryland \emph{et al.} conjecture that in certain Fokker-Planck equations, which may be regarded as linear versions of~\eqref{eq:smoluchowski}, the distribution $\psi$ asymptotically decomposes into $\psi_{\rm bulk}$ and $\psi_{\rm wall} \delta_{\p \Omega}$ solving a coupled system of PDEs. 
In~\cite{FuPerthameTang2021,fu2023confined}, a related limiting Fokker-Planck model is proposed and analyzed. In~\cite{berlyand2026multiscale}, such a decomposition is justified rigorously under an assumption that the motion of swimmers is only into the wall; in \eqref{eq:smoluchowski}, this would correspond to $(\bm{u}+\beta\bm{p}) \cdot \bm{n}\big|_{\p \Omega} \geq c > 0$. The rigorous asymptotics remain open under the full range of swimmer-wall interactions (incoming, outgoing, and grazing).

A further difficulty is to incorporate the nonlinear and nonlocal effects of hydrodynamics on the zero diffusivity limit. Our work is a starting point for rigorous convergence results incorporating these effects. The most direct 1D analogue of hydrodynamics~\eqref{eq:activestress} is a self-generated field $V$ satisfying the boundary value problem  
\begin{equation}\label{eq:Vhydro}
    \left(-\p_x^2 + \frac{1}{\ell^2}\right) V = \pm \p_x \rho \, , \quad V\big|_{x=0} = 0 \, ,
\end{equation}
    where, again, $\rho=c_++c_-$, and $\ell > 0$ is a screening length scale. Here $V$ may be represented by a \emph{discontinuous} kernel
    \begin{equation}
    \begin{aligned}
        V = \pm \int G_\ell(x,y)\rho(y)\,dy\,,\quad 
        G_\ell(x,y) = \frac{\ell}{2} \p_x \left( e^{-\frac{|x-y|}{\ell}} + e^{-\frac{|x+y|}{\ell}} \right) \, .
    \end{aligned}
    \end{equation}
    The kernel $G_\ell$ is not technically covered by our convergence theory, which requires a Lipschitz kernel (see Section~\ref{sec:vanishingdiff}), although we expect that the convergence holds. While $V$ vanishes on the boundary for smooth $\rho$, it does not vanish for $\rho$ having a Dirac mass on the boundary, so the boundary layer induces a self-interaction.
    
    In forthcoming work, we investigate the bifurcation structure of steady states for the 1D model with $V$ given by \eqref{eq:Vhydro} on both the torus $\T = \R/\Z$ and interval $[0,1]$ and analyze the effects of the boundary.


\section{Construction of the approximate solutions}\label{sec:approxsols}

Here we construct the approximate solution $\bm{c}^{\rm app}$ used to obtain the bounds \eqref{eq:cappdiff0}-\eqref{eq:cappdiff2}. We begin by recording some important properties of the incoming ODE \eqref{eq:Aequationintro} and the inviscid system \eqref{eq:inviscidequation}-\eqref{eq:Vinviscid} in both Cases 1 and 2. The proofs of these properties are given in Sections \ref{sec:ODE} and \ref{sec:outerproblem}. We then proceed to construct the approximate solution via matched asymptotics and derive residual bounds for the errors $\bm{c}^\kappa-\bm{c}^{\rm app}$.

\subsection{The incoming ODE}\label{sec:Aeqnearly}

Recall the incoming ODE, which we may write for a general function $A$ on the half-line $\{ X > 0\}$ as
\begin{equation}
    \label{eq:Aequationforproposition}
\begin{aligned}
    \gamma \p_X A - \p_X^2 A &= q e^{-\nu X} r(A) \\
    (\gamma - \p_X) A\big|_{X=0} &= 0 \, ,
\end{aligned}
\end{equation}
 where $\gamma, \nu > 0$ and $q \geq 0$. This equation arises in determining $C^{(1)}_+$ in~\eqref{eq:C1pluseqn}. We consider solutions which remain bounded as $X \to +\infty$.

In Case~1, when $r(A) = r_0$, the unique such solution is explicit:
\begin{equation} \label{eq:Aincase1}
    A(q,\gamma,\nu)(X) = \frac{q r_0}{\nu} \left( \frac{1}{\gamma} - \frac{1}{\gamma+\nu} e^{-\nu X} \right) \, .
\end{equation}

In Case~2, the following proposition is sufficient:
\begin{proposition}
    \label{pro:ODEsummaryforstuff}
    Fix $\gamma,\nu > 0$. There exists $q_*(\gamma,\nu) > 0$ and a curve $[0,q_*) \to C^2([0,+\infty)) : q \mapsto A(q,\gamma,\nu)(\cdot)$ of solutions to the ODE problem~\eqref{eq:Aequationforproposition}. The solutions each decay to a well-defined end state
    \begin{equation} 
        c_\infty(q,\gamma,\nu) := \lim_{X \to +\infty} A(q,\gamma,\nu)(X) \, ,
    \end{equation}
    which depends smoothly on $q \in [0,q_*)$. Moreover, for all $k \in \N_0$, the function $q \mapsto \p_X^k(A - c_\infty)$ depends smoothly on $q \in [0,q_*)$ in the weighted norm $\| e^{\gamma X} \cdot \|_{L^\infty(\R_+)}$.
\end{proposition}

See Section~\ref{sec:ODE} and Proposition~\ref{pro:Abar} for a detailed treatment of this problem.

\subsection{The inviscid problem}
\label{sec:invicsidproblemearly}
For $p \in [1,+\infty)$, $k \in \N$, and $T > 0$, we introduce the function spaces
\begin{align}
\label{eq:Xkdef}
&  X^{k, p} (T) = \big\{ f \in C([0,T];L^p(\R_{+})) : D^j_{t,x} f \in C([0,T];L^p(\R_{+})) \;\; \forall j \leq k \big\}
\end{align}
with norms
\begin{align}
    \label{eq:Xknormdef}
    \|f\|_{X^{k, p} (T)} := \sum_{j=0}^k \|D_{t,x}^j f\|_{ L^{\infty}_t L^p_x([0, T] \times \R_{+})} \, ,
\end{align}
where $D_{t, x}^j f$ denotes the matrix of all derivatives in $t$ and $x$ of $f(t,x)$ of order $j$.

\begin{proposition}
\label{pro:inviscidsolutionexists}
Suppose that $K$ satisfies~\eqref{eq:smoothnessofK}-\eqref{eq:decayonkernel} and $f$ satisfies~\eqref{eq:fintermsofr}-\eqref{eq:tumblingassumptiononr}. Suppose that Case~1 ($r \equiv r_0 > 0$ const.) or Case~2 ($K(0,y) = 0$) holds.

Let $\bm{c}^{\rm in} \in W^{2,1}(\R_+)$, $c_{\pm}^{\rm in} \geq 0$, $b_-^{\rm in} \geq 0$, and $V^{\rm in} := V[\rho^{\rm in},b_-^{\rm in}]$ (see \eqref{eq:Vinviscid}) satisfying the compatibility conditions
\begin{equation}
    \label{eq:compat1}
    c_+^{\rm in}(0) =
c_\infty\big((\beta_- - V^{\rm in}(0)) b^{\rm in}_-,V^{\rm in}(0)+\beta_+,\beta_- -V^{\rm in}(0)\big) =: g(b^{\rm in}_-,V^{\rm in}(0))
\end{equation}
\begin{equation}
    \label{eq:compat2}
    (\p_t c_+)(0,0) = (\p_{b_-} g_b)(b^{\rm in}_-,V^{\rm in}(0)) \dot b_-(0) + (\p_{V} g_b)(b^{\rm in}_-,V^{\rm in}(0)) (\p_t V)(0,0)
\end{equation}
where $(\p_t c_+)(0,0)$ and $\p_t V(0,0)$ are interpreted in the sense of equation~\eqref{eq:inviscidequation} and $\dot b_-(0)$ is interpreted in the sense of equation~\eqref{eq:bminuseqn} (Case 1) or \eqref{eq:bminuseqncase2} (Case 2). In Case~2, further assume that $\beta_- b_-^{\rm in} \in [0,q^*)$, so that~\eqref{eq:compat1} makes sense. Finally, suppose
\begin{equation}
    \label{eq:compat3}
    V^{\rm in}\big|_{x=0} + \beta_+ > 0 \, , \quad  V^{\rm in}\big|_{x=0} - \beta_- < 0 \, .
\end{equation}

Then there exists a maximal existence time $T^* \in (0,+\infty]$ and unique non-negative solution 
\begin{equation}
    c_{\pm} \in X^{2,1}(T) \, , \; b_- \in W^{2,1}([0,T])  \quad \forall T \in (0,T^*)
\end{equation}
to the inviscid system \eqref{eq:inviscidequation}-\eqref{eq:cinv_plus_bvalcase2} satisfying (relevant to Case~1)
\begin{equation}
    \min_{t \in [0,T]} V\big|_{x=0} + \beta_+ > 0 \, , \quad \max_{t \in [0,T]} V\big|_{x=0} - \beta_- < 0  \quad \forall T \in (0,T^*) \, .
\end{equation}

In Case~1, if $T^* < +\infty$, then
\begin{equation}
    \limsup_{t \to T^*_-} \frac{1}{V\big|_{x=0}(t) + \beta_+} + \frac{1}{\big |V\big|_{x=0}(t) - \beta_- \big|} = + \infty \, .
\end{equation}

In Case~2, if $T^* < +\infty$, then
\begin{equation}
    \limsup_{t \to T^*_-} q(t) = q^*(\beta_+,\beta_-)
\end{equation}
where $q(t) := \beta_- b_-$ and $q^*$ is as in Proposition~\ref{pro:ODEsummaryforstuff}.
\end{proposition}
See Section~\ref{sec:outerproblem} for a detailed treatment of this problem.

\subsection{Formal asymptotics}\label{subsec:formalasymp}
Equipped with Propositions \ref{pro:ODEsummaryforstuff} and \ref{pro:inviscidsolutionexists}, we proceed to construct the approximate solution $\bm{c}^{\rm app}$ used in \eqref{eq:cappdiff0}-\eqref{eq:cappdiff2}.

\subsubsection{Outer solution}
Away from the boundary, where the diffusivity can be considered negligible, we begin with the formal procedure of expanding the solution to the viscous equation \eqref{eq:viscousequation} as
\begin{equation}
    \bm{c} = \bm{c}^{(0)} + \kappa \bm{c}^{(1)} + \kappa^2 \bm{c}^{(2)} + \cdots \,,
\end{equation}
where $\bm{c}^{(0)}$ satisfies an inviscid problem
\begin{equation}
\begin{aligned}
    \p_t c_+ + \p_x ((V+\beta_+) c_+) &=  f(c_+,c_-) \\
    \p_t c_- + \p_x ((V-\beta_-) c_-) &= - f(c_+,c_-) \, .
\end{aligned}
\end{equation}
(One may also wish to expand $V$.) This equation is incomplete without a boundary condition on $c_+^{(0)}$ at $x=0$. With some foresight, we anticipate that $\bm{c}^{(0)}$ should be $c^{\rm inv}_{\pm}(t,x)$, the inviscid solution satisfying \eqref{eq:inviscidequation} on $\R_+$, from Proposition~\ref{pro:inviscidsolutionexists}. Throughout, we will also denote the $V$ generated by the inviscid solution as 
\begin{equation}\label{eq:Vinv}
  V^{\rm inv}(t,x) := \int_{\R_+}K(x,y)\big(c^{\rm inv}_+(t,y)+c^{\rm inv}_-(t,y)\big)\,dy + b_-(t)K(x,0)\,,
\end{equation}
as in \eqref{eq:Vinviscid}.

\subsubsection{Inner solution}\label{subsec:innersol}
Within the boundary layer near $x=0$, we consider the system \eqref{eq:viscousequation} under the rescaling $X=\frac{x}{\kappa}$. To describe the inner solution at these scales, we make the change of variables
\begin{equation}
  C_{\pm}(t,X) = c_{\pm}\left(t,  x \right)\,,
\end{equation}
so that $C_{\pm}$ should satisfy
\begin{equation}\label{eq:innereq}
  \p_tC_\pm +\frac{1}{\kappa}\p_X\big((V(t,\kappa X)\pm \beta_\pm)C_\pm\big) = \frac{1}{\kappa}\p_X^2C_\pm \pm f(C_+,C_-)
\end{equation}
on the spatial domain $\R_+$, along with the no-flux boundary conditions
\begin{equation}\label{eq:innerBC}
  (V \pm \beta_\pm)C_{\pm} - \p_XC_{\pm} =0 \qquad \text{at }X=0\, .
\end{equation}

We formally expand $C_{\pm}(t,X)$ in powers of $\kappa$ as
\begin{equation}\label{eq:Cexpand}
  C_{\pm}(t,X) = \frac{1}{\kappa} C^{(0)}_{\pm}(t,X) + C^{(1)}_{\pm}(t,X)+ \kappa C^{(2)}_{\pm}(t,X)+\cdots\,.
\end{equation}
In addition, we will approximate $V(t,\kappa X)$ in the boundary layer by $V^{\rm inv}(t,\kappa X)$, where $V^{\rm inv}$ is as in \eqref{eq:Vinv}, and which we further Taylor expand about $X=0$ as\footnote{Recall that the full $V$ is induced non-locally by the solution in both the inner \emph{and} the outer regions.} 
\begin{equation}
  V^{\rm inv}(t,\kappa X) = \underbrace{V^{(0)}(t)}_{V^{\rm inv}(t,0)} + \kappa X\underbrace{V^{(1)}(t)}_{\p_XV^{\rm inv}(t,0)} + \kappa^2X^2\underbrace{V^{(2)}(t)}_{\frac{1}{2}\p_X^2V^{\rm inv}(t,0)}+\cdots \,.
\end{equation}
Here we recall the requirements \eqref{eq:V0condition} that $V^{(0)}(t)-\beta_-<0$ and $V^{(0)}(t)+\beta_+ > 0$ on the time interval of interest. 
Using the above expansions in \eqref{eq:innereq} and matching orders in $\kappa$, we obtain the equation for $C_{\pm}^{(0)}$:
\begin{equation}\label{eq:C0eq}
    (V^{(0)}\pm\beta_\pm)\p_XC^{(0)}_{\pm} - \p_X^2 C^{(0)}_{\pm} = 0.
\end{equation}
Using the no-flux boundary condition at $X=0$, we may solve for $C^{(0)}_{\pm}$ as
\begin{align}\label{eq:C0sol}
    C_{\pm}^{(0)}(t,X) = C_{\pm}^{(0)}(t,0)\, e^{(V^{(0)}(t)\pm\beta_\pm) X}\,.
\end{align}
Since $V^{(0)}-\beta_-<0$, $C_-^{(0)}(t,X)$ decays as $X\to\infty$. However, in order to have the possibility for $C_+^{(0)}(t,X)$ to match with the leading outer solution $c^{\rm inv}_{+}$ as $X\to\infty$, we must have 
\begin{align}
 C_+^{(0)}(t,0)\equiv 0\,.
\end{align}
For convenience, we write $C_-^{(0)}(t,0) =: q(t)$, so
\begin{equation}
    \label{eq:Cminusexpressionwithq}
    C_-^{(0)}(t,X) = q(t) e^{(V^{(0)} (t)- \beta_-) X} \, .
\end{equation}
The leading order mass $B_-(t)$ in the boundary layer is obtained by integrating $C_-^{(0)}$ on $\R_+$:
\begin{equation}
    \label{eq:Bminusexpression}
    B_-(t) := \frac{q(t)}{\beta_- - V^{(0)}(t)} \, .
\end{equation}

At next order in $\kappa$, we have that $C^{(1)}_{\pm}$ should satisfy
\begin{equation}\label{eq:C1eq}
\begin{aligned}
  &(V^{(0)}\pm\beta_\pm)\p_XC^{(1)}_{\pm} - \p_X^2 C^{(1)}_{\pm} \\
  &\qquad = -\p_tC^{(0)}_\pm -\p_X(XV^{(1)}C^{(0)}_{\pm}) \pm \kappa [f(C_+,C_-)]_{\rm main} \,,
\end{aligned}
\end{equation}
with no-flux boundary conditions at $X=0$. The brackets $[\cdot]_{\rm main}$ are used to indicate the leading order (in $\kappa$) of the tumbling term $f$, which we now extract. Here we recall that $f$ is of the form  
\begin{equation}
  f(C_+,C_-) = - C_+r(C_-) + C_-r(C_+)\,.
\end{equation}
Using that $C_\pm \approx \frac{1}{\kappa} C^{(0)}_{\pm} + C^{(1)}_{\pm}$ and $C^{(0)}_+ \equiv 0$, we obtain
\begin{equation}
    \kappa f(C_+,C_-) \approx - \kappa C_+^{(1)} r(C_-) + C_-^{(0)} r(C_+^{(1)}) \, ,
\end{equation}
and therefore, by the smoothness and saturation assumptions on $r$, we have
\begin{equation}
    \kappa [f(C_+,C_-)]_{\rm main} := C_-^{(0)} r(C_+^{(1)}) \, .
\end{equation}
Thus,~\eqref{eq:C1eq} and the expression~\eqref{eq:Cminusexpressionwithq} for $C_-^{(0)}$ yield the following equations for $C^{(1)}_{\pm}$:
\begin{align}
 (V^{(0)}+\beta_+)\p_XC^{(1)}_+ - \p_X^2 C^{(1)}_+ 
  &= q(t) e^{(V^{(0)}-\beta_-) X} r(C^{(1)}_+)
  \label{eq:C1pluseqn}\\
  (V^{(0)}-\beta_-)\p_XC^{(1)}_- - \p_X^2 C^{(1)}_-
  &=  -\bigg[\p_t q(t) + q(t) \big(X\p_tV^{(0)}+ r(C^{(1)}_+) \big)\bigg] e^{(V^{(0)}-\beta_-) X} \nonumber \\
  &\qquad  - q(t) \p_X\big(XV^{(1)} e^{(V^{(0)}-\beta_-) X}\big) \,. 
  \label{eq:C1minuseqn}
\end{align}

The equation~\eqref{eq:C1pluseqn} and the no-flux condition together constitute a nonlinear ODE problem for $C^{(1)}_+$ of the form
\begin{equation}
\begin{aligned}
    \gamma \p_X A - \p_x^2 A &= q e^{-\nu X} r(A) \\
    (\gamma - \p_X) A\big|_{x=0} &= 0 \, ,
\end{aligned}
\end{equation}
where
\begin{equation}\label{eq:qgammanu}
    q = q(t) \geq 0 \, , \quad \gamma = V^{(0)}(t) + \beta_+ > 0 \, , \quad \nu = \beta_- - V^{(0)}(t) > 0 \, .
\end{equation}
The solvability of this ODE is discussed in Section~\ref{sec:Aeqnearly}, see Proposition~\ref{pro:ODEsummaryforstuff}. We have
\begin{equation}
    C_+^{(1)}(t,X) = A(q,V^{(0)}(t)+\beta_+,\beta_- - V^{(0)}(t))(X) \, .
\end{equation}

The matching condition on $C^{(1)}_+$ is therefore
\begin{equation}
    c^{\rm inv}_+(t,0) = c_\infty\left(q(t),V^{(0)}(t)+\beta_+,\beta_- -V^{(0)}(t)\right) \, ,
\end{equation}
which can equivalently be expressed in terms of the leading order boundary mass $B_-(t)$ defined in~\eqref{eq:Bminusexpression}.

We now analyze the equation~\eqref{eq:C1minuseqn} for $C^{(1)}_-$. Integrating once in $X$ and using the no-flux boundary condition, we obtain 
\begin{equation}
    \label{eq:initialequationforC1minus}
\begin{aligned}
  \p_X C^{(1)}_- -(V^{(0)}-\beta_-)C^{(1)}_- &= G_a(t)(e^{(V^{(0)}-\beta_-) X} -1) + XG_b(t)e^{(V^{(0)}-\beta_-) X}  \\
  &\quad  - q(t) \int_0^X r(C^{(1)}_+) e^{(V^{(0)} - \beta_-)Y} \, dY \,,
\end{aligned}
\end{equation}
where
\begin{equation}
    \begin{aligned}
  G_a(t)&= \frac{\p_t q(t)}{V^{(0)}-\beta_-} - \frac{\p_tV^{(0)} q(t)}{(V^{(0)}-\beta_-)^2}\\
  G_b(t)&= q(t) \bigg(\frac{\p_tV^{(0)}}{V^{(0)}-\beta_-} +V^{(1)} \bigg)\,.
\end{aligned}
\end{equation}
For purposes of matching, the relevant solution to this ODE should converge to a constant and $\p_X  C^{(1)}_- \to 0$ as $X \to +\infty$.
This will be borne out more explicitly in~\eqref{eq:C1minusexplicitformula}, but we may already utilize the matching condition
\begin{equation}
    \lim_{X \to +\infty} C^{(1)}_-(t,X) = c^{\rm inv}_{-}\big|_{x=0}
\end{equation}
by sending $X \to +\infty$ in~\eqref{eq:initialequationforC1minus}. In particular, this yields an ODE for $q(t)$: 
\begin{equation}\label{eq:qode}
\begin{aligned}
  &\underbrace{\frac{1}{(V^{(0)}-\beta_-)} \p_t q(t) - \frac{\p_tV^{(0)}}{(V^{(0)}-\beta_-)^2} q(t)}_{= - \p_t B_-} - \underbrace{q(t) \int_0^{+\infty} r(C^{(1)}_+) e^{(V^{(0)} - \beta_-)Y}  \, dY}_{= (V^{(0)}+\beta_+) c_\infty(q,V^{(0)}+\beta_+,\beta_--V^{(0)})}\\
  &\qquad = (V^{(0)}-\beta_-) c^{\rm inv}_-(t,0)\,.
\end{aligned}
\end{equation}
where we used the identity \eqref{eq:cformula} to simplify the integral term.
Equivalently, after we use~\eqref{eq:Bminusexpression} to recast $q$ in terms of $B_-$,
\begin{equation}
\begin{aligned}
    \p_t B_- &= - (V^{(0)} + \beta_+) c_\infty\left((\beta_- - V^{(0)}) B_-,V^{(0)}+\beta_+,\beta_--V^{(0)}\right)\\
    &\qquad + (\beta_- - V^{(0)}) c^{\rm inv}_-(t,0) \, ,
\end{aligned}
\end{equation}
or, with arguments suppressed, 
\begin{equation} \label{eq:boundarymassode}
    \p_t B_- = - (V^{(0)} + \beta_+) c_\infty + (\beta_- - V^{(0)}) c^{\rm inv}_-\big|_{x=0} \, .
\end{equation}
This says that the boundary gains mass due to incoming left-swimmers and loses mass due to tumbling and right-moving swimming.

We now specialize to Case~1 and Case~2. First, when $r \equiv r_0$ is constant, then $c_\infty(q,\gamma,\nu) = \frac{r_0 q}{\nu \gamma}$, and then, by \eqref{eq:qgammanu},
\begin{equation}
    \p_t B_- = - r_0 B_- + (\beta_- - V^{(0)}) c^{\rm inv}_-\big|_{x=0} \, .
\end{equation}
Second, when $V\big|_{x=0} = 0$, we have
\begin{equation}
    \p_t B_- = - \beta_+ c_\infty(\beta_- B_-, \beta_+,\beta_-) + \beta_- c^{\rm inv}_-\big|_{x=0} \, .
\end{equation}
 Thus, by \eqref{eq:bminuseqn}-\eqref{eq:cinv_plus_bvalcase2}, $B_- = b_{-}$.

The expression for $C^{(1)}_-$ is obtained by writing Duhamel's formula from the ODE~\eqref{eq:initialequationforC1minus} with free parameter $C^{(1)}_-(t,0)$; we choose $C^{(1)}_-(t,0) = 0$ for convenience.\footnote{If one were expanding to higher order, then this constant would be chosen to obtain a next-order correction to the boundary flux.}  The resulting expression~is 
\begin{equation}
    \label{eq:C1minusexplicitformula}
\begin{aligned}
  C^{(1)}_-(t,X)
  &= G_a(t)\int_0^Xe^{(V^{(0)}-\beta_-)(X-Y)}(e^{(V^{(0)}-\beta_-) Y}-1)\,dY\\
  &\qquad +G_b(t)\int_0^Xe^{(V^{(0)}-\beta_-)(X-Y)}\,Ye^{(V^{(0)}-\beta_-) Y}\,dY \\
  &\qquad - q(t) \int_0^X e^{(V^{(0)} - \beta_-)(X-Y)} \int_0^Y r(C^{(1)}_+) e^{(V^{(0)} - \beta_-)Z} \, dZ \, dY \\
  &= \text{ $G_a$ term } + \text{ $G_b$ term } - \text{ $q(t)$ term} \, .
\end{aligned}
\end{equation}
These terms may be computed more explicitly as (cf. \eqref{eq:qode})
\begin{equation}
\begin{aligned}
    \text{ $G_a$ term } &= G_a(t)\bigg( Xe^{(V^{(0)}-\beta_-)X} +\frac{1}{V^{(0)}-\beta_-}(1-e^{(V^{(0)}-\beta_-) X})\bigg)\\
  \text{ $G_b$ term } &= \frac{X^2}{2}G_b(t)e^{(V^{(0)}-\beta_-)X} \\
    \text{ $q(t)$ term } &= q(t) \int_0^X e^{(V^{(0)} - \beta_-)(X-Y)} \int_0^{+\infty} r(C^{(1)}_+) e^{(V^{(0)} - \beta_-)Z} \, dZ \, dY - R_q \\
    &= \frac{V^{(0)} + \beta_+}{\beta_- - V^{(0)}} c_\infty(q,V^{(0)} + \beta_+,\beta_- - V^{(0)}) (1 - e^{(V^{(0)} - \beta_-) X}) - R_q\\
    R_q &= q(t) \int_0^X e^{(V^{(0)} - \beta_-)(X-Y)} \int_Y^{+\infty} r(C^{(1)}_+) e^{(V^{(0)} - \beta_-)Z} \, dZ \, dY \\
    &= O(e^{(V^{(0)} - \beta_-)X}) \, .
\end{aligned}
\end{equation}
Furthermore, by \eqref{eq:C1minusexplicitformula}, the fact that $G_a (t) = -\partial_t b_{-}$ (cf. \eqref{eq:qode}), and \eqref{eq:boundarymassode},
\begin{equation}\label{eq:c1minusasy}
\begin{aligned} 
    C^{(1)}_-(t,X) = c^{\rm inv}_-\big|_{x=0} +  O(e^{(V^{(0)} - \beta_-)X/2}).
\end{aligned}
\end{equation}

In summary, our leading order approximation for the behavior of the particles within the boundary layer is given by the inner solution 
\begin{equation}
\label{eq:ckappa_pm}
\boxed{
\begin{aligned}
    c_-^{\rm bl}(x,t) &= \frac{1}{\kappa} C^{(0)}_-\left(t,\frac{x}{\kappa}\right) \,+ && \hspace{-.2cm} C^{(1)}_-\left(t,\frac{x}{\kappa}\right) \\
    c_+^{\rm bl}(x,t) &= &&\hspace{-.2cm} C^{(1)}_+\left(t,\frac{x}{\kappa}\right)
\end{aligned}
}
\end{equation}
which is well defined under the assumptions of Theorem~\ref{thm:mainthm}.

\subsection{The boundary layer error}\label{subsec:errBL}

We next consider the quantitative estimates satisfied by our boundary layer approximation. For the remainder of Section~\ref{sec:approxsols}, we suppose the assumptions of Theorem~\ref{thm:mainthm}. In particular, there exists $T>0$ and a solution $\bm{c}^{\rm inv} \in X^{2,1}(T)$ to the inviscid system satisfying  
\begin{equation}
    V^{(0)}(t,0) + \beta_+ \geq \delta > 0 \, , \; V^{(0)}(t,0) - \beta_- \leq -\delta < 0 \, , \quad \forall t \in [0,T] \, .
\end{equation}
On the boundary, we have\footnote{Strictly speaking, our error estimates do not require the $C^2$ time regularity, only $C^1$. For the $V^{(0)},V^{(1)}$ time regularity, see~\eqref{eq:twotimederivativeest}. The $b_-$ time regularity follows by inserting $b_{-}, c_{-} \in C^1$ back into the ODE for $b_-$.}
\begin{equation}
c^{\rm inv}_{\pm}\big|_{x=0} \in C^1([0,T])  \,;\quad  V^{(0)}\,,\; V^{(1)}\,,\; b_- \in C^2([0,T]) \,.
\end{equation}
Subsequently, for $q$ given by \eqref{eq:Cminusexpressionwithq}, we have $q \in C^2 ([0,T])$, from which we deduce that 
\begin{align}
C^{(1)}_+ \,,\; e^{\delta X/2} \p_X^k C^{(1)}_+ &\in C^2([0,T];{\rm BC}(\R_+)) \text{ for all } k \in \N\, , \label{eq:C1plus_derivs}\\
C^{(1)}_- \,,\; e^{\delta X/2} \p_X^k C^{(1)}_- &\in C^1([0,T];{\rm BC}(\R_+)) \text{ for all } k \in \N\,, \label{eq:C1minus_derivs}
\end{align}
which is consistent with the time regularity of $c^{\rm inv}_-\big|_{x=0}$. Here, we are invoking Proposition~\ref{pro:ODEsummaryforstuff} for $C^{(1)}_+$ and the expression~\eqref{eq:C1minusexplicitformula} for $C^{(1)}_-$. Additionally, 
\begin{equation}
    \label{eq:C1expdecaytoendstate}
    \abs{C^{(1)}_{\pm}(X) - c^{\rm inv}_{\pm}\big|_{x=0}} \lesssim e^{-\delta X/2} \, .
\end{equation}
For $C^{(1)}_{-}(X)$, this is justified in \eqref{eq:c1minusasy}, while for $C^{(1)}_{+}(X)$, this follows from the boundary condition \eqref{eq:cplusunifiedbc} and Proposition \ref{pro:ODEsummaryforstuff}.


For the remainder of Section~\ref{sec:approxsols}, we fix $\alpha \in [1/2,1)$ and a length scale
\begin{equation}
    \label{eq:lengthscaledef}
    \ell_\kappa := 2 \kappa^{\alpha} \, .
\end{equation}
It will also be convenient to impose the upper bound $\kappa \leq 1$.

We now estimate the error in the equations satisfied by $c^{\rm bl}_\pm$ within the boundary layer. \emph{For the remainder of the section, the implied constants may depend on the above quantities ($\| \bm{c}^{\rm inv} \|_{X^{2,1}(T)}$, etc.).}

\begin{lemma}[Boundary layer residual]\label{lem:BL_residual}
For $V^{\rm inv}$ as in \eqref{eq:Vinv}, the inner solution $c^{\rm bl}_\pm(t,x)$ given by \eqref{eq:ckappa_pm} satisfies
\begin{equation}
  \p_tc^{\rm bl}_\pm +\p_x\big((V^{\rm inv}\pm \beta_\pm)c^{\rm bl}_\pm\big) -\kappa\p_x^2c^{\rm bl}_\pm \mp f(c^{\rm bl}_+,c^{\rm bl}_-) = \mc{R}_\pm \,.
\end{equation}
Near the boundary, we have the bound 
\begin{equation}
\begin{aligned}
  \norm{\mc{R}_\pm}_{L^\infty_tL^1_x([0,T]\times[0,\ell_\kappa])} &\lesssim \kappa^\alpha + \kappa^{2\alpha-1}\, .
\end{aligned}
\end{equation}
\end{lemma}

\begin{proof}
We begin with $c^{\rm bl}_-(t,x) = \frac{1}{\kappa} C^{(0)}_-(t,X) + C^{(1)}_-(t,X)$, which satisfies the equation\footnote{Recall $X = x/\kappa$, so $\p_x = \kappa^{-1} \p_X$} (see \eqref{eq:C0eq} and \eqref{eq:C1eq})
\begin{equation}
\begin{aligned}
  &\p_tc^{\rm bl}_- + \p_x\big((V^{\rm inv}-\beta_-)c^{\rm bl}_-\big) - \kappa\p_x^2c^{\rm bl}_- + f(c^{\rm bl}_+,c^{\rm bl}_-) \\
  &\quad = \underbrace{\frac{1}{\kappa}(V^{\rm inv}\big|_{x=0}-\beta_-)\p_xC^{(0)}_- - \p_x^2C^{(0)}_-}_{=0}   
   \,+\,\mc{R}_{a-} +\mc{R}_{b-} +\mc{R}_{c-} \\
  & +  \underbrace{(V^{\rm inv}\big|_{x=0}-\beta_-)\p_xC^{(1)}_-  - \kappa \p_x^2 C^{(1)}_- 
  +\frac{1}{\kappa}\left(\p_tC^{(0)}_- 
  +\p_x\big(x\p_xV^{\rm inv}\big|_{x=0}\,C^{(0)}_-\big) + C^{(0)}_-\,r(C^{(1)}_+)\right) }_{=0}\,,  
\end{aligned}
\end{equation}
where the remainder terms are given by
\begin{equation}
\begin{aligned}
  \mc{R}_{a-}&= \p_tC^{(1)}_- + \p_x\big((V^{\rm inv}(t,x)-V^{\rm inv}(t,0)) C^{(1)}_-\big)\\
  \mc{R}_{b-}&= \frac{1}{\kappa}\p_x\big((V^{\rm inv}(t,x)-V^{\rm inv}(t,0) - x\p_xV^{\rm inv}(t,0))C^{(0)}_-\big) \\
  \mc{R}_{c-}&= -C^{(1)}_+ \,r\left(\frac{1}{\kappa}C^{(0)}_- + C^{(1)}_-\right) + C^{(1)}_-\,r(C^{(1)}_+) \,.
\end{aligned}
\end{equation}
For $(t,x)\in [0,T]\times[0,\ell_\kappa]$, using \eqref{eq:C1minus_derivs} and \eqref{eq:Vinviscid}, we may estimate 
\begin{equation}\label{eq:Raminusest}
\begin{aligned}
  \norm{\mc{R}_{a-}}_{L^\infty_tL^1_x}&\le \norm{\p_tC^{(1)}_-}_{L^\infty_tL^1_x} + \norm{\p_xV^{\rm inv}}_{L^\infty_tL^\infty_x}\norm{C^{(1)}_-}_{L^\infty_tL^1_x}\\
  &\quad + \norm{V^{\rm inv}(t,x)-V^{\rm inv}(t,0)}_{L^\infty_tL^\infty_x}\norm{\p_xC^{(1)}_-}_{L^\infty_tL^1_x}\\
  &\lesssim\kappa^\alpha + \norm{\p_xV^{\rm inv}}_{L^\infty_tL^\infty_x}\kappa^\alpha
  + \kappa^\alpha\norm{\p_xV^{\rm inv}}_{L^\infty_tL^\infty_x}\norm{\p_xC^{(1)}_-}_{L^\infty_tL^1_x}\\
  &\lesssim \kappa^\alpha\,.
\end{aligned}
\end{equation}

Furthermore, using~\eqref{eq:Cminusexpressionwithq}-\eqref{eq:Bminusexpression}, we have
\begin{equation}
\begin{aligned}
  \norm{\mc{R}_{b-}}_{L^\infty_tL^1_x}&\le \norm{\p_xV^{\rm inv}(t,x)-\p_xV^{\rm inv}(t,0)}_{L^\infty_tL^\infty_x}\norm{\frac{1}{\kappa}C^{(0)}_-}_{L^\infty_tL^1_x} \\
  & +\norm{V^{\rm inv}(t,x)-V^{\rm inv}(t,0) - x\p_xV^{\rm inv}(t,0)}_{L^\infty_tL^\infty_x}\norm{\frac{1}{\kappa}\p_xC^{(0)}_-}_{L^\infty_tL^1_x}\\
  &\lesssim \kappa^\alpha\norm{\p_x^2V^{\rm inv}}_{L^\infty_tL^\infty_x} + \kappa^{2\alpha}\norm{\p_x^2V^{\rm inv}}_{L^\infty_tL^\infty_x}\kappa^{-1}\\
  &\lesssim \kappa^\alpha + \kappa^{2\alpha-1}\,.
\end{aligned}
\end{equation}

Finally, using~\eqref{eq:tumblingassumptiononr}, we may bound
\begin{equation}\label{eq:Rcminus}
\begin{aligned}
  \norm{\mc{R}_{c-}}_{L^\infty_tL^1_x}&\lesssim \norm{C^{(1)}_+}_{L^\infty_tL^1_x} \norm{r(\kappa^{-1}C^{(0)}_- + C^{(1)}_-)}_{L^\infty_tL^\infty_x} \\
  &\quad + \norm{C^{(1)}_-}_{L^\infty_tL^1_x}\norm{r(C^{(1)}_+)}_{L^\infty_tL^\infty_x}\\
  &\lesssim \kappa^\alpha\,.
\end{aligned}
\end{equation}

We next turn to $c^{\rm bl}_+ (t, x) = C^{(1)}_{+} (t, X)$, which satisfies (see \eqref{eq:C1pluseqn} and \eqref{eq:Cminusexpressionwithq})
\begin{equation}
\begin{aligned}
  &\p_tc^{\rm bl}_+ + \p_x\big((V^{\rm inv}+\beta_+)c^{\rm bl}_+\big) - \kappa\p_x^2c^{\rm bl}_+ - f(c^{\rm bl}_+,c^{\rm bl}_-) \\
  &\quad = \underbrace{(V^{\rm inv}(t,0)+\beta_+) \p_x C^{(1)}_+  - \kappa \p_x^2C^{(1)}_+  - \frac{1}{\kappa}C^{(0)}_-\,r(C^{(1)}_+)}_{=0} \, +\,\mc{R}_{a+} -\mc{R}_{c-} \,.
\end{aligned}
\end{equation}
Here $\mc{R}_{c-}$ is as in \eqref{eq:Rcminus}, while the remainder $\mc{R}_{a+}$ is given by 
\begin{equation}
  \mc{R}_{a+}= \p_tC^{(1)}_+ + \p_x\big((V^{\rm inv}(t,x)-V^{\rm inv}(t,0)) C^{(1)}_+\big)\,.
\end{equation}
In the region $(t,x)\in [0,T]\times[0,\ell_\kappa]$, as in \eqref{eq:Raminusest}, we may bound
\begin{equation}
\begin{aligned}
  \norm{\mc{R}_{a+}}_{L^\infty_tL^1_x}
  &\lesssim \kappa^\alpha\,.
\end{aligned}
\end{equation}

Defining $\mc{R}_-=\mc{R}_{a-}+\mc{R}_{b-}+\mc{R}_{c-}$ and $\mc{R}_+=\mc{R}_{a+}-\mc{R}_{c-}$, we obtain Lemma~\ref{lem:BL_residual}.
\end{proof}

\subsection{Matching bounds}
Given the inner solution $c^{\rm bl}_\pm(t,x)$ in \eqref{eq:ckappa_pm}, valid in the sense of Lemma \ref{lem:BL_residual} within the boundary layer near $x=0$, we will construct a full approximate solution $c_\pm^{\rm app}(t,x)$, $x\in\R_+$. In Case~2, this will be done at the level of $c$, whereas in Case~1, this will be done at the level of $m$.

Let $\varphi$ be a smooth cutoff function on $[0,+\infty)$ satisfying $\varphi' \leq 0$ and
\begin{equation}
  \varphi(y) = \begin{cases}
    1 &  \text{for } 0\le y\le 1\\
    0 & \text{for } y\ge 2\,.
  \end{cases}
\end{equation}
We denote
\begin{equation}\label{eq:varphikappa}
  \varphi^\kappa(x) = \varphi\left(\frac{x}{\kappa^\alpha}\right)\,.
\end{equation}

Before constructing an approximate solution $c^{\rm app}_\pm$, we make note of the following matching estimate within the region ${\rm supp}(\varphi^\kappa)\cap {\rm supp}(1-\varphi^\kappa) \subset [\kappa^\alpha,2\kappa^\alpha]$.
\begin{lemma}[Matching region bounds]\label{lem:matchingbd}
  In the matching region $[0,T] \times [\kappa^\alpha,2\kappa^\alpha]$, the difference $c^{\rm bl}_\pm-c_\pm^{\rm inv}$ between the inner and outer solutions may be bounded as
  \begin{equation}
  \begin{aligned}
    \abs{c^{\rm bl}_\pm-c_\pm^{\rm inv}}\lesssim \kappa^\alpha \,,
    \qquad
    \abs{\p_x(c^{\rm bl}_\pm-c_\pm^{\rm inv})}\lesssim 1\, .
  \end{aligned}
  \end{equation}
\end{lemma}

\begin{proof}
First, using the expression \eqref{eq:Cminusexpressionwithq} for $C^{(0)}_-$ and the matching condition \eqref{eq:C1expdecaytoendstate}, we may write the difference $c^{\rm bl}_-(t,x)-c_-^{\rm inv}(t,0)$ as 
\begin{equation}\label{eq:ckappaminus_eq} 
\begin{aligned}
  &c^{\rm bl}_-(t,x)-c_-^{\rm inv}(t,0) = \underbrace{\frac{1}{\kappa}C^{(0)}_-\left(t,\frac{x}{\kappa}\right)}_{\lesssim \frac{1}{\kappa}e^{-\delta x/\kappa}}+ \underbrace{C^{(1)}_-\left(t,\frac{x}{\kappa}\right)-c_-^{\rm inv}(t,0)}_{|\cdot|\,\lesssim \,e^{-\delta x/(2\kappa)}}\,.
\end{aligned}
\end{equation}

Then, on $[\kappa^\alpha,2\kappa^\alpha]$, we may bound
\begin{equation}
  \abs{c^{\rm bl}_-(t,x)-c_-^{\rm inv}(t,0)} \lesssim (1+\kappa^{-1})\,e^{-\delta\kappa^{\alpha-1}/2}\,.
\end{equation}
Furthermore, by Taylor's theorem, on $[\kappa^\alpha,2\kappa^\alpha]$, we have
\begin{equation}
  \abs{c_-^{\rm inv}(t,x)-c_-^{\rm inv}(t,0)} \lesssim \kappa^\alpha \norm{\p_xc^{\rm inv}_-}_{L^\infty_x}\,.
\end{equation}
In addition, using \eqref{eq:Cminusexpressionwithq} and \eqref{eq:C1minus_derivs}, we may bound 
\begin{equation}
  \abs{\p_x c^{\rm bl}_-(t,x)} \lesssim e^{-\delta\kappa^{\alpha-1}/2}\,,
\end{equation}
so that
\begin{equation}
  \abs{\p_x (c^{\rm bl}_-(t,x)-c_-^{\rm inv}(t,x))} \lesssim e^{-\delta\kappa^{\alpha-1}/2} + \norm{c_-^{\rm inv}}_{L^\infty_tC^1_x}\,.
\end{equation}

Similarly, for $c^{\rm bl}_+(t,x)$, by \eqref{eq:C1expdecaytoendstate}, we have
\begin{equation}\label{eq:ckappaplus_diff}
\begin{aligned}
  \abs{c^{\rm bl}_+(t,x) - c_+^{\rm inv}(t,0)}&\lesssim e^{-\delta x/(2\kappa)} \,,
\end{aligned}
\end{equation}
so that on $[\kappa^\alpha,2\kappa^\alpha]$, we may bound
\begin{equation}
  \abs{c^{\rm bl}_+(t,x)-c_+^{\rm inv}(t,0)} \lesssim e^{-\delta\kappa^{\alpha-1}/2}\,.
\end{equation}
In addition, by \eqref{eq:C1plus_derivs}, we have
\begin{equation}
  \abs{\p_x c^{\rm bl}_+(t,x)}\lesssim e^{-\delta\kappa^{\alpha-1}/2}\,,
\end{equation}
so that
\begin{equation}
  \abs{\p_x c^{\rm bl}_+(t,x)-\p_x c_+^{\rm inv}(t,x)}\lesssim e^{-\delta\kappa^{\alpha-1}/2}+ \norm{c_+^{\rm inv}}_{L^\infty_tC^1_x}\,.
\end{equation}
Altogether, we obtain Lemma \ref{lem:matchingbd}.
\end{proof}

\subsection{Construction of approximate solution and residual bounds with nonlinear tumbling (Case~2)}
To prove a convergence result with nonlinear tumbling, we will require that the field $V$ generated by the particles vanishes at the boundary: $V\big|_{x=0}=0$. We are thus in Case 2.

We construct our full approximate solution by gluing the inner approximation $c^{\rm bl}_\pm(t,x)$ to the outer solution $c^{\rm inv}_\pm(t,x)$ as 
\begin{equation}\label{eq:capp}
\boxed{
  c^{\rm app}_\pm(t,x) = \varphi^\kappa(x)c^{\rm bl}_\pm(t,x) + (1-\varphi^\kappa(x))c^{\rm inv}_\pm(t,x)\,.
  }
\end{equation}

We proceed to plug the approximate solution $c^{\rm app}_\pm$ given by \eqref{eq:capp} into the original equations \eqref{eq:viscousequation} and bound the remainders in terms of $\kappa$. We show the following.
\begin{lemma}[Residual bounds for $c^{\rm app}_\pm$]\label{lem:res_bds}
  The approximate solution $c^{\rm app}_\pm$ constructed in \eqref{eq:capp} satisfies 
\begin{equation}
\begin{aligned}
  &\p_tc_\pm^{\rm app} +\p_x\big((V^{\rm app}\pm \beta_\pm)c_\pm^{\rm app}\big) - \kappa\p_x^2c_\pm^{\rm app} \mp f(c_+^{\rm app},c_-^{\rm app}) =E_\pm \,,
\end{aligned}
\end{equation}
where $V^{\rm app}=V[\bm{c}^{\rm app}](x,t)$ is defined in \eqref{eq:Vexample} and $E_\pm$ satisfies 
\begin{equation}
  \norm{E_\pm}_{L^\infty_tL^1_x([0,T]\times\R_+)} \lesssim \kappa^\alpha+ \kappa^{2\alpha-1}\,. 
\end{equation}
\end{lemma}

\begin{proof}
We begin by writing
\begin{equation}\label{eq:diff_app_term}
\begin{aligned}
  \kappa\p_x^2c^{\rm app}_\pm 
  &=
  \kappa\big(\p_x^2\varphi^\kappa (c^{\rm bl}_\pm-c^{\rm inv}_\pm)+ 2\p_x\varphi^\kappa \p_x(c^{\rm bl}_\pm -c^{\rm inv}_\pm) \\
  &\qquad + \varphi^\kappa \p_x^2c^{\rm bl}_\pm  + (1-\varphi^\kappa)\p_x^2c^{\rm inv}_\pm \big) \\
  &= \varphi^\kappa \,\kappa \p_x^2c^{\rm bl}_\pm + I^\kappa_\pm + I^{\rm inv}_\pm\,,\\
  I^\kappa_\pm &= \kappa\big(\p_x^2\varphi^\kappa (c^{\rm bl}_\pm-c^{\rm inv}_\pm)+ 2\p_x\varphi^\kappa \p_x(c^{\rm bl}_\pm -c^{\rm inv}_\pm) \big)\,,\\
  I^{\rm inv}_\pm &= (1-\varphi^\kappa)\,\kappa\p_x^2c^{\rm inv}_\pm\,.
\end{aligned}
\end{equation}
Using the support of $\p_x\varphi^\kappa$ and Lemma \ref{lem:matchingbd}, we may bound $I^\kappa_\pm$ as
\begin{equation}
\begin{aligned}
  \norm{I^\kappa_\pm}_{L^\infty_tL^1_x} &\lesssim \kappa\kappa^{-2\alpha} \kappa^\alpha \kappa^\alpha + \kappa\kappa^{-\alpha} \kappa^\alpha
  \lesssim \kappa\,.
\end{aligned}
\end{equation}
Furthermore, we may bound $I^{\rm inv}_\pm$ as 
\begin{equation}
\begin{aligned}
  \norm{I^{\rm inv}_\pm}_{L^\infty_tL^1_x} &\le \kappa\norm{\p_x^2c^{\rm inv}_\pm}_{L^\infty_tL^1_x}\,.
\end{aligned}
\end{equation}

Next, for the vector quantity $\bm{c}^{\rm app} = (c_+^{\rm app},c_-^{\rm app})$, we denote 
\begin{equation}
    V^{\rm app} = V[\bm{c}^{\rm app}]= V[\varphi^\kappa \bm{c}^\kappa + (1-\varphi^\kappa)\bm{c}^{\rm inv} ]\,.
\end{equation}
Recalling the definition \eqref{eq:Vinv} of $V^{\rm inv}$ and using the general form \eqref{eq:Vexample} of $V[\cdot]$, we may write
\begin{equation}\label{eq:VappVinv_diff}
\begin{aligned}
  V^{\rm app}-V^{\rm inv} &= V[\bm{c}^{\rm inv}+\varphi^\kappa (\bm{c}^\kappa-\bm{c}^{\rm inv})] - V[\bm{c}^{\rm inv}] - b_-K(x,0) \\
  &= \int_{\R_+}K(x,y)\varphi^\kappa(y) \big(\rho^\kappa(t,y)-\rho^{\rm inv}(t,y)\big)\,dy - b_-K(x,0)\,,
\end{aligned}
\end{equation}
where $\rho^\kappa = c^{\rm bl}_-+c^{\rm bl}_+$ and $\rho^{\rm inv}= c_-^{\rm inv}+c_+^{\rm inv}$. 
We further write
\begin{equation}
\begin{aligned}
  \int_{\R_+}&K(x,y)\varphi^\kappa(y) \big(\rho^\kappa(t,y)-\rho^{\rm inv}(t,y)\big)\,dy - b_-K(x,0) = R_1+R_2+R_3\,,\\
  R_1&= \frac{1}{\kappa}\int_{\R_+}K(x,y)\varphi^\kappa(y)C_-^{(0)}\left(t,\frac{y}{\kappa}\right)\,dy - b_-K(x,0) \\
  R_2 &= \int_{\R_+}K(x,y)\varphi^\kappa(y)\bigg(c^{\rm bl}_-(t,y)- \frac{1}{\kappa}C_-^{(0)}\left(t,\frac{y}{\kappa}\right)- c_-^{\rm inv}(t,0)\\
  &\qquad\qquad  +c^{\rm bl}_+(t,y) - c_+^{\rm inv}(t,0)\bigg)\,dy \\
  R_3 &= \int_{\R_+}K(x,y)\varphi^\kappa(y)\left(c_-^{\rm inv}(t,0)-c_-^{\rm inv}(t,y) +c_+^{\rm inv}(t,0)-c_+^{\rm inv}(t,y)\right)\,dy \,.
\end{aligned}
\end{equation}

We may bound $R_3$ as 
\begin{equation}
\begin{aligned}
  \abs{R_3} &\lesssim \int_{{\rm supp}(\varphi^\kappa)}\abs{K(x,y)}y\big(\norm{\p_xc_-^{\rm inv}}_{L^\infty_tL^\infty_x}+\norm{\p_xc_+^{\rm inv}}_{L^\infty_tL^\infty_x}\big) \,dy\\
  &\lesssim \kappa^{2\alpha}\norm{K(x,\cdot)}_{L^\infty({\rm supp}(\varphi^\kappa))}\,,\\
  \abs{\p_xR_3} &\lesssim \kappa^{2\alpha}\norm{\p_xK(x,\cdot)}_{L^\infty({\rm supp}(\varphi^\kappa))}\,.
\end{aligned}
\end{equation}
Furthermore, using the exponential decay bound \eqref{eq:C1expdecaytoendstate}, we may estimate $R_2$ as 
\begin{equation}
\begin{aligned}
  \abs{R_2} &\lesssim \int_{{\rm supp}(\varphi^\kappa)}\abs{K(x,y)}e^{-\delta y/\kappa}\,dy\\
  &\lesssim \kappa  \norm{K(x,\cdot)}_{L^\infty({\rm supp}(\varphi^\kappa))}\,,\\
  \abs{\p_xR_2} &\lesssim \kappa \norm{\p_xK(x,\cdot)}_{L^\infty({\rm supp}(\varphi^\kappa))}\,.
\end{aligned}
\end{equation}

Finally, we may use \eqref{eq:Bminusexpression} to write $R_1$ as
\begin{equation}
\begin{aligned}
 R_1 &= 
  \underbrace{K(x,0)\frac{1}{\kappa}\int_{\R_+}C_-^{(0)}\left(t,\frac{y}{\kappa}\right)\,dy -b_-K(x,0)}_{=0}\\
  &\quad - \underbrace{K(x,0)\frac{1}{\kappa}\int_{\R_+}(1-\varphi^\kappa(y))C_-^{(0)}\left(t,\frac{y}{\kappa}\right)\,dy}_{R_{1,1}}\\
  &\quad + \underbrace{\frac{1}{\kappa}\int_{\R_+}(K(x,y)-K(x,0))\varphi^\kappa(y)C_-^{(0)}\left(t,\frac{y}{\kappa}\right)\,dy}_{R_{1,2}}\,. 
\end{aligned}
\end{equation}
Using the form \eqref{eq:Cminusexpressionwithq} of $C_-^{(0)}$ and the smoothness of the kernel $K$, we may estimate
\begin{equation}
\begin{aligned}
  \abs{R_{1,1}} &\lesssim e^{-\delta\kappa^{\alpha-1}} |K(x,0)| \\
  \abs{R_{1,2}} &\lesssim \norm{\p_yK(x,\cdot)}_{L^\infty({\rm supp}(\varphi^\kappa))}\int_{{\rm supp}(\varphi^\kappa)}\frac{y}{\kappa} e^{-\delta y/\kappa}\,dy \\
  &\lesssim \kappa \norm{\p_yK(x,\cdot)}_{L^\infty({\rm supp}(\varphi^\kappa))} \, .
\end{aligned}
\end{equation}
In addition, 
\begin{equation}
  \abs{\p_xR_1} \lesssim e^{-\delta\kappa^{\alpha-1}} |\p_x K(x,0)| + \kappa \norm{\p_x\p_yK(x,\cdot)}_{L^\infty({\rm supp}(\varphi^\kappa))}\,.
\end{equation}
We may thus bound the difference $V^{\rm app}-V^{\rm inv}$ as 
\begin{equation}\label{eq:app_inv_diff}
\begin{aligned}
  \abs{V^{\rm app}-V^{\rm inv}} &\le \abs{R_1} + \abs{R_2} + \abs{R_3}\\
  &\lesssim (\kappa^{2\alpha}+\kappa)\norm{K(x,\cdot)}_{L^\infty({\rm supp}(\varphi^\kappa))} \\
  &\qquad + \kappa\norm{\p_yK(x,\cdot)}_{L^\infty({\rm supp}(\varphi^\kappa))}\\
  \abs{\p_x(V^{\rm app}-V^{\rm inv})} &\lesssim (\kappa^{2\alpha}+\kappa)\norm{\p_xK(x,\cdot)}_{L^\infty({\rm supp}(\varphi^\kappa))} \\
  &\qquad+ \kappa\norm{\p_x\p_yK(x,\cdot)}_{L^\infty({\rm supp}(\varphi^\kappa))}\,.
\end{aligned}
\end{equation}
We will further use that $V$ vanishes at $x=0$, i.e. that $K(0,y)=\p_yK(0,y)=0$, to refine the first bound of \eqref{eq:app_inv_diff} to
\begin{equation}\label{eq:app_inv_diff_refine}
\begin{aligned}
  \abs{V^{\rm app}-V^{\rm inv}} 
  \lesssim (\kappa^{2\alpha}+\kappa)\abs{x}\norm{\p_xK}_{  L^{\infty}_{x, y} }
 + \kappa\abs{x}\norm{\p_x\p_yK  }_{  L^{\infty}_{x, y} }\,.
\end{aligned}
\end{equation}

We then rewrite the transport/propulsion terms as  
\begin{equation}
\begin{aligned}
  \p_x\big((V^{\rm app}\pm \beta_\pm)&c^{\rm app}_\pm\big) = 
  \p_x\big(\varphi^\kappa (V^{\rm app}\pm \beta_\pm)c^{\rm bl}_\pm + (1-\varphi^\kappa)(V^{\rm app}\pm \beta_\pm)c^{\rm inv}_\pm \big)\\
  &= \varphi^\kappa \p_x\big((V^{\rm app}\pm \beta_\pm)c^{\rm bl}_\pm\big) + (1-\varphi^\kappa)\p_x\big((V^{\rm app}\pm \beta_\pm)c^{\rm inv}_\pm \big)\\
  &\qquad + \p_x\varphi^\kappa \,(V^{\rm app}\pm \beta_\pm)(c^{\rm bl}_\pm -c^{\rm inv}_\pm)\\
  &=\varphi^\kappa \p_x\big((V^{\rm inv}\pm \beta_\pm)c^{\rm bl}_\pm\big) + (1-\varphi^\kappa)\p_x\big((V^{\rm inv}\pm \beta_\pm)c^{\rm inv}_\pm \big)\\
  &\qquad + H^a_\pm + H^b_\pm + H^c_\pm \,,
\end{aligned}
\end{equation}
where the remainder terms are given by 
\begin{equation}
\begin{aligned}
  H^a_\pm &= \varphi^\kappa \p_x\big((V^{\rm app}-V^{\rm inv})c^{\rm bl}_\pm\big) \\
  H^b_\pm &= (1-\varphi^\kappa)\p_x\big((V^{\rm app}-V^{\rm inv})c^{\rm inv}_\pm \big)\\
  H^c_\pm &= \p_x\varphi^\kappa \,(V^{\rm app}\pm \beta_\pm)(c^{\rm bl}_\pm -c^{\rm inv}_\pm)\,.
\end{aligned}
\end{equation}
Using \eqref{eq:app_inv_diff_refine} and \eqref{eq:app_inv_diff}, we may estimate
\begin{equation}
\begin{aligned}\label{eq:haest}
  \norm{H^a_\pm}_{L^\infty_tL^1_x} &\le \norm{\p_x(V^{\rm app}-V^{\rm inv})}_{L^\infty_tL^\infty_x}\norm{\varphi^\kappa c^{\rm bl}_\pm}_{L^\infty_tL^1_x}\\ 
  &\quad + \norm{V^{\rm app}-V^{\rm inv}}_{L^\infty_tL^\infty_x({\rm supp}(\varphi^\kappa))}\norm{\varphi^\kappa\p_xc^{\rm bl}_\pm}_{L^\infty_tL^1_x} \\
  &\lesssim (\kappa^{2\alpha}+\kappa)\int_{{\rm supp}(\varphi^\kappa)}\frac{1}{\kappa}e^{-\delta x/\kappa}\,dx\\
  &\quad + (\kappa^{3\alpha}+\kappa^{1+\alpha}) \int_{{\rm supp}(\varphi^\kappa)}\frac{1}{\kappa^2}e^{-\delta x/\kappa}\,dx\\
  &\lesssim \kappa^\alpha +\kappa^{3\alpha-1}
\end{aligned}
\end{equation}
as well as 
\begin{equation}
\begin{aligned}
  \norm{H^b_\pm}_{L^\infty_tL^1_x} &\le \norm{\p_x(V^{\rm app}-V^{\rm inv})}_{L^\infty_tL^\infty_x}\norm{(1-\varphi^\kappa)c^{\rm inv}_\pm}_{L^\infty_tL^1_x} \\
  &\quad +
 \norm{V^{\rm app}-V^{\rm inv}}_{L^\infty_tL^\infty_x} \norm{(1-\varphi^\kappa)\p_xc^{\rm inv}_\pm }_{L^\infty_tL^1_x}\\
  &\lesssim (\kappa^{2\alpha}+\kappa)\norm{c^{\rm inv}_\pm }_{L^\infty_t W^{1,1}_x}\,.
\end{aligned}
\end{equation}
By Lemma \ref{lem:matchingbd}, we additionally have
\begin{equation}
\begin{aligned}
  \norm{H^c_\pm}_{L^\infty_tL^1_x} &\lesssim \kappa^{-\alpha}\norm{c^{\rm bl}_\pm -c^{\rm inv}_\pm}_{L^\infty_tL^1_x({\rm supp}(\p_x\varphi^\kappa))} \\
  &\lesssim \kappa^{-\alpha} \kappa^\alpha \kappa^\alpha 
  \lesssim \kappa^\alpha\,,
\end{aligned}
\end{equation}
where we have used that $|{\rm supp}(\p_x\varphi^\kappa)| \leq \kappa^\alpha$.

Finally, we may write the tumbling term $f$ as 
\begin{equation}
\begin{aligned}
  f&(c_+^{\rm app},c_-^{\rm app}) = -c_+^{\rm app}r(c_-^{\rm app}) +c_-^{\rm app}r(c_+^{\rm app}) \\
  &= \varphi^\kappa\big(c^{\rm bl}_-\,r(\varphi^\kappa c^{\rm bl}_+ + (1-\varphi^\kappa)c^{\rm inv}_+) - c^{\rm bl}_+ \,r(\varphi^\kappa c^{\rm bl}_- + (1-\varphi^\kappa)c^{\rm inv}_-) \big) \\
  &\quad + (1-\varphi^\kappa)\big(c^{\rm inv}_-r(\varphi^\kappa c^{\rm bl}_+ + (1-\varphi^\kappa)c^{\rm inv}_+)
  -c^{\rm inv}_+r(\varphi^\kappa c^{\rm bl}_- + (1-\varphi^\kappa)c^{\rm inv}_-)\big) \\
  &= \varphi^\kappa\big(c^{\rm bl}_-\,r(c^{\rm bl}_+) - c^{\rm bl}_+ \,r(c^{\rm bl}_- ) \big) - J^\kappa_+ + J^\kappa_- \\
  &\quad + (1-\varphi^\kappa)\big(c^{\rm inv}_-r(c^{\rm inv}_+)
  -c^{\rm inv}_+r(c^{\rm inv}_-)\big) - J^{\rm inv}_+ + J^{\rm inv}_-\,,
\end{aligned}
\end{equation}
where we define the remainders 
\begin{equation}
\begin{aligned}
  J^\kappa_\pm &= \varphi^\kappa\big( r(c^{\rm bl}_\pm)
  - r(c^{\rm bl}_\pm + (1-\varphi^\kappa)(c^{\rm inv}_\pm-c^{\rm bl}_\pm) \big)c^{\rm bl}_\mp\\
  J^{\rm inv}_\pm &= (1-\varphi^\kappa)\big(r(c^{\rm inv}_\pm) - r(c^{\rm inv}_\pm + \varphi^\kappa (c^{\rm bl}_\pm-c^{\rm inv}_\pm)) \big)c^{\rm inv}_\mp\,.
\end{aligned}
\end{equation}
These may be estimated using Lemma \ref{lem:matchingbd} and~\eqref{eq:tumblingassumptiononr} as 
\begin{equation}
\begin{aligned}
  &\norm{J^\kappa_\pm}_{L^\infty_tL^1_x} \\
  &\le \norm{r(c^{\rm bl}_\pm)
  - r(c^{\rm bl}_\pm + (1-\varphi^\kappa)(c^{\rm inv}_\pm-c^{\rm bl}_\pm)) }_{L^\infty_tL^\infty_x({\rm supp}(\varphi^\kappa))}\norm{\varphi^\kappa c^{\rm bl}_\mp}_{L^\infty_tL^1_x}\\
  &\qquad\lesssim \norm{r'}_{L^\infty}\norm{c^{\rm inv}_\pm-c^{\rm bl}_\pm}_{L^\infty_tL^\infty_x({\rm supp}(\varphi^\kappa)\cap \,{\rm supp}(1-\varphi^\kappa))}\\
  &\qquad\lesssim \kappa^\alpha\,,
\end{aligned}
\end{equation}
where we have used that $\norm{\varphi^\kappa c^{\rm bl}_\mp}_{L^\infty_tL^1_x}\lesssim 1$ (see \eqref{eq:haest}). 
Furthermore, we have
\begin{equation}
\begin{aligned}
  \norm{J^{\rm inv}_\pm}_{L^\infty_tL^1_x} &\le \norm{r(c^{\rm inv}_\pm) - r(c^{\rm inv}_\pm + \varphi^\kappa (c^{\rm bl}_\pm-c^{\rm inv}_\pm)) }_{L^\infty_tL^\infty_x({\rm supp}(1-\varphi^\kappa))}\norm{c^{\rm inv}_\mp}_{L^\infty_tL^1_x}\\
  &\lesssim \norm{r'}_{L^\infty}\norm{c^{\rm inv}_\pm-c^{\rm bl}_\pm}_{L^\infty_tL^\infty_x({\rm supp}(\varphi^\kappa)\cap \,{\rm supp}(1-\varphi^\kappa))}\\
  &\lesssim \kappa^\alpha\,.
\end{aligned}
\end{equation}

In summary, upon plugging $c_\pm^{\rm app}$ into the full system \eqref{eq:viscousequation}, we obtain 
\begin{equation}\label{eq:capp_eqn}
\begin{aligned}
  &\p_tc_\pm^{\rm app} +\p_x\big((V^{\rm app}\pm \beta_\pm)c_\pm^{\rm app}\big) - \kappa\p_x^2c_\pm^{\rm app} \mp f(c_+^{\rm app},c_-^{\rm app}) \\
  &\qquad = \varphi^\kappa \,\bigg( \underbrace{\p_tc^{\rm bl}_\pm + \p_x\big((V^{\rm inv}\pm \beta_\pm)c^{\rm bl}_\pm\big) - \kappa \p_x^2c^{\rm bl}_\pm \mp f(c^{\rm bl}_+,c^{\rm bl}_-)}_{=\mc{R}_\pm} \bigg) \\
  &\quad\qquad + (1-\varphi^\kappa)\bigg(\underbrace{\p_tc^{\rm inv}_\pm + \p_x\big((V^{\rm inv} \pm \beta_\pm)c^{\rm inv}_\pm \big) \mp f(c^{\rm inv}_+,c^{\rm inv}_-)}_{=0}\bigg) \\
  &\quad\qquad  + H^a_\pm + H^b_\pm + H^c_\pm - I^\kappa_\pm - I^{\rm inv}_\pm 
  \pm J^\kappa_+ \mp J^\kappa_-  \pm J^{\rm inv}_+ \mp J^{\rm inv}_- \,.
\end{aligned}
\end{equation}
Defining 
\begin{equation}
  E_\pm =\varphi^\kappa\mc{R}_\pm + H^a_\pm + H^b_\pm + H^c_\pm - I^\kappa_\pm - I^{\rm inv}_\pm \pm J^\kappa_+ \mp J^\kappa_-  \pm J^{\rm inv}_+ \mp J^{\rm inv}_-
\end{equation}
and using Lemma \ref{lem:BL_residual} for $\mc{R}_\pm$ and the above bounds for the $H$, $I$, $J$ remainders, we obtain Lemma \ref{lem:res_bds}. 
\end{proof}

\subsection{Construction of approximate solution and residual bounds with linear tumbling (Case 1)}\label{subsec:approxCase1}
We next turn to Case~1. To incorporate the effects of a more general drift velocity $V$ that does not vanish at the boundary, we will use the system \eqref{eq:MAIN_massvars} for the mass variables $m_\pm$ instead of working directly with the $c_\pm$ equations. 

Define
\begin{equation}\label{eq:minv_mkappa}
  m_\pm^{\rm inv}(t,x) = \int_0^x d\mu_\pm^{\rm inv}(t,x') \,,\qquad 
  m^{\rm bl}_\pm(t,x) = \int_0^xc^{\rm bl}_\pm(t,x')\,dx'\,,
\end{equation}
for $\mu_+^{\rm inv} =  c_+^{\rm inv} \, dx$, $\mu_-^{\rm inv} = c_-^{\rm inv} \, dx + \delta_0(x) b_-(t)$, and $c^{\rm bl}_\pm$ as in \eqref{eq:ckappa_pm}. We take
\begin{equation}\label{eq:mapp}
\boxed{
  m_\pm^{\rm app}(t,x) = \varphi^\kappa(x) m^{\rm bl}_\pm(t,x)+ (1-\varphi^\kappa(x))m_\pm^{\rm inv}(t,x)
}
\end{equation}
for $\varphi^\kappa$ as in \eqref{eq:varphikappa}, and 
\begin{equation}
    \label{eq:cpmappform}
\begin{aligned}
    c_{\pm}^{\rm app}(t,x) &:= \p_x m_\pm^{\rm app}(t,x) \\
    &= \varphi^\kappa(x) c^{\rm bl}_\pm(t,x)+ (1-\varphi^\kappa(x)) c_\pm^{\rm inv}(t,x) + \p_x \varphi^\kappa (m_\pm^{\rm bl} - m_{\pm}^{\rm inv}) \, .
\end{aligned}
\end{equation}
Notably,~\eqref{eq:cpmappform} differs from~\eqref{eq:capp} by the matching term $\p_x \varphi^\kappa (m_\pm^{\rm bl} - m_{\pm}^{\rm inv})$. We define
\begin{equation}
    \label{eq:Vappfromdefm}
\begin{aligned}
  V^{\rm inv}(t,x)&=V[\bm{c}^{\rm inv}] +b_{-} K (x, 0) \\
  V^{\rm app}(t,x)&=V[\bm{c}^{\rm app}] \\
  &= \underbrace{V[\varphi^\kappa\p_x \bm{m}^{\rm bl} + (1-\varphi^\kappa)\p_x \bm{m}^{\rm inv}]}_{=: \hat{V}^{\rm app}} + V[\p_x \varphi^\kappa (\bm{m}^{\rm app} - \bm{m}^{\rm inv})] \, .
\end{aligned} 
\end{equation}
Here, $\hat{V}^{\rm app}$ corresponds to the definition of $V^{\rm app}$ from Case~2.

By integrating the results of Lemmas \ref{lem:BL_residual} and \ref{lem:matchingbd}, we obtain the following bounds for $m^{\rm bl}_\pm$ within the boundary layer and matching region.
\begin{lemma}[Boundary layer and matching residuals for mass variable]\label{lem:mass_BL}
The inner mass variable solution $m^{\rm bl}_\pm(t,x)$ given by \eqref{eq:minv_mkappa} satisfies 
  \begin{equation}
  \begin{aligned}
    \p_t m^{\rm bl}_\pm + (V^{\rm inv}\pm\beta_\pm) \p_x m^{\rm bl}_\pm - \kappa \p_x^2 m^{\rm bl}_\pm \pm r_0 (m_+^{\rm bl} - m_-^{\rm bl}) &= \mc{R}^m_\pm \, .
    \end{aligned}
    \end{equation}
    Near the boundary, we have
    \begin{equation}
        \begin{aligned}
    \norm{\mc{R}_\pm^m}_{L^\infty_tL^1_x([0,T]\times[0,\ell_k])} &\lesssim \kappa^{2\alpha}+ \kappa^{3\alpha-1}\, .
  \end{aligned}
  \end{equation}

Furthermore, within the matching region $[0,T] \times [\kappa^\alpha,2\kappa^\alpha]$, we have 
  \begin{equation}
  \begin{aligned}
      \abs{m^{\rm bl}_\pm-m_\pm^{\rm inv}} \lesssim \kappa^{2\alpha}\,,
      \qquad
      \abs{\p_x(m^{\rm bl}_\pm-m_\pm^{\rm inv})} \lesssim \kappa^{\alpha}\, .
  \end{aligned}    
  \end{equation}  
\end{lemma}
Note that each of the bounds gains a factor of $\kappa^\alpha$ over the analogous bound for $c^{\rm bl}_\pm = \p_xm^{\rm bl}_\pm$.

We may then show the following residual bounds for $m_\pm^{\rm app}$ given by \eqref{eq:mapp}.
\begin{lemma}[Residual bounds for $m_\pm^{\rm app}$]\label{lem:resbds_mass}
The mass variable approximation $m_\pm^{\rm app}$ given by \eqref{eq:mapp} satisfies
  \begin{equation}
    \p_t m_\pm^{\rm app} + (V^{\rm app}\pm\beta_\pm) \p_x m_\pm^{\rm app} - \kappa \p_x^2 m_\pm^{\rm app} \pm r_0 (m_+^{\rm app} - m_-^{\rm app}) = E_\pm^m \,,
  \end{equation}
where $E_\pm^m$ satisfies
\begin{equation}
  \norm{E_\pm^m}_{L^\infty_tL^1_x([0,T]\times\R_+)} \lesssim \kappa^{\alpha}+\kappa^{2\alpha-1}\,.
\end{equation}
\end{lemma}

\begin{proof}
First, using an analogous decomposition to \eqref{eq:diff_app_term} for the diffusion term, we may write   
\begin{equation}
\begin{aligned}
  &\p_t m_\pm^{\rm app}-\kappa \p_x^2 m_\pm^{\rm app}\pm r_0(m_+^{\rm app}-m_-^{\rm app})\\
  &\quad =
  \varphi^\kappa \left(\p_tm^{\rm bl}_\pm - \kappa\p_x^2m^{\rm bl}_\pm 
  \pm r_0(m^{\rm bl}_+-m^{\rm bl}_-) \right)\\
  &\qquad + (1-\varphi^\kappa) \left( \p_tm_\pm^{\rm inv} \pm r_0(m_+^{\rm inv}-m_-^{\rm inv}) \right)
  - I^m_\pm\,,\\
  I_\pm^m &= \kappa\big(\p_x^2\varphi^\kappa(m^{\rm bl}_\pm-m_\pm^{\rm inv})+2\p_x\varphi^\kappa\p_x(m^{\rm bl}_\pm-m_\pm^{\rm inv})\big)
  +(1-\varphi^\kappa)\kappa\p_x^2m_\pm^{\rm inv}\,.
\end{aligned}
\end{equation}
Using the matching region bounds of Lemma \ref{lem:matchingbd} and that $\p_x^2m^{\rm inv}_\pm=\p_xc^{\rm inv}_\pm$ in the interior, we may bound
\begin{equation}
  \norm{I^m_\pm}_{L^\infty_tL^1_x} \lesssim \kappa \,.
\end{equation}

Next, turning to the transport/propulsion term, we note that the difference $\hat{V}^{\rm app}-V^{\rm inv}$ continues to satisfy the bound \eqref{eq:app_inv_diff}. However, we are no longer able to take advantage of additional smallness in the boundary layer since $V$ is no longer required to vanish at $x=0$. We will thus use that, from \eqref{eq:app_inv_diff}, we have 
\begin{equation}\label{eq:Vdiff_non0}
\begin{aligned}
  \abs{\hat{V}^{\rm app}-V^{\rm inv}} 
  &\lesssim \kappa^{2\alpha}+\kappa\,.
\end{aligned}
\end{equation}
Additionally, from Lemma~\ref{lem:mass_BL} and the mapping property $V : L^1(\R_+) \to L^\infty(\R_+)$,
\begin{equation}
    \label{eq:extraVterm}
     |V[\p_x \varphi^\kappa (\bm{m}^{\rm app} - \bm{m}^{\rm inv})]| \lesssim \kappa^{2\alpha} \, . 
\end{equation}
Combining~\eqref{eq:Vappfromdefm},~\eqref{eq:Vdiff_non0}, and~\eqref{eq:extraVterm}, we obtain
\begin{equation}
    \label{eq:VappVinvdiff}
    |V^{\rm app} - V^{\rm inv}| \lesssim \kappa^{2\alpha} +\kappa\, . 
\end{equation}
We may then write 
\begin{equation}
\begin{aligned}
  (V^{\rm app}\pm\beta_\pm)\p_xm^{\rm app}_\pm &= 
  \varphi^\kappa(V^{\rm inv}\pm\beta_\pm)\p_xm^{\rm bl}_\pm
  +(1-\varphi^\kappa)(V^{\rm inv}\pm\beta_\pm)\p_xm^{\rm inv}_\pm\\
  &\quad + H^{m,a}_\pm + H^{m,b}_\pm + H^{m,c}_\pm
\end{aligned}
\end{equation}
where the remainder terms are given by
\begin{equation}
\begin{aligned}
  H^{m,a}_\pm &= \varphi^\kappa(V^{\rm app}-V^{\rm inv})\p_xm^{\rm bl}_\pm \\
  H^{m,b}_\pm &= (1-\varphi^\kappa)(V^{\rm app}-V^{\rm inv})\p_xm^{\rm inv}_\pm \\
  H^{m,c}_\pm &= (V^{\rm app}\pm\beta_\pm)\p_x\varphi^\kappa (m^{\rm bl}_\pm-m^{\rm inv}_\pm)\,.
\end{aligned}
\end{equation}
Using \eqref{eq:VappVinvdiff} and Lemma \ref{lem:mass_BL} and proceeding as in \eqref{eq:haest}, we may bound each of these remainders in turn as
\begin{equation}
\begin{aligned}
  \norm{H^{m,a}_\pm}_{L^\infty_tL^1_x} &\le \norm{V^{\rm app}-V^{\rm inv}}_{L^\infty_tL^\infty_x}\norm{\varphi^\kappa c^{\rm bl}_\pm}_{L^\infty_tL^1_x} \\
  &\lesssim  \kappa^{2\alpha}  +\kappa \\ 
  \norm{H^{m,b}_\pm}_{L^\infty_tL^1_x} &\le \norm{V^{\rm app}-V^{\rm inv}}_{L^\infty_tL^\infty_x}\norm{(1-\varphi^\kappa)c^{\rm inv}_\pm}_{L^\infty_tL^1_x} \\
  &\lesssim (\kappa^{2\alpha} +\kappa) \norm{c^{\rm inv}_\pm}_{L^\infty_tL^1_x}\ \\
  \norm{H^{m,c}_\pm}_{L^\infty_tL^1_x} &\le  \norm{V^{\rm app}\pm\beta_\pm}_{L^\infty_tL^\infty_x}\norm{\p_x\varphi^\kappa(m^{\rm bl}_\pm-m^{\rm inv}_\pm)}_{L^\infty_tL^1_x}\\
  &\lesssim \kappa^{2\alpha}\,.
\end{aligned}
\end{equation}

Altogether, inserting $m^{\rm app}_\pm$ into the system \eqref{eq:MAIN_massvars}, we obtain
\begin{equation}
\begin{aligned}
  &\p_t m_\pm^{\rm app} + (V^{\rm app}\pm\beta_\pm) \p_x m_\pm^{\rm app} - \kappa \p_x^2 m_\pm^{\rm app} \pm r_0 (m_+^{\rm app} - m_-^{\rm app})\\
  &\qquad = \varphi^\kappa\left(\underbrace{\p_t m^{\rm bl}_\pm + (V^{\rm inv}\pm\beta_\pm) \p_x m^{\rm bl}_\pm - \kappa \p_x^2 m^{\rm bl}_\pm \pm r_0 (m^{\rm bl}_+ - m^{\rm bl}_-)}_{=\mc{R}^m_\pm} \right)\\
  &\qquad\quad +(1-\varphi^\kappa)\left(\underbrace{\p_t m_\pm^{\rm inv} + (V^{\rm inv}\pm\beta_\pm) \p_x m_\pm^{\rm inv} \pm r_0 (m_+^{\rm inv} - m_-^{\rm inv})}_{=0} \right)\\
  &\qquad\qquad - I^m_\pm +H^{m,a}_\pm + H^{m,b}_\pm + H^{m,c}_\pm\,.
\end{aligned}
\end{equation}
Defining 
\begin{equation}
  E_\pm^m = \varphi^\kappa\mc{R}^m_\pm - I^m_\pm +H^{m,a}_\pm + H^{m,b}_\pm + H^{m,c}_\pm\,,
\end{equation}
we obtain Lemma~\ref{lem:resbds_mass}.
\end{proof}

\section{Vanishing diffusivity limit and proof of Theorem~\ref{thm:mainthm}}
\label{sec:vanishingdiff}

In this section, we estimate the difference between a solution to the diffusive system~\eqref{eq:viscousequation}-\eqref{eq:Vexample} and a \emph{given} approximate solution to the same system. We have in mind that the approximate solutions are those furnished by Section~\ref{sec:approxsols}. We suppose that the kernel $K$ satisfies~\eqref{eq:smoothnessofK}-\eqref{eq:decayonkernel} and $f$ satisfies~\eqref{eq:fintermsofr}-\eqref{eq:tumblingassumptiononr}. It will be convenient to write the system \eqref{eq:viscousequation}-\eqref{eq:nofluxBCs} in the compact form
\begin{align}
    \p_t \bm{c} + \p_x ((V + B) \bm{c}) &= \kappa \p_x^2 \bm{c} +\bm{R}[\bm{c}]  \label{eq:mainequation}\\
    (V + B - \kappa \p_x) \bm{c}\big|_{x=0} &= 0  \\
V &= V[\bm{c}] \, , \label{eq:mainequationc}
\end{align}
where $\bm{c} = (c_+,c_-)$, $B = {\rm diag}(\beta_+,-\beta_-)$, and $(\bm{R}[\bm{c}])_{\pm} = \pm f(c_+,c_-)$.

\begin{proposition}[Error estimate in Case 2]
    \label{pro:errorestcase2}
    Let $T > 0$ and $\bm{c}^{\rm in} \in L^1(\R_+)$. Invoke the above assumptions on $K$ and $f$ and assume that Case 2 ($K(0,y) = 0$ for all $y \geq 0$) holds. Suppose that there exists $\kappa_0 > 0$ such that, for every $\kappa \in (0,\kappa_0]$, there exists an approximate solution $\bm{c}^{\rm app}$ satisfying
    \begin{align}
    \sup_{\kappa \in (0,\kappa_0]} \| |\bm{c}^{\rm app}| + \min(|x|,1) &|\p_x \bm{c}^{\rm app}| \|_{L^\infty_t L^1_x([0,T] \times \R_+)} < +\infty \,, \label{eq:capp_1}\\
\label{eq:capp_2}
    \p_t \bm{c}^{\rm app} + \p_x ((V^{\rm app} + B) \bm{c}^{\rm app}) &= \kappa \p_x^2 \bm{c}^{\rm app} +\bm{R}[\bm{c}^{\rm app}] - \bm{E}^{\kappa} \, , \\
    (B - \kappa \p_x) \bm{c}^{\rm app}\big|_{x=0} &= 0 \, , \label{eq:capp_3}
\end{align}
where $\bm{E}^{\kappa} \in L^\infty_t L^1_x([0,T] \times \R_+)$, $V^{\rm app} := V[\bm{c}^{\rm app}]$, and
\begin{equation}
    \bm{c}^{\rm in} - \bm{c}^{\rm app}\big|_{t=0} =: \bm{E}_0^{\kappa} \in L^1(\R_+) \, .
\end{equation}
Let $\bm{c}^\kappa$ be the solution to~\eqref{eq:mainequation}-\eqref{eq:mainequationc} with initial condition $\bm{c}^{\rm in}$. Then
\begin{equation}
    \| \bm{c}^\kappa - \bm{c}^{\rm app} \|_{L^\infty_t L^1_x([0,T] \times \R_+)} \lesssim_T \| \bm{E}^{\kappa} \|_{L^\infty_t L^1_x([0,T] \times \R_+)} + \| \bm{E}_0^{\kappa} \|_{L^1_x(\R_+)} \, .
\end{equation}
\end{proposition}

\begin{remark}
    In Case~2, the existence and uniqueness of an exact solution $\bm{c}$ to the diffusive system~\eqref{eq:mainequation}-\eqref{eq:mainequationc} can be established via a fixed point argument in the space $C([0,T];L^1(\R_+))$.
\end{remark}

\begin{proof}
Let $\tilde{\bm{c}} := \bm{c}^\kappa - \bm{c}^{\rm app}$ and $V^\kappa := V[\bm{c}^\kappa]$. Then
\begin{equation}
\begin{aligned}
    &\p_t \tilde{\bm{c}} + B \p_x \tilde{\bm{c}}  - \kappa \p_x^2 \tilde{\bm{c}}  = - \p_x ([V^\kappa - V^{\rm app}] \bm{c}^{\rm app}) - \partial_x (V^\kappa \tilde{\bm{c}})
     \\
    &\qquad   + (\bm{R}[\bm{c}^\kappa] - \bm{R}[\bm{c}^{\rm app}]) + \bm{E}^{\kappa}
\end{aligned}
\end{equation}
with boundary conditions
\begin{equation}
    (B - \kappa \p_x) \tilde{\bm{c}}\big|_{x=0} = 0 \, .
\end{equation}
We estimate both components of $\tilde{\bm{c}}$ in $L^1$. That is, we multiply each equation by $\sgn \tilde{c}_{\pm}$, respectively, integrate over $\R_+$, and sum the equations.\footnote{To justify the computations, one actually multiplies by  $\eta_\varepsilon'(\tilde{c}_{\pm})$, where $\eta_\varepsilon$ is a suitable approximation of $|x|$, e.g., $\eta_\varepsilon(u) = \sqrt{u^2 + \varepsilon^2} - \varepsilon$ (cf. \cite{berlyand2026multiscale}).} This requires estimating various terms in $L^1$; in particular, 
\begin{equation}
\begin{aligned}
        &\int_{\R_+} |\p_x ([V^\kappa - V^{\rm app}] \bm{c}^{\rm app})|\,dx\\
        &\quad\leq \int_{\R_+} |\p_x (V^\kappa - V^{\rm app})| \, |\bm{c}^{\rm app}|\,dx + \int_{\R_+} |V^\kappa - V^{\rm app}| \, |\p_x \bm{c}^{\rm app}|\,dx\\
        & \quad \lesssim  \int_{\R_+} \big(|\tilde{c}_+| + |\tilde{c}_-| \big)\,dx \, ,
\end{aligned}
\end{equation}
where
\begin{equation}
\begin{aligned}
    \| V^\kappa - V^{\rm app} \|_{\rm Lip} &\lesssim \| \tilde{\bm{c}} \|_{L^1} \, , \quad \| \bm{c}^{\rm app} \|_{L^1} \lesssim 1\\
    \| (V^\kappa - V^{\rm app})/\min(x,1) \|_{L^\infty} &\lesssim \| \tilde{\bm{c}} \|_{L^1} \, , \quad \| \min(x,1) \p_x \bm{c}^{\rm app} \|_{L^1} \lesssim 1\,,
\end{aligned}
\end{equation}
Furthermore, expanding $\partial_x (V^\kappa \tilde{c}_{\pm})$, integrating by parts, and using the fact that $V^\kappa|_{x=0}=0$, we obtain
\begin{align*}
    \int_{\R_{+}} \partial_x (V^\kappa \tilde{c}_{\pm}) \, ({\rm sgn} \,  \tilde{c}_{\pm}) \, dx  = 0 \, .
\end{align*}
Next, by \eqref{eq:tumblingassumptiononr},
\begin{equation}
\begin{aligned}
    \int_{\R_+} |f(\bm{c}^\kappa) - f(\bm{c}^{\rm app})| \, dx \lesssim \int_{\R_+} |\tilde{\bm{c}}| \,dx\, .
\end{aligned}
\end{equation}
Altogether, we ultimately arrive at the differential inequality
\begin{equation}
    \frac{d}{dt} \int_{\R_+} \big(|\tilde{c}_+| + |\tilde{c}_-|\big)\,dx \leq C \int_{\R_+} \big(|\tilde{c}_+| + |\tilde{c}_-| \big)\,dx + \int_{\R_+} \big(|E_+^\kappa| + |E_-^\kappa|\big)\,dx \, .
\end{equation}
By Gr{\"o}nwall's inequality, we then obtain 
\begin{equation}
    \int_{\R_+} \big(|\tilde{c}_+| + |\tilde{c}_-|\big)\,dx \leq e^{Ct}\big( \| E_{0,+}^\kappa \|_{L^1} + \| E_{0,-}^\kappa \|_{L^1} +  t \| \bm{E}^\kappa \|_{L^\infty_t L^1_x([0,t] \times \R_+)}\big) \, .
\end{equation}
\end{proof}

In Case~1, we exploit the mass variable formulation:
\begin{align}
    \p_t \bm{m} + (V + B) \p_x \bm{m} &= \kappa \p_x^2 \bm{m} + R \bm{m}  \label{eq:massvariablesec3a} \\
    {\bm{m}}\big|_{x=0} &= 0  \label{eq:massvariablesec3b} \\
    V &= V[\p_x \bm{m}] \,, \label{eq:massvariablesec3c}
\end{align}
with the correspondence $\bm{m}(x) = \int_0^x \bm{c}(y) \, dy$ and tumbling operator $R := r_0 \left[\begin{smallmatrix}
    -1 & 1 \\
    1 & -1
\end{smallmatrix} \right]$.

\begin{proposition}[Error estimate in Case 1]
    \label{pro:errorestcase1}
Let $T > 0$ and $\bm{c}^{\rm in} \in L^1(\R_+)$. Define $\bm{m}^{\rm in}(x) := \int_0^x \bm{c}^{\rm in}(y) \, dy$. Invoke the above assumptions on $K$ and suppose that Case~1 ($r \equiv r_0 > 0$ const.) holds.

Suppose that there exists $\kappa_0 > 0$ such that, for every $\kappa \in (0,\kappa_0]$, there exists an approximate solution $\bm{m}^{\rm app}$ satisfying 
    \begin{align}
    \sup_{\kappa \in (0,\kappa_0]} \| \p_x \bm{m}^{\rm app} &\|_{L^\infty_t L^1_x([0,T] \times \R_+)} < +\infty  \label{eq:bmappass} \\
\label{eq:bmappass2}
    \p_t \bm{m}^{\rm app} + (V^{\rm app} + B) \p_x \bm{m}^{\rm app} &= \kappa \p_x^2 \bm{m}^{\rm app} + R \bm{m}^{\rm app} - \bm{E}^{m,\kappa} \\
    {\bm{m}}^{\rm app}\big|_{x=0} &= 0  \label{eq:bmappass3}
\end{align}
where $\bm{E}^{m,\kappa} \in L^\infty_t L^1_x([0,T] \times \R_+)$, $V^{\rm app} = V[\p_x \bm{m}^{\rm app}]$, 
\begin{equation}
    \bm{m}^{\rm in} - \bm{m}^{\rm app}\big|_{t=0} =: \bm{E}_0^{m,\kappa} \in L^1(\R_+) \, .
\end{equation}
Let $\bm{m}^\kappa$ be the solution to~\eqref{eq:massvariablesec3a}-\eqref{eq:massvariablesec3c}.
Then
\begin{equation}
    \| \bm{m}^\kappa - \bm{m}^{\rm app} \|_{L^\infty_t L^1_x([0,T] \times \R_+)} \lesssim \| \bm{E}^{m,\kappa} \|_{L^\infty_t L^1_x([0,T] \times \R_+)} + \| \bm{E}_0^{m,\kappa} \|_{L^1_x(\R_+)} \, .
\end{equation}
\end{proposition}

\begin{remark}
    In Case~1, existence and uniqueness of an exact solution $\bm{m}$ to~\eqref{eq:massvariablesec3a}-\eqref{eq:massvariablesec3c} can be established via fixed point argument in $C([0,T];X)$, where $X$ consists of $\dot W^{1,1}(\R_+)$ functions vanishing at $x=0$. Then $\bm{c} = \p_x \bm{m}$ furnishes a solution to~\eqref{eq:mainequation} which satisfies the boundary conditions in a weak sense. That solutions $\tilde{\bm{m}} \in C([0,T];X)$ to the equation~\eqref{eq:tildemequatoin} below additionally belong to $C([0,T];L^1)$ may also be established by fixed point argument.
\end{remark}

\begin{proof}
    Let $\tilde{\bm{m}} := \bm{m}^\kappa - \bm{m}^{\rm app}$. Then $\tilde{\bm{m}}$ satisfies
    \begin{equation}\label{eq:tildemequatoin}
    \p_t \tilde{\bm{m}} + B \p_x \tilde{\bm{m}} -\kappa \p_x^2 \tilde{\bm{m}}=- (V^\kappa - V^{\rm app}) \p_x \bm{m}^\kappa - V^{\rm app} \p_x \tilde{\bm{m}}  + R \tilde{\bm{m}} + \bm{E}^{m,\kappa}
\end{equation}
with the zero Dirichlet condition
\begin{equation}
    \tilde{\bm{m}}\big|_{x=0} = 0 \, .
\end{equation}
As in the proof of Proposition~\ref{pro:errorestcase2}, we multiply the equations by suitable approximation of ${\rm sgn} \, \tilde{m}_{\pm}$ and perform $L^1$ estimates on both components of $\tilde{\bm{m}}$. This again requires estimating various terms in $L^1$. First,
\begin{equation}
    \int_{\R_+} |V^\kappa -  V^{\rm app}| \, |\p_x \bm{m}^{\kappa}|\,dx \leq \| V^\kappa -V^{\rm app} \|_{L^\infty} \| \p_x \bm{m}^{\kappa} \|_{L^1} \lesssim \| \tilde{\bm{m}} \|_{L^1} \, ,
\end{equation}
by the assumptions on $V[\cdot]$ and the \emph{a priori} control
\begin{equation}
    \| |\p_x m_+| + |\p_x m_-| \|_{L^1} = \| |c_+| + |c_-| \|_{L^1} \leq \| |c^{\rm in}_+| + |c^{\rm in}_-| \|_{L^1} \, .
\end{equation}
Second, integrating by parts in the integral involving the second term on the right-hand side of~\eqref{eq:tildemequatoin}, we have
\begin{equation}
    \int_{\R_+} |\p_x V^{\rm app}| \, |\tilde{\bm{m}}|\,dx \lesssim \| \tilde{\bm{m}} \|_{L^1} \, ,
\end{equation}
by the assumptions on $V[\cdot]$ and~\eqref{eq:bmappass}. Third, we have
\begin{equation} 
    R \bm{\tilde{m}} \cdot (\sgn \tilde{m}_+,\sgn \tilde{m}_-) = (\tilde{m}_- - \tilde{m}_+) \sgn \tilde{m}_+ + (\tilde{m}_+ - \tilde{m}_-) \sgn \tilde{m}_- \leq 0 \, .
\end{equation}
Altogether, we again have
\begin{equation}
\begin{aligned}
    \frac{d}{dt} \int_{\R_+} \big(|\tilde{m}_+| + |\tilde{m}_-|\big)\,dx &\leq C \int_{\R_+} \big(|\tilde{m}_+| + |\tilde{m}_-|\big)\,dx \\
    &\qquad + \int_{\R_+} \big(|E_+^{m,\kappa}| + |E_-^{m,\kappa}|\big)\,dx \, .
\end{aligned}
\end{equation}
We therefore conclude by Gr{\"o}nwall's inequality.
\end{proof}

With Propositions~\ref{pro:errorestcase2} and~\ref{pro:errorestcase1} in hand, we can complete the proof of Theorem~\ref{thm:mainthm}.

\begin{proof}[Proof of Theorem~\ref{thm:mainthm}]
Fix $\alpha \in [1/2,1)$. We begin with \textbf{Case~2}, in which case $\bm{c}^{\rm app}$ is given directly by \eqref{eq:capp} for $\varphi^\kappa$ as in \eqref{eq:varphikappa}.

 Suppose the hypotheses of the theorem hold; in particular, $b_-^{\rm in}=0$. Recalling the form \eqref{eq:ckappa_pm} of $c_\pm^{\rm bl}(x,t)$, we may then calculate that $c_\pm^{\rm bl}\big|_{t=0}= C^{(1)}_{\pm}\big|_{t=0}$, so that the initial boundary layer approximation is bounded independent of $\kappa$:
    \begin{equation}\label{eq:cbl_initial}
        \abs{c_\pm^{\rm bl}\big|_{t=0}} \lesssim 1\,.
    \end{equation}
Since $\bm{c}^{\rm app}\big|_{t=0}= \varphi^\kappa\bm{c}^{\rm bl}\big|_{t=0} + (1-\varphi^\kappa) \bm{c}^{\rm in}$, by \eqref{eq:cbl_initial}, the initial condition error satisfies
    \begin{equation}
        \norm{\bm{E}_0^{\kappa}}_{L^1_x(\R_+)} \lesssim \kappa^\alpha \,.
    \end{equation}
   
We further note that 
\begin{equation}
    \norm{\bm{c}^{\rm app}}_{L^\infty_tL^1_x([0,T]\times \R_+)} \le \norm{\bm{c}^{\rm bl}}_{L^\infty_tL^1_x([0,T]\times\rm{supp}(\varphi^\kappa))}
    + \norm{\bm{c}^{\rm inv}}_{L^\infty_tL^1_x([0,T]\times \R_+)}
    \lesssim 1\,,
\end{equation}
independent of $\kappa$, and by \eqref{eq:C0sol}-\eqref{eq:Bminusexpression}, \eqref{eq:C1plus_derivs}-\eqref{eq:C1minus_derivs}, and Lemma~\ref{lem:matchingbd},
\begin{equation}
\begin{aligned}
    &\norm{{\rm min}(x,1)\abs{\p_x\bm{c}^{\rm app}}}_{L^\infty_tL^1_x([0,T]\times \R_+)}
    \le \norm{x\p_x\bm{c}^{\rm bl}}_{L^\infty_tL^1_x([0,T]\times {\rm supp}(\varphi^\kappa))}
    \\
    &\quad 
    + \norm{\p_x \bm{c}^{\rm inv}}_{L^\infty_tL^1_x([0,T]\times \R_+)}
    + \norm{x\abs{\p_x\varphi^\kappa (\bm{c}^{\rm bl} - \bm{c}^{\rm inv})}}_{L^\infty_tL^1_x([0,T]\times {\rm supp}(\p_x\varphi^\kappa))} \\
    &\qquad\qquad \lesssim \norm{\frac{x}{\kappa^2}e^{-\frac{\delta x}{2\kappa}}}_{L^\infty_tL^1_x([0,T]\times {\rm supp}(\varphi^\kappa))}
    + \sum_{\pm} \norm{c^{\rm inv}_{\pm}}_{X^{1,1}(T)} + \kappa^\alpha \, ,
\end{aligned}
\end{equation}
again independent of $\kappa$; in particular, $\bm{c}^{\rm app}$ satisfies the condition \eqref{eq:capp_1}.

Using Lemma~\ref{lem:res_bds}, we then have that $\bm{c}^{\rm app}(t,x)$ satisfies \eqref{eq:capp_1}-\eqref{eq:capp_3} with 
\begin{equation}
    \norm{\bm{E}^\kappa}_{L^\infty_tL^1_x([0,T]\times\R_+)}
    \lesssim \kappa^\alpha +\kappa^{2\alpha-1}\,.
\end{equation}
Thus by Proposition \ref{pro:errorestcase2}, we have that 
\begin{equation}
    \norm{\bm{c}^\kappa-\bm{c}^{\rm app}}_{L^\infty_tL^1_x([0,T]\times\R_+)}
    \lesssim \kappa^\alpha +\kappa^{2\alpha-1}\,.
\end{equation}

 We next verify convergence of $c^{\rm app}_{\pm}$ to $\mu^{\rm inv}_{\pm}$ as in \eqref{eq:cappdiff2}. It is evident that $\norm{\cdot}_{\rm FM(\overline{\R}_+)} \leq \| \cdot \|_{L^1(\R_+)}$ on $L^1(\R_+)$ functions, and we first note that
\begin{equation}\label{eq:needbd1}
    \sup_{t\in[0,T]}\norm{(1-\varphi^\kappa) \bm{c}^{\rm inv}-\bm{c}^{\rm inv}}_{{\rm FM}(\overline{\R}_+)}
    \le
    \sup_{t\in[0,T]}\norm{\varphi^\kappa \bm{c}^{\rm inv}}_{L^\infty_tL^1_x([0,T]\times \R_+)} \lesssim \kappa^\alpha \,.
\end{equation}
Furthermore, using that $C^{(1)}_\pm(t,\frac{x}{\kappa})$ in \eqref{eq:ckappa_pm} is bounded in $x$ (see \eqref{eq:C1expdecaytoendstate}), we have 
\begin{equation}\label{eq:needbd2}
    \sup_{t\in[0,T]}\norm{\varphi^\kappa C^{(1)}_\pm}_{{\rm FM}(\overline{\R}_+)} \lesssim \kappa^\alpha\,.
\end{equation}
Finally, we show that $\frac{\varphi^\kappa}{\kappa}C^{(0)}_-$ converges to $\delta_0(x) b_-(t)$. We will require a variation of the fact that, for any $\epsilon>0$,
\begin{equation}\label{eq:to_delta}
    \| \frac{\varphi^{c\epsilon}}{\epsilon} e^{-x/\epsilon} - \delta_0 \|_{\rm FM(\bar{\R}_+)} \lesssim \epsilon \, ,
\end{equation}
where $\varphi^{c\epsilon}$ is a cutoff as in \eqref{eq:varphikappa} for any $\alpha \in[1/2,1)$ and any $c>0$.
Indeed, by the test function characterization \eqref{eq:fmnorm} of the ${\rm FM}$ norm, we obtain 
\begin{equation}
\begin{aligned}
    &\left| \int_0^{+\infty} \frac{\varphi^{c\epsilon}}{\epsilon} e^{-x/\epsilon} \psi(x) \, dx - \psi(0) \right| \\
    &\quad\le \left| \int_0^{+\infty} \frac{1}{\epsilon} e^{-x/\epsilon} \psi(x) \, dx - \psi(0) \right| + \left| \int_0^{+\infty} \frac{1-\varphi^{c\epsilon}}{\epsilon} e^{-x/\epsilon} \psi(x) \, dx  \right|\\
    &\quad\leq \underbrace{\| \p_x \psi \|_{L^\infty} \int_0^{+\infty} \frac{x}{\epsilon} e^{-x/\epsilon}\,dx}_{\leq \epsilon} 
    + \underbrace{\left| \int_{(c\epsilon)^\alpha}^{+\infty} \frac{1}{\epsilon} e^{-x/\epsilon} \psi(x) \, dx  \right|}_{\lesssim  \exp(-c^{\alpha} \varepsilon^{\alpha-1})  }   \, ,
\end{aligned}
\end{equation}
by Taylor expansion of $\psi$ around zero.
Recalling the form \eqref{eq:Cminusexpressionwithq} of $C^{(0)}_-$, at each time, we will make use of \eqref{eq:to_delta} with $\epsilon = \frac{\kappa}{\beta_--V^{(0)}(t)}$ and a prefactor of $\frac{q(t)}{\beta_--V^{(0)}(t)}$. In particular, we have
\begin{equation}\label{eq:to_delta2}
    \| \frac{\varphi^\kappa}{\kappa}C^{(0)}_-(t,\cdot) - \frac{q(t)}{\beta_--V^{(0)}(t)}\delta_0 \|_{\rm FM(\bar{\R}_+)} \lesssim \kappa \, .
\end{equation}
Recalling \eqref{eq:Bminusexpression}, i.e., that $b_-(t)=\frac{q(t)}{\beta_--V^{(0)}(t)}$, we obtain
\begin{equation}\label{eq:to_delta3}
    \| \frac{\varphi^\kappa}{\kappa}C^{(0)}_-(t,\cdot) - b_-(t)\delta_0 \|_{\rm FM(\bar{\R}_+)} \lesssim \kappa \, .
\end{equation}
Combining \eqref{eq:needbd1}, \eqref{eq:needbd2}, and \eqref{eq:to_delta3}, we obtain 
\begin{equation}\label{eq:capp_muinv}
    \sup_{t\in[0,T]}\norm{c_\pm^{\rm app}-\mu^{\rm inv}_{\pm}}_{{\rm FM}(\overline{\R}_+)} \lesssim \kappa^\alpha\,.
\end{equation}

Finally, an application of the triangle inequality completes the proof of Theorem~\ref{thm:mainthm} in Case~2 with $1-\theta = 2\alpha-1$.

We now turn to \textbf{Case~1}. In this case, the approximate solution was defined in \eqref{eq:mapp} as $\bm{m}^{\rm app}$ (gluing the inner and outer approximate solution was done at the level of $\bm{m}$ rather than $\bm{c}$), and $\bm{c}^{\rm app} := \p_x \bm{m}^{\rm app}$. (This is required to correctly interpret~\eqref{eq:cappdiff0}-\eqref{eq:cappdiff2}.)

Note that
$\bm{m}^{\rm app}\big|_{t=0}= \varphi^\kappa\int_0^x\bm{c}^{\rm bl}\big|_{t=0}dx' + (1-\varphi^\kappa) \bm{m}^{\rm in}$ so that, by \eqref{eq:cbl_initial}, the initial error again satisfies
\begin{equation}
    \norm{\bm{E}_0^{m,\kappa}}_{L^1_x(\R_+)} \lesssim \kappa^\alpha\,.
\end{equation}
Using Lemma~\ref{lem:resbds_mass}, we have that $\bm{m}^{\rm app}(t,x)$ satisfies 
\eqref{eq:bmappass}-\eqref{eq:bmappass3} with
\begin{equation}
    \norm{\bm{E}^{m,\kappa} }_{L^\infty_tL^1_x([0,T]\times\R_+)}
    \lesssim \kappa^\alpha +\kappa^{2\alpha-1}\,.
\end{equation}
By Proposition \ref{pro:errorestcase1}, we then have
\begin{equation}\label{eq:mappmkapp}
    \norm{\bm{m}^\kappa-\bm{m}^{\rm app}}_{L^\infty_tL^1_x([0,T]\times\R_+)}
    \lesssim \kappa^\alpha +\kappa^{2\alpha-1}\,.
\end{equation}

To quantify the error at the $\bm{c}$ level in the ${\rm FM}$ norm, we require the following observation: Suppose that $c \in L^1(\R_+)$ and $m(x) := \int_0^x c(y) \, dy$. Suppose that, additionally, $m \in L^1(\R_+)$. Then
\begin{equation}\label{eq:L1controlofFM}
    \| c(x) \, dx \|_{\rm FM(\overline{\R}_+)} \leq \| m \|_{L^1(\R_+)} \, .
\end{equation}
Indeed, by integration by parts,
\begin{equation}
    \left| \int_{\overline{\R}_+} \phi(x) c(x) \, dx \right| = \left| \phi(0) m(0) - \int \p_x \phi(x) m(x) \, dx \right| \, .
\end{equation}
A consequence of \eqref{eq:mappmkapp}-\eqref{eq:L1controlofFM} is that
\begin{equation}\label{eq:case1FMbd1}
    \sup_{t\in[0,T]}\norm{c_\pm^\kappa-c_\pm^{\rm app}}_{{\rm FM}(\overline{\R}_+)} \lesssim \kappa^\alpha +\kappa^{2\alpha-1}\,.
\end{equation}

We next turn to the difference $\bm{c}^{\rm app}-\bm{c}^{\rm inv}$. We may write
\begin{equation}
\begin{aligned}
    \bm{c}^{\rm app} &= \p_x\bm{m}^{\rm app}
    = \p_x\big(\varphi^\kappa\bm{m}^{\rm bl}+(1-\varphi^\kappa)\bm{m}^{\rm inv} \big)\\
    &= \underbrace{\varphi^\kappa \bm{c}^{\rm bl}+(1-\varphi^\kappa)\bm{c}^{\rm inv}}_{=:\bm{J}^a} +\underbrace{\p_x\varphi^\kappa(\bm{m}^{\rm bl}-\bm{m}^{\rm inv})}_{=:\bm{J}^b}\,.
\end{aligned}
\end{equation}
Following the same arguments as in \eqref{eq:capp_muinv}, we have
\begin{equation}
    \sup_{t\in[0,T]}\norm{J_\pm^a - \mu^{\rm inv}_\pm}_{{\rm FM}(\overline{\R}_+)} \lesssim \kappa^\alpha +\kappa^{2\alpha-1}\,.
\end{equation}
Furthermore, using Lemma \ref{lem:mass_BL}, we may bound
\begin{equation}
    \sup_{t\in[0,T]}\norm{J_\pm^b}_{{\rm FM}(\overline{\R}_+)} \le 
    \sup_{t\in[0,T]}\norm{\p_x\varphi^\kappa(m_\pm^{\rm bl}-m_\pm^{\rm inv})}_{L^1(\R_+)} 
    \lesssim  \kappa^{2\alpha}\,.
\end{equation}
Altogether, we obtain
\begin{equation}\label{eq:case1FMbd2}
    \sup_{t\in[0,T]}\norm{c^{\rm app}_\pm-\mu^{\rm inv}_\pm}_{{\rm FM}(\overline{\R}_+)} \lesssim \kappa^\alpha+\kappa^{2\alpha-1}\,.
\end{equation}

In conclusion, applying the triangle inequality to \eqref{eq:case1FMbd1} and \eqref{eq:case1FMbd2} yields Theorem \ref{thm:mainthm} in Case~1 with $1-\theta = 2\alpha-1$. \end{proof}

\section{The incoming ODE}\label{sec:ODE}

We consider the following nonlinear ODE problem for a bounded function $A : [0,+\infty) \to \R$ depending on parameters $\nu, \gamma > 0$ and $q \in \R$:
\begin{equation}
    \label{eq:Aproblem}
\begin{aligned}
        \gamma \p_X A - \p_X^2 A &= q e^{- \nu X} r(A) \\
        (\gamma - \p_X) A\big|_{X=0} &= 0 \, .
\end{aligned}
\end{equation}
In the context of Section~\ref{sec:approxsols}, we have (at any fixed time)
\begin{equation}
    A(X) = C_+^{(1)}(X) \, , \quad q = C_-^{(0)}\big|_{X=0} \, , \quad \gamma = V^{(0)}+\beta_+ \, , \quad \nu = V^{(0)}-\beta_- \, ;
\end{equation}
i.e., $A$ corresponds to the swimmers reentering the domain from the wall. 
Integrating once in $X$, equation~\eqref{eq:Aproblem} is equivalent to
\begin{equation}
    \gamma A - \p_X A = q \int_0^X e^{- \nu Z} r(A(Z)) \, dZ \, .
\end{equation}
The solutions which remain bounded as $X \to +\infty$ are precisely the solutions to the integral equation\footnote{One can justify this by writing the standard Duhamel formula and observing that there is a unique choice of $A(0)$ (which determines the contribution $e^{\gamma X} A(0)$ to Duhamel's formula) for which solutions will be bounded. This is akin to justification of the forward-backward Duhamel formula used in the proof of the stable manifold theorem, see, e.g.,~\cite[Section 9.2, p. 256-257]{TeschlODEBook}.}
\begin{equation}
    \label{eq:integralequationforA}
    A(X) = q \int_X^{+\infty} e^{\gamma (X-Y)} \int_0^{Y} e^{-\nu Z} r(A(Z)) \, dZ \, dY \, .
\end{equation}
Such solutions moreover decay to a constant as $X \to +\infty$, and we wish to quantify the decay. By the saturation assumption~\eqref{eq:tumblingassumptiononr} on $r(\cdot)$, there exists $R_0 > 0$ such that $|r(A(Z))| \leq R_0$, and therefore
\begin{equation}
    \label{eq:cformula}
    c_\infty(A) := \lim_{X \to +\infty} A(X) = \frac{q}{\gamma} \int_0^{+\infty} e^{-\nu Z} r(A(Z)) \, dZ
\end{equation}
converges, and
\begin{equation}\label{eq:tailest1}
    \left| \int_Y^{+\infty} e^{-\nu Z} r(A(Z)) \, dZ \right| \leq \frac{R_0}{\nu} e^{-\nu Y}
\end{equation}
\begin{equation}\label{eq:tailest2}
    |A(X) - c_\infty(A)| \leq |q| \int_X^{+\infty} e^{\gamma(X-Y)} \frac{R_0}{\nu} e^{-\nu Y}  \, dY \leq |q| \frac{R_0}{\nu} \frac{1}{\gamma+\nu} e^{-\nu X} \, .
\end{equation}
From this, one may differentiate the expression for $A$ to obtain exponential decay estimates on $\p_x^k A$, $k \in \N$. Given the assumption above~\eqref{eq:tumblingassumptiononr} that $r > 0$, such solutions $A$ have the same sign as $q$.

We now apply the implicit function theorem~\cite[Theorem I.1.1]{KielhoferBook} to obtain solutions to equation~\eqref{eq:integralequationforA} for certain values of the parameters. 
For $\nu > 0$, let $\mathbb{X}_\nu$ be the Banach space consisting of pairs $(B,c) \in C_0([0,+\infty)) \times \R$ satisfying
\begin{equation}
    \| (B,c) \|_{\mathbb{X}_\nu} := \| e^{\nu X} B \|_{L^\infty(\R_+)} + |c| < +\infty \, .
\end{equation}
We may identify $\mathbb{X}$ with the Banach space of functions which decay exponentially with rate $e^{-\nu X}$ to a constant value at infinity, so that there is a unique decomposition $A = B+c$. 
By the change of variables
\begin{equation}
    \label{eq:changeofvariablesforbarA}
    \bar{A}(\bar{X}) = A(q,\gamma,\nu)(X) \, , \quad \bar{q} = q/\nu^2 \, , \quad \bar{\gamma} = \gamma/\nu \, , \quad \bar{X} = \nu X \, ,
\end{equation}
it will be sufficient to consider the problem with $\nu=1$: 
\begin{equation}
    \label{eq:barAproblem}
    \begin{aligned}
        \bar{\gamma} \p_{\bar{X}} \bar{A} - \p_{\bar{X}}^2 \bar{A} &= \bar{q} e^{-\bar{X}} r(\bar{A}) \\
        (\bar{\gamma} - \p_{\bar{X}}) \bar{A}\big|_{\bar{X}=0} &= 0 \, .
\end{aligned}
\end{equation}
We abbreviate $\mathbb{X} := \mathbb{X}_1$ and omit bars from our notation when convenient. Consider the nonlinear function
\begin{equation}
\begin{aligned}
    \mathbf{F}(A,q,\gamma) &: \mathbb{X} \times \R \times \R_+ \to \mathbb{X} \, , \\
   \mathbf{F}(A,q,\gamma) &= A - q \int_X^{+\infty} e^{\gamma (X-Y)} \int_0^Y e^{-Z} r(A(Z)) \, dZ \, dY \, ,
\end{aligned}
\end{equation}
or, in $(B,c)$ components,
\begin{equation}
    [\mathbf{F}(A,q,\gamma)]_c = c - \frac{q}{\gamma} \int_0^{+\infty} e^{-Z} r(A(Z)) \, dZ
\end{equation}
\begin{equation}
    [\mathbf{F}(A,q,\gamma)]_B = B + q \int_X^{+\infty} e^{\gamma (X-Y)} \int_{Y}^{+\infty} e^{-Z} r(A(Z)) \, dZ \, dY \, .
\end{equation}
Under the assumptions on $r$, it is not difficult to verify that $\mathbf{F}$ is smooth, and
\begin{equation}
    \label{eq:derivativeofF}
    D_A \mathbf{F}\big|_{(A,q,\gamma)} \tilde{A} = \tilde{A} - q \int_X^{+\infty} e^{\gamma (X-Y)} \int_0^Y e^{-Z} r'(A(Z)) \tilde{A}(Z) \, dZ \, dY \, ,
\end{equation}
or, in components,
\begin{equation}
    \label{eq:derivativeofFccomp}
    [D_A \mathbf{F}(A,q,\gamma) \tilde{A}]_c = \tilde{c} - \frac{q}{\gamma} \int_0^{+\infty} e^{-Z} r'(A(Z)) \tilde{A}(Z) \, dZ
\end{equation}
\begin{equation}
\label{eq:derivativeofFBcomp}
    [D_A \mathbf{F}(A,q,\gamma) \tilde{A}]_B = \tilde{B} + q \int_X^{+\infty} e^{\gamma (X-Y)} \int_{Y}^{+\infty} e^{-Z} r'(A(Z)) \tilde{A}(Z) \, dZ \, dY \, .
\end{equation}

Since, for any $\gamma_0 > 0$, $(A,q,\gamma) = (0,0,\gamma_0)$ is a solution and $D_A \mathbf{F}\big|_{(0,0,\gamma)} = {\rm Id}_{\mathbb{X}}$, the implicit function theorem yields a unique small solution $(q,\gamma) \mapsto A(X;q,\gamma)$ for $(q,\gamma)$ in a neighborhood of $(0,\gamma_0)$. We continue this solution to a maximal open parameter region $E \subset \R_q \times (\R_+)_{\gamma}$ containing an open neighborhood of the $\gamma$-semi-axis by continuation in $q$.\footnote{This is done by successive application of the implicit function theorem to obtain that, for each $\gamma_0$, there exists a unique maximal smooth curve $[q_{\rm min}(\gamma_0),q_{\rm max}(\gamma_0)] \to \mathbb{X} : q \mapsto A(X;q,\gamma_0)$ satisfying $A(X;0,\gamma) \equiv 0$. The implicit function theorem implies that these curves are also smooth jointly in $(q,\gamma_0)$.} 
The function $A$ is smooth in the $\mathbb{X}$ topology as a function of $(q,\gamma) \in E$. Moreover, from the integral formulation~\eqref{eq:integralequationforA}, we may further deduce that $\p_x^k A$, $k \in \N$, is smooth in the $\| e^X \cdot \|_{L^\infty_X(\R_+)}$ topology as a function of $(q,\gamma) \in E$. In conclusion, we have

\begin{proposition}
    \label{pro:Abar}
There exists a maximal $q$-connected\footnote{In this context, we mean that any open set $E' \supsetneq E$ having such a smooth function $\bar{A} : E' \to \mathbb{X}$ satisfying the ODE problem~\eqref{eq:barAproblem} must contain two points $(q_1,\gamma_0) \ne (q_2,\gamma_0)$ for which the line connecting them does not belong to $E'$.} open set $E \subset \R_{\bar{q}} \times (\R_+)_{\bar{\gamma}}$ containing the $\bar{\gamma}$-semi-axis $\{ (0,\bar{\gamma}) \in \R \times \R_+ \}$ and a smooth function
\begin{equation}
    \label{eq:barAlives}
    \bar{A}(\bar{q},\bar{\gamma}) : E \to \mathbb{X}
\end{equation}
satisfying the ODE problem~\eqref{eq:barAproblem}. Hence, $\bar{c}_\infty(\bar{q},\bar{\gamma}) := \lim_{X \to +\infty} \bar{A}(\bar{q},\bar{\gamma})(X)$ is a smooth function. The functions $(\bar{q},\bar{\gamma}) \mapsto \p_X^k A$, $k \in \N$, are smooth in the topology induced by $\| e^X \cdot \|_{L^\infty_X(\R_+)}$.
\end{proposition}

Under certain assumptions, the curve of solutions can be continued indefinitely in $q \geq 0$:
\begin{corollary}\label{cor:rprimeneg}
If $r' \leq 0$, then there exists such $E$ containing $[0,+\infty)_{\bar{q}} \times (\R_+)_{\bar{\gamma}}$.
\end{corollary}
This is relevant for saturated nonlinearities of type~\eqref{eq:example_r} with \emph{decreasing} reaction rate:
\begin{equation}
    r(u) = r_0 - \frac{r_1 u^2}{1+\zeta u^2} \, , \quad r_0, r_1 > 0 \, , \; \frac{r_1}{\zeta} < r_0 \, , \quad u \in \R \, .
\end{equation}

\begin{proof}
Particular to the case $r' \leq 0$, we observe that the Frech{\'e}t derivative is invertible whenever $q \geq 0$, so that, given the obvious \emph{a priori} bound on $\mathbf{F}$,\footnote{The \emph{a priori} bound ensures that the curve of solutions does not escape to infinity at finite~$q$, so that the only obstruction to continuation would be the (lack of) invertibility of $D_A\mathbf{F}$.}
\begin{align}
    \big|[\mathbf{F}(A,q,\gamma)]_c\big| &\leq |c| + \frac{|q|}{\gamma} R_0\,,
    \label{eq:aprioribdFc} \\
    \| e^X [\mathbf{F}(A,q,\gamma)]_B \|_{L^\infty(\R_+)} &\leq \| e^X B \|_{L^\infty(\R_+)} + \frac{|q|}{1+\gamma} R_0 \, , \label{eq:aprioribdFB}
\end{align}
the curve of solutions can be extended to $\{ q \geq 0 \}$.

Suppose that $\tilde{A}$ is a bounded\footnote{We do not impose exponential decay to a constant for this argument.} solution to the linearized problem
\begin{equation}
    \label{eq:tildeAproblem}
\begin{aligned}
        \gamma \p_X \tilde{A} - \p_X^2 \tilde{A} &= q e^{- X} r'(A) \tilde{A}\\
        (\gamma - \p_X) \tilde{A}\big|_{X=0} &= 0 \, .
\end{aligned}
\end{equation}
The essential feature will be that the potential in~\eqref{eq:tildeAproblem} has an advantageous sign.
Integrating once, we have
\begin{equation}
    \label{eq:tildeApositive}
    \p_X \tilde{A} = \gamma \tilde{A} - q \int_0^X e^{- Y} r'(A(Y)) \tilde{A} \, dY \geq \gamma \tilde{A} \, .
\end{equation}
Without loss of generality, we may assume that $\tilde{A}(0) \neq 0$. (If $\tilde{A}(0) = 0$, then the boundary conditions imply that $\p_X \tilde{A}(0) = 0$, so by ODE uniqueness, $\tilde{A}$ must be the zero solution.) Then the differential inequality~\eqref{eq:tildeApositive} implies that $\tilde{A}$ grows exponentially, which contradicts the boundedness. This demonstrates that $D_A \mathbf{F}$ has a trivial kernel in $\mathbb{X}$.

To see the invertibility of $D_A \mathbf{F}$, we observe that $D_A \mathbf{F}$ is a compact perturbation of the identity (hence, Fredholm of index zero) in a suitable Banach space $\mathbb{Y}$ of continuous functions which decay to a constant with the sub-optimal exponential rate $e^{-X/2}$. The above argument yields that $D_A \mathbf{F}$ has trivial kernel in $\mathbb{Y}$, so by the Fredholm theory, $D_A \mathbf{F}$ is boundedly invertible on $\mathbb{Y}$. Since $D_A \mathbf{F} : \mathbb{X} \to \mathbb{X}$, we therefore also have invertibility\footnote{and bounded invertibility, by the open mapping theorem} on the smaller space $\mathbb{X}$.
\end{proof}

\begin{corollary}
    Let $R_1 \geq |r'|$ denote an upper bound on the derivative of $r$.
    Let $\gamma > 0$. Then $E$ contains the open interval of $(\bar{q},\gamma)$ satisfying
    \begin{equation}
        \label{eq:assumptiononqandwhatnot}
    R_1 |\bar{q}| \leq D(\gamma)^{-1} \, ,
    \end{equation}
    where $D(\gamma)$ is defined in~\eqref{eq:Dgammadef}.
\end{corollary}

This is relevant because solutions of the outer equation have an \emph{a priori} bound on the total mass, which bounds $b_- = q/(\beta_- - V^{(0)})$ and $V^{(0)}$ and can therefore keep the parameter values within $E$ for solutions with sufficiently small mass.

\begin{proof}
Let $\mathbb{Y}$ be the Banach space of $L^\infty(\R_+)$ functions $B$ satisfying $\| B \|_{\mathbb{Y}} := \| e^{X} B \|_{L^\infty_X(\R_+)} < +\infty$. We estimate $D_A \mathbf{F}$ componentwise directly from the formulas~\eqref{eq:derivativeofFccomp}-\eqref{eq:derivativeofFBcomp} (cf. \eqref{eq:tailest1}-\eqref{eq:tailest2}): 
    \begin{equation}
        |[D_c \mathbf{F}]_c - 1| \leq \frac{|q|}{\gamma} R_1
    \end{equation}
    \begin{equation}
        \| e^X [D_c \mathbf{F}]_B \|_{\mathbb{Y}} \leq \frac{|q|}{\gamma+1} R_1
    \end{equation}
    \begin{equation}
        \| [D_B \mathbf{F}]_c \|_{\mathbb{Y} \to \R} \leq \frac{|q|}{2\gamma} R_1
    \end{equation}
    \begin{equation}
        \| [D_B \mathbf{F} - I]_B \|_{\mathbb{Y} \to \mathbb{Y}} \leq \frac{|q|}{2(2+\gamma)} R_1 \, .
    \end{equation}
    Hence, bounding the operator norm of $D_A \mathbf{F} - I$, thought of as a block operator, by the maximum of the norms of its columns, we obtain
    \begin{equation}
        \label{eq:Dgammadef}
        \| D_A \mathbf{F} - I\|_{\mathbb{X} \to \mathbb{X}} \leq \underbrace{\max \left( \frac{1}{\gamma} + \frac{1}{\gamma+1}, \frac{1}{2\gamma} + \frac{1}{2(2+\gamma)} \right)}_{=: D(\gamma)} |q| R_1 \, , 
    \end{equation}
    so that when~\eqref{eq:assumptiononqandwhatnot} is satisfied, $\| D_A \mathbf{F} - I\|_{\mathbb{X} \to \mathbb{X}} < 1$, and hence $D_A \mathbf{F}$ is invertible by Neumann series. Given the \emph{a priori} upper bounds~\eqref{eq:aprioribdFc}-\eqref{eq:aprioribdFB} on $\mathbf{F}$, the invertibility of $D_A \mathbf{F}$ is enough to continue the solution in the parameter $q$.
\end{proof}

Define
\begin{equation}
    A(q,\gamma,\nu)(X) := \bar{A}(q/\nu^2,\gamma/\nu)(\nu X) \, , \quad X \geq 0 \, ,
\end{equation}
whenever $\nu > 0$ and $(q/\nu^2,\gamma/\nu) \in E$. Notably, $A$ decays to a constant with the exponential rate $e^{-\nu X}$ is smooth in $(q,\gamma)$ in the $\mathbb{X}_\nu$ topology. However, differentiation in $\nu$ multiplies by $X$.

Let $c_\infty(q,\gamma,\nu) := c_\infty(A(q,\gamma,\nu))$, as given by the formula~\eqref{eq:cformula}. Then
\begin{equation}
    c_\infty(q,\gamma,\nu) = \bar{c}_\infty(q/\nu^2,\gamma/\nu)
\end{equation}
is a smooth function in all variables.


\section{Solvability of the outer problem}
\label{sec:outerproblem}

In this section, we prove Proposition~\ref{pro:inviscidsolutionexists} in several stages:

First, we record a solution theory for the initial boundary-value problem (IBVP) for the transport equation in Sobolev spaces on a half-line based on the method of characteristics. See Section~\ref{sec:estimatestransport}.

Second, we incorporate semilinear terms into the equation, including the ODE $\dot b_- = f({\rm tr} \, c_-,b_-,t)$ on the boundary, by contraction mapping. This requires some structural observations to properly exploit the gain of one time derivative from the smoothing of the boundary ODE. See Section~\ref{sec:semilinearterms}.

Third, we incorporate the quasilinear drift term $\p_x (V c_{\pm})$ by a suitable iteration procedure which exploits \emph{a priori} estimates on $V$ from the conservation of mass. A subtle point is to keep the conditions $V\big|_{x=0} + \beta_+ > 0$ and $V\big|_{x=0} - \beta_- < 0$ in the iteration. See Section~\ref{sec:quasilinearterms}.

Finally, we prove uniqueness and characterize the maximal time of existence $T^*$.

\subsection{Estimates on the transport equation}
    \label{sec:estimatestransport}

In this section, we consider the following IBVP for the transport equation:
\begin{equation}
            \label{eq:IBVP}
            \begin{aligned}
        \partial_t c + v \partial_x c+ \phi c &= h \, ,\quad  t,x > 0 \\
        c\big|_{t=0}&=c^{\rm in} \,,
        \end{aligned} 
\end{equation}
where $v$, $\phi$, and $h$ are each functions of $x$ and $t$.

If the velocity at the boundary is outgoing ($v\big|_{x=0} < 0$), then under suitable background assumptions, it is not difficult to solve~\eqref{eq:IBVP} by the method of characteristics; no additional boundary data is required.
If the velocity at the boundary is incoming ($v\big|_{x=0} > 0$), then we supplement~\eqref{eq:IBVP} with the boundary condition
\begin{align}
    \label{eq:cbdrycondition}
        c\big|_{x=0}&=g \, .
    \end{align}
While the approach to the the inflow problem~\eqref{eq:IBVP}-\eqref{eq:cbdrycondition} via the method of characteristics is, in principle, well known, it is difficult to locate the precise statements we need in the literature. Therefore, we present two statements and the main ingredients needed to prove them. 

Let $T > 0$. Suppose that $v \in  L^{\infty} (0, T; W^{1, \infty} (\R_{+}))$ satisfies
\begin{equation}
    \inf_{t \in [0,T]} v(t,0) > 0 \, .
\end{equation}
Define the \emph{characteristic curves} $X(s;t,x)$ according to
\begin{align}
    \label{ODE}
 &   \frac{d}{ds} X (s; t, x)= v (s, X (s; t, x))\,, \quad X (t; t, x)=x
 \end{align}
 and the \emph{backward exit time}
 \begin{align}
 \label{taub}
  &  \tau_b (t, x) := \sup\{s \ge 0\;:\; X (\tau; t,  x) > 0\,, \tau  \in (t-s, t) \} \, .
\end{align}
Let $x^*(t)$ be the forward characteristic emanating from the origin. Let $c^{\rm in} \in L^1(\R_+)$ and $g \in L^1(0,T)$. Then
\begin{equation}
    \label{eq:ccandidatesol}
    c(x,t) := \begin{cases}
        c^{\rm in} \circ X(0;t,x)\,, & x > x^*(t) \\
        g(t-\tau_b(t,x))\,, & x < x^*(t)
    \end{cases}
\end{equation}
is a candidate solution to~\eqref{eq:IBVP}-\eqref{eq:cbdrycondition}  on $(0,T) \times \R_+$ in the special case $\phi = h = 0$. To ensure that the solution is well defined, it is necessary to verify properties of $X$ (standard) and $\tau_b$. Since $\tau_b(t,x)$ solves
\begin{equation}
    X(t-\tau;t,x) = 0 \, , \quad \text{ or, equivalently,} \quad X(t;t-\tau,0) = x \, ,
\end{equation}
we can study its regularity via (a Lipschitz version of) the implicit function theorem.  Differentiating $X(t-\tau;t,x)$ in $\tau$ at $\tau_b(t,x)$, we obtain
\begin{equation}
    v(t-\tau,X(t-\tau_b;t,x)) = v(t-\tau_b,0) > 0 \, ,
\end{equation}
from which we obtain that $\tau_b$ is locally a Lipschitz function of $(t,x)$. Crucially, it will be necessary to estimate $(\p_x \tau_b)^{-1}$, which appears in the change of variables
\begin{equation}
    \label{eq:changeofvars}
    \int_0^{x^*(t)} |c(t,x)| \, dx = \int_0^{x^*(t)} |g(t-\tau_b(t,x))| \, dx = \int_0^t g(t-\tau) |\p_x \tau_b|^{-1} \, d\tau \, .
\end{equation}
 Here, we used the fact that the curve $t=\tau_b (t, x), t \ge 0$ coincides with $x^{*} (t), t \ge 0$, which can be proved via an implicit function theorem argument  sketched below.
By differentiating $X(t-\tau_b(t,x);t,x) = 0$ in $x$ and rearranging, we obtain
\begin{equation}
    \label{eq:tauxb}
    \partial_x \tau_b  (t, x) =  \frac{1}{v (t-\tau_b, 0)}  (\partial_x X) (t-\tau_b; t, x) \, .
\end{equation}
Since~\eqref{eq:changeofvars} contains $|\p_x \tau_b|^{-1}$, the $v (t-\tau_b, 0)$ coefficient from~\eqref{eq:tauxb} will appear in the numerator. %
With this in hand, it is possible to prove that the candidate solution $c(t,x)$ in~\eqref{eq:ccandidatesol} belongs to $C([0,T];L^1(\R_+))$ with $c\big|_{t=0} = c^{\rm in}$,
\begin{equation}
    c(\cdot,x) \to g(\cdot) \text{ in } L^1(0,T) \text { as } x \to 0^+ \, ,
\end{equation}
and $c(t,x)$ is the unique weak solution to~\eqref{eq:IBVP}-\eqref{eq:cbdrycondition}.

It will be convenient to introduce the notation
\begin{equation}
    \Sigma_T := (0,T) \times \R_+\,.
\end{equation}
When $\phi \in L^\infty_t L^\infty_x(\Sigma_T)$ and $h \in L^\infty_t L^1_x(\Sigma_T)$, for any $0 \le s \le t \le T$, we set
\begin{align}
    \label{Phi}
    \Phi (s, t, x) := \exp \left(-\int_s^t \phi (\tau, X (\tau; t, x)) \, d\tau \right) \, .
\end{align}
Then, letting $H(x)$ denote the Heaviside function, the solution formula is instead
\begin{align}
    \label{IBVPsol}
      & c (t, x)=   H (\tau_b (t, x) -  t) c_1 (t, x) +  H (t- \tau_b (t, x)) c_2 (t, x)\,, \\
   & c_1 (t, x) = c^{\rm in} (X (0; t, x)) \Phi (0, t, x) + \int_0^t h (s, X (s; t, x)) \Phi (s, t, x) \, ds\,, \\
   & c_2 (t, x)=g (t-\tau_b) \Phi (t-\tau_b, t, x) + \int_{t-\tau_b}^t h (s, X (s; t, x)) \Phi (s, t, x) \, ds\,, 
\end{align}
which alternatively may be realized piecewise, as in~\eqref{eq:ccandidatesol}.

To propagate an $x$-derivative, we can differentiate the above formulas directly. Specializing to $\phi=h=0$, we make the following observation in the region $\{ x < x^*(t) \}$:
\begin{equation}
    \p_x c(t,x) = - g'(t-\tau_b) \p_x \tau_b
\end{equation}
\begin{equation}
    \int_0^{x^*(t)} |\p_x c(t,x)| \, dx = \int_0^{x^*(t)} |g'(t-\tau_b) \p_x \tau_b| \, dx = \int_0^t |g'(t-\tau)| \, d\tau \, ,
\end{equation}
since $|\p_x \tau_b| \, dx = d\tau$. To ensure continuity across $x < x^*(t)$ and $x > x^*(t)$ (hence, weak differentiability across $x^*(t)$), the compatibility condition $g(0) = c^{\rm in}(0)$ will be necessary.

Subsequently, estimates on a $t$-derivative can be obtained from the equation~\eqref{eq:IBVP} and the $x$-derivative. 

Following this reasoning, one may obtain solvability in the space $X^{1,1}(T)$ defined in~\eqref{eq:Xkdef}-\eqref{eq:Xknormdef}. To state the bounds, we introduce the functional spaces
\begin{equation}
    \label{eq:y11norm}
Y^{1,1}(T) = C([0,T];W^{1,1}(\R_+)) \, , \quad     \| c \|_{Y^{1,1}(T)} := \| c \|_{L^\infty_t W^{1,1}_x(\Sigma_T)}\,,
\end{equation}
\begin{equation}
    Y^{2,1}(T) = \{ c \in Y^{1,1}(T) : \p_t c, \p_x c \in C([0,T];W^{1,1}(\R_+)) \}\,,
\end{equation}
\begin{equation}
    \| c \|_{Y^{2,1}(T)} := \| c \|_{L^\infty_t W^{2,1}_x(\Sigma_T)} + \| \p_t c \|_{L^\infty_t W^{1,1}_x(\Sigma_T)} \, .
\end{equation}

\begin{lemma}[$X^{1,1}(T)$-solvability of the transport equation]
    \label{lemma 2.3}
Let $T > 0$. Suppose 
    \begin{align}
        \label{eq2.3.1}
      v, \phi \in L^\infty_t W^{1,\infty}_x(\Sigma_T) \, , \quad h \in L^1_t W^{1,1}_x(\Sigma_T)\,,
    \end{align}
    \begin{equation}
        \label{eq:deltaboundinlemma}
        \inf_{t \in [0, T]} v (t, 0) \geq \delta > 0\,,
    \end{equation}
    and
    \begin{equation}
        c^{\rm in} \in W^{1, 1}(\R_{+}) \, , \quad g \in W^{1,1}(0, T)
    \end{equation}
    satisfying the zeroth-order compatibility condition
    \begin{align}
           \label{transportcomp1}
    g (0)= c^{\rm in} (0) \, .
\end{align}
Then, the IBVP \eqref{eq:IBVP}-\eqref{eq:cbdrycondition} has a unique strong solution $c \in Y^{1,1}(T)$, and
\begin{equation}
\label{eq:Y1bds}
\begin{aligned}
  &  \|c\|_{ Y^{1, 1} (T)  }  \leq N_1 \exp \left (N_2 T \||\phi| +|\partial_x v|\|_{L^\infty_t L^\infty_x(\Sigma_T)} \right)\times  \\
  &\qquad \times \left( \|c^{\rm in}\|_{  W^{1, 1}(\R_{+}) } + \|g\|_{ W^{1,1}(0, T)} + \|h\|_{L^1_t W^{1,1}_x(\Sigma_T)} \right)
\end{aligned}
\end{equation}
where 
\begin{align}
    N_1 = N_1 \left( \|v\|_{L^\infty_t L^\infty_x(\Sigma_T)} , \|\phi\|_{  L^{\infty}_t W^{1, \infty}_x(\Sigma_T) } \right) \, .
\end{align} 

If, in addition, 
\begin{align}
    \label{eq2.3.2}
    v, \phi \in C([0,T];W^{1,\infty}_x(\R_+)), \quad h \in C([0,T];L^1(\R_+)) \, ,
\end{align}
then $c \in X^{1, 1} (T)$, and
\begin{equation}
      \| c \|_{X^{1,1}(T)} \leq (1 + \| |v| + |\phi| \|_{L^\infty_t L^\infty_x(\Sigma_T)}) \| c \|_{Y^{1,1}(T)} + \| h \|_{ L^\infty_t L^1_x(\Sigma_T)} \, .
\end{equation}
\end{lemma}

Related calculations (though more difficult, due to the presence of the Boltzmann kernel) can be found in \cite[Proposition 1]{CaoKimLee}. 

For $\p_x^2 c$, copies of $1/v(\tau,0)$ will appear through $\p_x \tau_b$ (see~\eqref{eq:tauxb}) and $\p_x^2 \tau_b$ in $\int_0^{x^*(t)} |\p_x^2 c(t,x)| \, dx$, even after the change of variables; therefore, the bounds may depend on how transversal the characteristics are at $x=0$, measured via the $\delta$-bound~\eqref{eq:deltaboundinlemma}.

The nonlinear problem \eqref{eq3.1}-\eqref{eq3.5} in Section~\ref{sec:semilinearterms} will further require estimates on $\p_t \p_x c$, which can be obtained from differentiating the equation~\eqref{eq:IBVP} in $x$ and using the $\p_x^2 c$ estimates. 
Alternatively, one may differentiate the equation in time,
\begin{equation}
(\p_t + v \p_x + \phi) \p_t c + \p_t v \p_x c + \p_t \phi c = \p_t h \, ,
\end{equation}
and consider the IBVP for $\p_t c$ with the initial condition
\begin{equation}
    \p_t c\big|_{t=0} = - v \p_x c^{\rm in} - \phi c^{\rm in} + h\big|_{t=0}
\end{equation}
 (notice the presence of $h$) and boundary condition
\begin{equation}
    \p_t c\big|_{x=0} = \dot g \, .
\end{equation}
Thus, $\p_t c$ should satisfy estimates of the type in Lemma~\ref{lemma 2.3}.

\begin{lemma}[$X^{2,1}(T)$-solvability of the transport equation]
    \label{lemma 2.8}
Invoke the assumptions \eqref{eq2.3.1}-\eqref{transportcomp1}  of Lemma \ref{lemma 2.3} and assume, additionally, 
    \begin{equation}
            \p_t v, \p_t \phi, \p_x v, \p_x \phi \in L^\infty_t W^{1, \infty}_x (\Sigma_T) \, , \quad   \p_t h, \p_x h \in L^1_t W^{1, 1}_x (\Sigma_T) \, ,
    \end{equation}
    and
    \begin{equation}
          c^{\rm in} \in W^{2,1}(\R_{+})\, , \quad   g \in W^{2,1} (0, T) \,,
    \end{equation}
    satisfying the second-order compatibility condition
\begin{align}
    \label{transportcomp2}
    g' (0)+ v (0, 0)  (c^{\rm in})'(0) + \phi (0, 0) c^{\rm in} (0) = h (0, 0)\,.
\end{align}
    Then the IBVP \eqref{eq:IBVP}-\eqref{eq:cbdrycondition} has a unique strong solution $c \in Y^{2,1}(T)$, and
\begin{equation}\label{eq2.8.1}
  \begin{aligned}
  & \| |\p_x^2 c| + |\p_t \p_x c| \|_{L^\infty_t L^1_x(\Sigma_T)} 
  \leq N_1 \exp \left(N_2 T \||\phi|+|\partial_x v|\|_{L^{\infty}_t L^\infty_x(\Sigma_T)} \right) \times \\
  &\quad  \times \Big( \|c^{\rm in}\|_{  W^{2, 1}(\R_{+}) } + \|g\|_{ W^{2, 1} (0, T) }   
  + \|h\|_{L^1_t W^{2, 1}_x (\Sigma_T)} \\
  &\qquad\qquad  + \| \p_t h \|_{L^1_t W^{1,1}_x(\Sigma_T)}
  + \| h\big|_{t=0} \|_{W^{1,1}(\R_{+})} \Big) \, ,
  \end{aligned} 
\end{equation}
 where 
 \begin{align}
    N_1= N_1 \left( \delta, \|[v, \phi]\|_{L^{\infty}_t W^{2, \infty}_x (\Sigma_T)},  \|\partial_t [v, \phi]\|_{ L^{\infty}_t L^\infty_x (\Sigma_T)} \right) \, .
 \end{align}
 If, in addition,
 \begin{align*}
     \p_t v, \p_t \phi, \p_x v, \p_x \phi  \in  C([0,T];W^{1,\infty}(\R_+)) \, , \quad \partial_t h, \p_x h \in C([0,T];L^1(\R_+)) \,,
 \end{align*}
 then $c  \in X^{2, 1} (T)$, and
 \begin{align*}
     \|\partial_t^2 c\|_{ L^{\infty}_t L^1_x (\Sigma_T) } \le \|\partial_t h\|_{ L^{\infty}_t L^1_x (\Sigma_T) } + \text{RHS of \,} \eqref{eq2.8.1}.
 \end{align*}
 \end{lemma}

Finally, we record the following elementary $L^1$ estimate: 
\begin{lemma}[Eulerian $L^{\infty}_t L^1_x$ estimates]
    \label{lemma 2.6}

Let $T>0$, $v \in L^{\infty}_t W^{1, \infty}_x(\Sigma_T)$, $\phi \in L^\infty_t L^\infty_x(\Sigma_T)$, and 
$h \in 
L^1(\Sigma_T)$.

\noindent $(i)$ Assume that $\sup_{t \in [0,T]} v\big|_{x=0} < 0$.   
    Let $c \in Y^{1, 1} (T)$ be the strong solution to the transport equation \eqref{eq:IBVP}.
    Then, for any $t \in [0, T]$,
    \begin{equation}
    \begin{aligned}
        \label{eq2.6.1}
       & \|c (t, \cdot)\|_{L^1 (\R_{+})} + \int_0^t \abs{vc\big|_{x=0}} \, ds  \\
       &\quad \le \exp (t \||\phi|+|\partial_x v|\|_{ L^{\infty}(0, t;L^\infty(\R_+))} \big(\|c^{\rm in}\|_{L^1 (\R_{+})} +\|h\|_{ L^1(\Sigma_t) }\big)\,.
    \end{aligned}
    \end{equation}

\noindent $(ii)$ Assume that $\inf_{t \in [0,T]} v\big|_{x=0} > 0$. Let $c\in Y^{1,1}(T)$ be the strong solution to the IBVP~\eqref{eq:IBVP}-\eqref{eq:cbdrycondition}.
    Then, for any $t \in [0, T]$,
    \begin{equation}
    \begin{aligned}
        \label{eq2.6.2}
     \|c (t, \cdot)\|_{L^1 (\R_{+})} 
     & \le \exp (t \||\phi|+|\partial_x v|\|_{ L^{\infty}(0,t;L^\infty(\R_+)) }) \times \\
     &\quad \times \big(\|c^{\rm in}\|_{L^1 (\R_{+})} +\|h\|_{ L^1 (\Sigma_t) } + \int_0^t \abs{v\big|_{x=0}\, g} \, ds\big)\,.
    \end{aligned}    
    \end{equation}
\end{lemma}

\begin{proof}
    The proof is standard. We multiply the equation by $\text{sgn} \, (c)$,  integrate over $(0, t) \times \R_{+}$, and use Gronwall's inequality. 
\end{proof}

\subsection{Incorporating semilinearities}
\label{sec:semilinearterms}

The goal of this section is to construct a unique large data solution to the problem
\begin{align}
    \label{eq3.1}
    \p_t c_+ + \p_x ((V+\beta_+) c_+) &=  f(c_+,c_-) \\
        \label{eq3.2}
    \p_t c_- + \p_x ((V-\beta_-) c_-) &= -f(c_+,c_-)
\end{align}
\begin{equation}
    \label{eq3.3}
    \dot b_- = f_b({\rm tr} \,  c_-, b_-,t) \, ,  \quad b_{-} (0) =b_-^{\rm in}
\end{equation}
\begin{equation}
    \label{eq3.5}
    c_+\big|_{x=0} = g_b(b_-,t) \, , \quad
    c_{\pm}\big|_{t=0} = c^{\rm in}_{\pm} \, ,
\end{equation}
where $\beta_{\pm} > 0$, $V(x,t)$ is a given background velocity field, and $f,f_b,g_b$ are (possible) nonlinearities. The constants below may implicitly depend on $\beta_{\pm}$.

The unknowns are $c_{\pm}(x,t)$ and $b_-(t)$.  To measure them, it will be convenient to introduce the spaces
\begin{equation}
    \label{eq:boldXdef}
    \mathbb{Y}^{k,1}(T) := Y^{k,1}(T) \times Y^{k,1}(T) \times W^{k,1}(T) \, , \quad k =1,2 \, ,
\end{equation}
with the sum norm. In~\eqref{eq:boldXdef}, $W^{k,1}(T)$ is the space $W^{k,1}(0,T)$ renormed with 
\begin{equation}
    \label{eq:redefinebnorm}
    \| b \|_{W^{k,1}(T)} := \sum_{j=0}^{k-1} \max (\| \p_t^j b \|_{L^1},\| \p_t^j b \|_{L^\infty}) + \| \p_t^k b \|_{L^1(0,T)} \, , \; \; \; \forall b \in W^{k,1}(0,T) \, ,
\end{equation}
since embeddings of the type $\| b \|_{L^\infty(0,T)} \lesssim \frac{1}{T} \| b \|_{L^1} + \| \p_t b \|_{L^1}$ have a disadvantageous constant for small $T$.

Solutions to the transport equation have an improved trace estimate along non-characteristic hypersurfaces, e.g., solutions to $\p_t c - \p_x c = 0$ satisfy $c(z,0) = c(0,z)$, $z \geq 0$, so that $c^{\rm in} \in W^{1,p}_x(\R_+)$ implies $c\big|_{x=0} \in W^{1,p}_t(\R_+)$. For our construction, we require only the \emph{na{\"i}ve trace estimate} 
\begin{equation}
    \label{eq:naivetrace}
{\rm tr} : L^\infty_t W^{1,1}_x([0,T] \times \R_+) \to L^\infty_t(0,T) 
\end{equation}
obtained by applying the spatial $W^{1,1}_x(\R_+)$ trace inequality $\| c \|_{L^\infty(\R_+)} \leq \| \p_x c \|_{L^1(\R_+)}$, $\forall c \in W^{1,1}(\R_+)$, on a.e. time slice. This is because the ODE for $b_-$ smooths one degree of regularity before $b_-$ enters the inflow condition~\eqref{eq3.5}.

\begin{proposition}[$Y^{1,1}$ solution]
\label{pro:X1psol}
Let $\bar{T} > 0$ and consider $V \in C([0,\bar{T}];W^{2,\infty}(\R_+))$ satisfying
\begin{equation}\label{eq:Vandbetabounds}
    V\big|_{x=0} + \beta_+ \geq \delta \, , \quad V\big|_{x=0} - \beta_- \leq -\delta
\end{equation}
for some $\delta > 0$.
Let
\begin{equation}
    \label{eq:fBUCcondts}
      f \in {\rm BUC}^2(\R^2) \, , \; f_b \in C([0,\bar{T}];{\rm BUC}^1(\R^2))\,,
      \end{equation}
      \begin{equation}
        \label{eq:gBUCconds}
        g_b \in C([0,\bar{T}];{\rm BUC}^2(\R)) \, , \; \p_t g_b \in C([0,\bar{T}];{\rm BUC}^1(\R)) \,.
 \end{equation}
Consider $c^{\rm in}_{\pm} \in W^{1,1}(\R_+)$ and $b_-^{\rm in} \in \R$ satisfying the compatibility condition
\begin{equation}
    c_+^{\rm in}(0) = g_b(b_-^{\rm in},0) \, .
\end{equation}

Then there exists $T \in (0,\bar{T}]$ such that there exists a unique solution
\begin{equation}
(c_+,c_-,b_-) \in \mathbb{Y}^{1,1}(T)
\end{equation}
to the system~\eqref{eq3.1}-\eqref{eq3.5} with initial data $(c_+^{\rm in},c_-^{\rm in},b_-^{\rm in})$. The guaranteed existence time satisfies
\begin{equation}
    T = T(\| c_{\pm}^{\rm in} \|,|b^{\rm in}|, \| f \|, \| f_b \|, \| g_b \|, \| V \|) > 0 \, .
\end{equation}
The solution satisfies a bound
\begin{equation}
    \label{eq:cboundfromprob}
    \| (c_+,c_-,b_-) \|_{\mathbb{Y}^{1,1}(T)} \lesssim 1\,,
\end{equation}
where the implied constant depends on the same quantities as $T$.\footnote{Here and in Proposition~\ref{pro:X2psol}, it is understood that $T$ depends on the norms in a decreasing way, and the implied constant in~\eqref{eq:cboundfromprob} depends on the norms in an increasing way.}
\end{proposition}

\begin{remark}
In Section~\ref{sec:quasilinearterms}, we want to apply Propositions~\ref{pro:X1psol} and~\ref{pro:X2psol} in instances when $f, f_b$ and $g_b$ do \emph{not} belong to ${\rm BUC}$, ${\rm BUC}^1$, or ${\rm BUC}^2$ (in our applications, $f_b$ and $g_b$ may not even be everywhere defined). To circumvent this, one may modify $f, f_b$, and $g_b$ away from the support of $c_{\pm}^{\rm in}$ and $b^{\rm in}_-$ to obtain $\tilde{f},\tilde{f}_b$, and $\tilde{g}_b$ which \emph{do} satisfy the assumptions of Propositions~\ref{pro:X1psol} and~\ref{pro:X2psol}. Then, since $c_{\pm} \in C([0,T];W^{1,1}(\R_+))$ from Lemma~\ref{lemma 2.3} and $b_- \in C([0,T])$, the solution remains in the region where $f = \tilde{f}$, $f_b = \tilde{f}_b$, and $g_b = \tilde{g}_b$ up to some time $\tilde{T} \in (0,T]$. See~\eqref{eq:gbmodification} and \eqref{eq:gbmodification2} for related cut-offs of $g_b$.
\end{remark}

\begin{remark}
    \label{rmk:tildeY11rmk}
    The space $Y^{1,1}(T)$ has continuity in time built into the norms, which provides more information on the solution. However, the fixed point argument also holds in the space of functions $\tilde{Y}^{1,1}(T) := \{ c \in L^\infty_t L^1_x(\Sigma_T) : \p_x c \in L^\infty_t L^1_x(\Sigma_T) \}$ (with the $\tilde{Y}^{1,1}(T)$ norm the same as the $Y^{1,1}(T)$ norm in~\eqref{eq:y11norm}), and with time-continuity replaced by $L^\infty_t$ in~\eqref{eq:fBUCcondts} and \eqref{eq:gBUCconds}. Hence, uniqueness holds in $\tilde{Y}^{1,1}(T)$. This will be utilized once, in the compactness argument below~\eqref{eq:wherecpmlives}.
\end{remark}

\begin{proof}[Proof of Proposition~\ref{pro:X1psol}]
Suppose the hypotheses of the proposition. For now, let $T \in (0,\bar{T}]$ be arbitrary and let $Q_+[T,c^{\rm in}_+;h,g] = Q_+[h,g]$,
\begin{equation}
    Q_+ : Y^{1,1}(T) \times W^{1,1}(T) \to Y^{1,1}(T) \, ,
\end{equation}
be the solution operator to the problem
\begin{align}
    \p_t c_+ + \p_x ((V+\beta_+) c_+) = h(x,t) \\
    c_+\big|_{x=0} = g(t) \, , \quad c_+\big|_{t=0} = c_+^{\rm in} \, ,
\end{align}
as discussed in the previous section. 
Likewise, let $Q_-[T,c^{\rm in}_-;h] = Q_-[h]$,
\begin{equation}
    Q_- : Y^{1,1}(T) \to Y^{1,1}(T)\,,
\end{equation}
denote the solution operator to the problem
\begin{align}
    \p_t c_- + \p_x ((V-\beta_-) c_-) &= h(x,t) \\
    c_-\big|_{t=0} &= c_-^{\rm in} \, .
\end{align}

Let $\Phi_3 : Y^{1,1}(T) \times W^{1,1}(T) \to W^{1,1}(T)$
be the operator
\begin{equation}
    \Phi_3[c_-,b_-](t) = b_-^{\rm in} + \int_0^t f_b({\rm tr} \, c_-,b_-,s)(s) \, ds \, .
\end{equation}
In fact, $\Phi_3$ maps into $W^{1,\infty} \supset W^{1,1}$ by the na{\"i}ve trace estimate~\eqref{eq:naivetrace} on $c_-$, the assumption on $f_b$ (composition of $f$ with $L^\infty(0,T)$ functions yields an $L^\infty(0,T)$ function), and the gain of one derivative from integration.

Let $\Phi[T;c_+,c_-,b_-] = \Phi[c_+,c_-,b_-]$,
\begin{equation}
    \Phi : \mathbb{Y}^{1,1}(T) \to \mathbb{Y}^{1,1}(T) \, ,
\end{equation}
be the operator with components
\begin{align}
    \Phi[c_+,c_-,b_-]_1 &= Q_+[T,c^{\rm in}_+;f(c_+,c_-),g_b(\Phi_3[c_-,b_-],t)]  \\
    \Phi[c_+,c_-,b_-]_2 &= Q_-[T,c^{\rm in}_-;-f(c_+,c_-)]\\
    \Phi[c_+,c_-,b_-]_3 &= \Phi_3[c_-,b_-] \, .
\end{align}
The operators $\Phi_1$ and $\Phi_2$ satisfy the claimed mapping properties under the assumptions on $f$ and $g_b$; more specifically, composition of $f$ with two $Y^{1,1}(T)$ functions yields an $Y^{1,1}(T)$ function, and composition of $g_b$ with a $W^{1,1}(T)$ function yields a $W^{1,1}(T)$ function. Fixed points of $\Phi$ are precisely the desired solutions to~\eqref{eq3.1}-\eqref{eq3.5}.

We demonstrate that $\Phi$ is a contraction in
\begin{equation}
    \mathcal{B}_{M_1,M_2}(T) := \big\{ (c_+,c_-,b_-) \in \mathbb{Y}^{1,1}(T) : \| c_{\pm} \|_{Y^{1,1}(T)} \leq M_1 \, , \; \| b_- \|_{W^{1,1}} \leq M_2 \big\}
\end{equation}
for appropriately chosen $M_1,M_2 > 0$ and $T \leq 1$. We restrict $T$ small enough such that the prefactor in~\eqref{eq:Y1bds} with $v = V$ and $\phi = \p_x V$ satisfies
\begin{equation}
    \label{eq:restrictionofT}
    N_1 \exp \left( 2N_2 T \| \p_x V \|_{L^\infty_t L^\infty_x(\Sigma_{\bar{T}})} \right) \leq 2N_1 \, . 
\end{equation}

\textbf{The map $\Phi$ stabilizes a ball}. First, we demonstrate that $\Phi : \mathcal{B}_{M_1,M_2}(T) \to \mathcal{B}_{M_1,M_2}(T)$ for appropriately chosen $M_1,M_2 > 0$ and $T$. Suppose $(c_+,c_-,b_-) \in \mathcal{B}_{M_1,M_2}(T)$. For $\Phi_3$, we have
\begin{equation}
    \| \Phi[c_+,c_-,b_-]_3(t) \|_{W^{1,1}(T)} \leq |b_-^{\rm in}| + \| \int_0^t f_b({\rm tr} \, c_-,b_-,s)(s) \, ds \|_{W^{1,1}(T)} \, .
\end{equation}
For the $L^\infty$ part of the norm, we have
\begin{equation}
    \label{eq:Phi3firstest}
     \| \Phi[c_+,c_-,b_-]_3(t) \|_{L^\infty(0,T)} \leq |b_-^{\rm in}| + C_{f_b} T \, ,
\end{equation}
where we require the na{\"i}ve trace estimate~\eqref{eq:naivetrace}. For the derivative part of the norm, we have
\begin{equation}\label{eq:Phi3secondest}
\begin{aligned}
    \| \p_t \int_0^t f_b({\rm tr} \, c_-,b_-,s)(s) \, ds \|_{L^1(0,T)} &= \| f_b({\rm tr} \, c_-,b_-,\cdot) \|_{L^1(0,T)}\\ 
    &\leq C_{f_b} \,T \, .
\end{aligned}
\end{equation}
This gain of $T$ in the $b_-$ equation is crucial. For $\Phi_1$, we have
\begin{equation}
    \label{eq:Phi1upperbound}
    \begin{aligned}
        &\| \Phi[c_+,c_-,b_-]_1 \|_{Y^{1,1}(T)} = \| Q_+[f(c_+,c_-),c^{\rm in}_+,g_b(\Phi_3,t)] \|_{Y^{1,1}(T)} \\
        &\quad \leq 2N_1 \left( \| c^{\rm in}_+ \|_{W^{1,1}(\R_+)} + T \| f(c_+,c_-) \|_{Y^{1,1}(T)} + \| g_b(\Phi_3,t) \|_{W^{1,1}(T)} \right) \\
        &\quad \leq 2N_1 \| c^{\rm in}_+ \|_{W^{1,1}(\R_+)} + C_f T + C_{g_b} \, ,
    \end{aligned}
\end{equation}
according to Lemma~\ref{lemma 2.3} and~\eqref{eq:restrictionofT}. For $\Phi_2$, we similarly have
\begin{equation}
    \label{eq:Phi2upperbound}
    \| \Phi[c_+,c_-,b_-]_2 \|_{Y^{1,1}(T)} \leq 2N_1 \| c_-^{\rm in} \|_{W^{1,1}(\R_+)} + C_{f} T \, .
\end{equation}

Fix $M_2 \geq 2 |b^{\rm in}_-| + 1$. Fix $M_1 \geq 3 N_1 \max_\pm \| c_{\pm}^{\rm in} \|_{W^{1,1}(\R_+)}  + 2C_{g_b}+ 1$. Restrict $T$ small enough, depending on $M_1$, to ensure that the sum of~\eqref{eq:Phi3firstest} and~\eqref{eq:Phi3secondest} is bounded above by $2|b_-^{\rm in}| + 1 \leq M_2$. Further restrict $T$ small enough, depending on $M_1,M_2$, to ensure that~\eqref{eq:Phi1upperbound} and~\eqref{eq:Phi2upperbound} are each bounded above by $M_1$. With these restrictions, $\Phi$ stabilizes the ball $\mathcal{B}_{M_1,M_2}(T)$.

\textbf{Contractivity of $\Phi$}. We now verify the contraction property for two inputs $c_{\pm}^{(i)}, b_-^{(i)}$, $i=1,2$, in $\mathcal{B}_{M_1,M_2}(T)$ after possibly shrinking $T$. We abbreviate $\Phi[c_+^{(i)},c_-^{(i)},b_-^{(i)}] = \Phi^{(i)}$, $i=1,2$. The $\Phi_3$ differences satisfy
\begin{equation}
    \label{eq:Phi3diffest}
\begin{aligned}
      &\| \Phi_3^{(1)} - \Phi_3^{(2)} \|_{W^{1,1}(T)} \lesssim \| \int_0^t \big(f_b({\rm tr} \, c_-^{(1)},b_-^{(1)}) - f_b({\rm tr} \, c_-^{(2)},b_-^{(2)})\big) \, ds \|_{L^\infty(0,T)} \\
      &\quad\quad\quad + \| f_b({\rm tr} \, c_-^{(1)},b_-^{(1)}) - f_b({\rm tr} \, c_-^{(2)},b_-^{(2)}) \|_{L^1(0,T)} \\
      &\quad \leq C T ( \| c_-^{(1)} - c_-^{(2)} \|_{Y^{1,1}(T)} + \| b_-^{(1)} - b_-^{(2)} \|_{W^{1,1}(T)} ) \, .
\end{aligned}
\end{equation}
Here, the $C^1$ condition on $f_b$ in~\eqref{eq:fBUCcondts} was used. The $\Phi_2$ differences satisfy (notice the zero initial condition in $Q_-$ below)
\begin{equation}
    \label{eq:Phi2diffest}
\begin{aligned}
    \| \Phi_2^{(1)} - \Phi_2^{(2)} \|_{Y^{1,1}(T)} 
    &= \| Q_-[T,V,0;- f(c_+^{(1)},c_-^{(1)}) + f(c_+^{(2)},c_-^{(2)})] \|_{Y^{1,1}(T)} \\
    &\leq C T (\| c_+^{(1)} - c_+^{(2)} \|_{Y^{1,1}(T)} + \| c_-^{(1)} - c_-^{(2)} \|_{Y^{1,1}(T)} ) \, .
    \end{aligned}
    \end{equation}
Here, the $C^2$ condition on $f$ in~\eqref{eq:fBUCcondts} was used. Finally, the $\Phi_1$ differences satisfy (notice the zero initial condition in $Q_+$ below)
\begin{equation}
    \label{eq:Phi1diffest}
\begin{aligned}
    &\| \Phi_1^{(1)} - \Phi_1^{(2)} \|_{Y^{1,1}(T)} \\
    &\quad = \| Q_+[T,V,0;f(c_+^{(1)},c_-^{(1)}) - f(c_+^{(2)},c_-^{(2)}), g_b(\Phi_3^{(1)},t)- g_b(\Phi_3^{(2)},t)] \|_{Y^{1,1}(T)}
    \\
    &\quad \leq C (T \sum_{\pm} \| c^{(1)}_{\pm} - c^{(2)}_{\pm} \|_{Y^{1,1}(T)} + \| \Phi_3^{(1)} - \Phi_3^{(2)} \|_{W^{1,1}(T)}) \, .
\end{aligned}
\end{equation}
Here, the $C^1$ conditions on $f$ and $\p_t g_b$ in~\eqref{eq:fBUCcondts}-\eqref{eq:gBUCconds} and the $C^2$ condition on $g_b$ in~\eqref{eq:gBUCconds} were used.

Combining~\eqref{eq:Phi3diffest},~\eqref{eq:Phi2diffest}, and~\eqref{eq:Phi1diffest}, we may restrict $T$ small enough to guarantee
\begin{equation}
    \begin{aligned}
        \| \Phi^{(1)} - \Phi^{(2)} \|_{\mathbb{Y}^{1,1}(T)}  \leq \frac{1}{2} \| (c_{\pm}^{(1)}, b_-^{(1)}) - (c_{\pm}^{(2)}, b_-^{(2)}) \|_{\mathbb{Y}^{1,1}(T)} \, .
    \end{aligned}
\end{equation}

In conclusion, the existence and uniqueness of solutions belonging to $\mathcal{B}_{M_1,M_2}(T)$ follows from the contraction mapping principle.

For uniqueness more generally, we observe that for any two solutions, there exist $M_1,M_2$ and $T$ such that the solutions fall into the contractive regime, which lets one propagate equality forward in time.
\end{proof}

\begin{proposition}[$Y^{2,1}$ solution]
\label{pro:X2psol}
In the setting of Proposition~\ref{pro:X1psol}, suppose furthermore that
\begin{gather}
    V \in C([0,\bar{T}];W^{3,\infty}(\R_+)) \, , \quad \p_t V \in C([0,\bar{T}];W^{2,\infty}(\R_+)) \label{eq:Vass2} \\
      f \in {\rm BUC}^3(\R^2) \label{eq:fbuc2} \\
      f_b \in C([0,\bar{T}];{\rm BUC}^2(\R^2)) \, , \quad \p_t f_b \in C([0,\bar{T}];{\rm BUC}^1(\R^2)) \label{eq:fbbuc2} \\
        \p_t^k g_b \in C([0,\bar{T}];{\rm BUC}^{3-k}(\R)) \, , \quad  k = 0,1,2 \, , \label{eq:gbbuc2}
 \end{gather}
 and
$c^{\rm in}_{\pm} \in W^{2,1}(\R_+)$ satisfy the additional compatibility condition
\begin{equation}
    (\p_t c_+)(0,0) = (\p_{b_-} g_b)(b^{\rm in}
    _-,0) \dot b_-(0) + (\p_t g_b)(b_-^{\rm in},0) \, ,
\end{equation}
where $(\p_t c_+)(0,0)$ is interpreted in the sense of equation~\eqref{eq3.1} and $\dot b_-(0)$ is interpreted in the sense of equation~\eqref{eq3.3}.

Then there exists $T > 0$ (possibly shorter than the $T$ in Proposition~\ref{pro:X1psol}) such that
\begin{equation}
(c_+,c_-,b_-) \in \mathbb{Y}^{2,1}(T) \, .
\end{equation}
The guaranteed existence time satisfies
\begin{equation}
    T = T(\| c_{\pm}^{\rm in} \|,|b_-^{\rm in}|, \| f \|, \| f_b \|, \| g_b \|, \| V \|, \delta) > 0 \, ,
\end{equation}
where the norms are those corresponding to the above function spaces. The solution satisfies the bounds
\begin{equation}
    \| (c_+,c_-,b_-) \|_{\mathbb{Y}^{2,1}(T)} \lesssim 1
\end{equation}
where the implied constant depends on the same quantities as $T$.
\end{proposition}

\begin{proof}
Assume the hypotheses of the proposition. Let $c^0_{\pm} \equiv c_\pm^{\rm in}$ and $b_-^0 \equiv b_-^{\rm in}$. Let $(c_{\pm}^n,b_-^n)$, $n \geq 1$, be the Picard iterates corresponding to $\Phi$ from the proof of Proposition~\ref{pro:X1psol}, which guarantees that, for sufficiently small $T$, the iterates have a uniform bound in $\mathbb{Y}^{1,1}(T)$ and satisfy
\begin{equation}
    \label{eq:iamdecreasingfosho}
    \| (c_\pm^{n+1},b_-^{n+1}) - (c_\pm^{n},b_-^{n}) \|_{\mathbb{Y}^{1,1}(T)} \leq \frac{1}{2} \| (c_\pm^{n},b_-^{n}) - (c_\pm^{n-1},b_-^{n-1}) \|_{\mathbb{Y}^{1,1}(T)} \, , \quad n \geq 1 \, .
\end{equation}
Our aim is to demonstrate that the iterates satisfy a uniform bound and are Cauchy in the space $\mathbb{Y}^{2,1}(T)$ for $T$ sufficiently small. The assumptions~\eqref{eq:Vass2} on $V$ will be used throughout in order to apply Lemma~\ref{lemma 2.8} (whose constant may depend on $\delta$).

\textbf{Uniform bound in $\mathbb{Y}^{2,1}(T)$}. Let $n \geq 1$. First, we estimate $b_-^{n+1}$. We have
\begin{equation}
    \| \dot b^{n+1}_{-} \|_{L^\infty(0,T)} = \| f_b({\rm tr} \, c_-^n,b_-^n,t) \|_{L^\infty(0,T)} \lesssim 1 \, .
\end{equation}
By differentiating the ODE for $b_-$, we obtain
\begin{equation}
    \label{eq:differentiationoftheODEforb}
    \ddot b^{n+1}_-(t) = \p_{c_-} f_b \, ({\rm tr} \, \p_t c_-^n) + \p_{b_-} f_b \, \dot b_-^n + \p_t f_b \, ,
\end{equation}
where $\p_{c_-} f_b, \p_{b_-} f_b$ and $\p_t f_b$ are evaluated at $({\rm tr} \, c_-^n, b_-^n, t)$. Hence,
\begin{equation}
    \| \ddot b^{n+1}_{-} \|_{L^1(0,T)} \lesssim T \| c_-^n \|_{Y^{2,1}(T)} + 1 \, .
\end{equation}
In summary,
\begin{equation}
    \label{eq:bnuniformbd}
    \| b^{n+1} \|_{W^{2,1}(T)} \leq T \| c_-^n \|_{Y^{2,1}(T)} + C_1 \, .
\end{equation}

Second, we estimate $c_-^{n+1}$:
\begin{equation}
    \label{eq:cminusuniformbd}
\begin{aligned}
       &\| c_-^{n+1} \|_{Y^{2,1}(T)} = \| Q_-[c_-^{\rm in};-f(c_+^n,c_-^n)] \|_{Y^{2,1}(T)} \\
       &\quad \lesssim \| c_-^{\rm in} \|_{W^{2,1}(\R_+)} + T \| f(c_+^n,c_-^n) \|_{L^\infty_t W^{2,1}_x(\Sigma_T)} + T \| \p_t [f(c_+^n,c_-^n)] \|_{L^\infty_t W^{1,1}_x(\Sigma_T)} \\
       &\quad\quad + \| f(c_+^{\rm in},c_-^{\rm in}) \|_{W^{1,1}(\R_+)} \\
       &\quad \leq CT \sum_{\pm} \| c_{\pm}^n \|_{Y^{2,1}(T)} + C_2 \, .
\end{aligned}
\end{equation}

Third, we estimate $c_+^{n+1}$ according to Lemma~\ref{lemma 2.8}:
\begin{equation}
    \label{eq:cplusuniformbd}
\begin{aligned}
    &\| c_+^{n+1}  \|_{Y^{2,1}(T)} = \| Q_+[c_+^{\rm in};f(c_+^n,c_-^n),c_-^{n-1}),g_b(b^{n+1}_{-},t)] \|_{Y^{2,1}(T)} \\
    &\quad \lesssim \| c_+^{\rm in} \|_{W^{2,1}(\R_+)} + T \| f(c_+^n,c_-^n) \|_{L^\infty_t W^{2,1}_x(\Sigma_T)} + T \| \p_t [f(c_+^n,c_-^n)] \|_{L^\infty_t W^{1,1}_x(\Sigma_T)} \\
       &\quad\quad + \| g_b(b^{n+1}_{-},t) \|_{W^{2,1}(0,T)} + \| f(c_+^{\rm in},c_-^{\rm in}) \|_{W^{1,1}(\R_+)} \\
       &\quad \leq C T \sum_{\pm} \| c_{\pm}^n \|_{Y^{2,1}(T)} + C \| b^{n+1}_- \|_{W^{2,1}(0,T)} + C_2 \, .
\end{aligned}
\end{equation}
Consider the constant $C$ above as a fixed constant.

Let $M_1' = 2C_1$ and $M_2' = 2C_2 + 2C M_1'$. Restrict $T$ sufficiently small such that $C T 2M_2' + C_2 \leq M_2'$ in~\eqref{eq:cminusuniformbd} and~\eqref{eq:cplusuniformbd}. Further restrict $T$ sufficiently small such that $T M_2' + C_1 \leq M_1'$ in~\eqref{eq:bnuniformbd}. Then, by induction, we have
\begin{equation}
    \| b^n_-\|_{W^{2,1}(T)} \leq M_1' \, , \quad \|c_{\pm}^n \|_{Y^{2,1}(T)} \leq M_2' \, , \quad \forall n \geq 2 \, .
\end{equation}

\textbf{Cauchy in $\mathbb{Y}^{2,1}(T)$}. We now prove, for $n \geq 2$, the property
\begin{equation}
    \label{eq:mysequenceiscauchyiny21}
    \begin{aligned}   
    &\| (c_\pm^{n+1},b_-^{n+1}) - (c_\pm^{n},b_-^{n}) \|_{\mathbb{Y}^{2,1}(T)} \\
    &\quad \leq \frac{1}{2}
    \| (c_\pm^{n},b_-^{n}) - (c_\pm^{n-1},b_-^{n-1}) \|_{\mathbb{Y}^{2,1}(T)} + C\| (c_\pm^{n},b_-^{n}) - (c_\pm^{n-1},b_-^{n-1}) \|_{\mathbb{Y}^{1,1}(T)}\,. 
    \end{aligned}
\end{equation}

First, we have
\begin{equation}
    \label{eq:bplushigherderivdiff}
\begin{aligned}
     \| \p_t (b^{n+1}_{-} - b^n_{-}) \|_{L^\infty(0,T)} & =\| f_b({\rm tr} \, c_-^n,b_-^n,t) - f_b({\rm tr} \, c_-^{n-1},b_-^{n-1},t) \|_{L^\infty(0,T)} \\
     &\lesssim \| c_-^n - c_-^{n-1} \|_{Y^{1,1}(T)} + \| b^n_{-} - b^{n-1}_{-} \|_{L^\infty(0,T)} \,,
\end{aligned}
\end{equation}
and, from the identity~\eqref{eq:differentiationoftheODEforb} for $\ddot b^{n+1}$,
\begin{equation}
\begin{aligned}
    &\| \p_t^2 (b^{n+1}_{-} - b^n_{-}) \|_{L^1(0,T)} \lesssim T \| {\rm tr} \, \p_t (c_-^n - c_-^{n-1}) \|_{L^\infty(0,T)} + T \| c_-^n - c_-^{n-1} \|_{Y^{1,1}(T)}\\
    &\qquad + T \| \p_t (b^n_{-} - b^{n-1}_{-}) \|_{L^\infty(0,T)} + T \| b^n_{-} - b^{n-1}_{-} \|_{L^\infty(0,T)} \, .
\end{aligned}
\end{equation}
In addition to the assumptions~\eqref{eq:fbbuc2} on $f_b$, the uniform bounds are necessary when estimating, for example, $\p_{b_-} f_b({\rm tr} \, c^n,b_-^n,t) \dot b_-^n - \p_{b_-} f_b({\rm tr} \, c^{n-1},b_-^{n-1},t) \dot b_-^{n-1}$.

Second, we have
\begin{equation}
\begin{aligned}
       \| c_-^{n+1} - c_-^n \|_{Y^{2,1}(T)} &= \| Q_-[T,V,0;-f(c_+^n,c_-^n)+f(c_+^{n-1},c_-^{n-1})] \|_{Y^{2,1}(T)} \\
       &\lesssim T \| f(c_+^n,c_-^n)-f(c_+^{n-1},c_-^{n-1}) \|_{L^\infty_t W^{2,1}_x(\Sigma_T)} \\
       &\quad\quad + T \| \p_t f(c_+^n,c_-^n)- \p_t f(c_+^{n-1},c_-^{n-1}) \|_{L^\infty_t W^{1,1}_x(\Sigma_T)} \\ 
       & \lesssim T \sum_{\pm} \| c_\pm^{n} - c_{\pm}^{n-1} \|_{Y^{2,1}(T)} \,,
\end{aligned}
\end{equation}
using that $f(c_+^n,c_-^n)\big|_{t=0} - f(c_+^{n-1},c_-^{n-1})\big|_{t=0} = 0$ to cancel the $h\big|_{t=0}$ term in~\eqref{eq2.8.1}. Here, we have also used $C^3$ assumption~\eqref{eq:fbuc2} on $f$ and the uniform bounds on $c_{\pm}^n$.

Third, we have
\begin{equation}
    \label{eq:cplushigherderivdiff}
\begin{aligned}
       &\| c_+^{n+1} - c_+^n \|_{Y^{2,1}(T)}\\
       &\quad= \| Q_+[0;f(c_+^n,c_-^n) - f(c_+^{n-1},c_-^{n-1}),g_b(b^{n+1}_{-},t) - g_b(b^{n}_{-},t)] \|_{Y^{2,1}(T)} \\
       &\quad \lesssim T \sum_{\pm} \| c_\pm^{n} - c_{\pm}^{n-1} \|_{Y^{2,1}(T)} + \| g_b(b^{n+1}_{-},t) - g_b(b^{n}_{-},t) \|_{W^{2,1}(T)} \\
       &\quad \lesssim T \sum_{\pm} \| c_\pm^{n} - c_{\pm}^{n-1} \|_{Y^{2,1}(T)} + \| c_-^n - c_-^{n-1} \|_{Y^{1,1}}\\
       &\quad\qquad + \| b^n_{-} - b^{n-1}_{-} \|_{L^\infty(0,T)} + T \| b_-^n - b_-^{n-1} \|_{W^{2,1}(T)} \, .
\end{aligned}
\end{equation}
This is where the assumptions~\eqref{eq:gbbuc2} on $g_b$ are used.

Finally, to complete the proof of~\eqref{eq:mysequenceiscauchyiny21} and the proposition, we sum~\eqref{eq:iamdecreasingfosho} and~\eqref{eq:bplushigherderivdiff}-\eqref{eq:cplushigherderivdiff} and restrict $T$ to be sufficiently small.
\end{proof}

We now establish two important properties of the above solutions.

\begin{lemma}[Conservation of total mass]
    \label{lem:massconservation}
Let $T>0$ and suppose that $(c_+,c_-,b_-)$ is a solution on $\Sigma_T$ belonging to the class in Proposition~\ref{pro:X1psol}. Suppose additionally that
\begin{equation}
    f(c_+,c_-) = - c_+ r(c_-) + c_- r(c_+) \, ,
\end{equation}
where $r(u) \geq 0$ for $u \geq 0$, and
\begin{equation}\label{eq:formoffb}
f_b(u,b_-,t) = (\beta_- - V\big|_{x=0}) u - (V\big|_{x=0}+\beta_+) g_b(b_-,t) \, .  
\end{equation}
Then 
\begin{equation}
 b_{-} (t) + \int_{\R_{+}} \big(c_{+} (t, x) + c_{-} (t, x)\big) \, dx =   b^{\rm in}_{-} + \int_{\R_{+}} \big( c^{\rm in}_{+}(x)+c^{\rm in}_{-} (x) \big)\, dx =: M
\end{equation}
for all $t\in[0,T]$.
\end{lemma}
\begin{proof}
    This follows from a direct computation.
\end{proof}

\begin{lemma}[Non-negativity]\label{lem:nonnegativity}
  Under the assumptions of Lemma~\ref{lem:massconservation}, suppose furthermore that the initial data is non-negative ($c^{\rm in}_{\pm} \ge 0$, $b_{-}^{\rm in} \geq 0$) and $g_b$ satisfies the property ${\rm sgn} \, g_b(b,t)={\rm sgn} \, b$ for all $b \in \R$ and $t \in [0,\bar{T}]$. Then $b_{-}, c_{\pm} \ge 0$.
\end{lemma}

\begin{proof}
We denote $\phi(x)= - \min(x, 0)$ and record that
\begin{equation}
    \phi' (x)=-1_{x < 0}, \quad  x \phi' (x)=\phi (x).
\end{equation}
Multiplying the  equations for $c_{\pm}$  by $\phi' (c_{\pm})$, respectively, and multiplying the equation for $\dot b_{-}$ by $\phi' (b_{-})$, we obtain
\begin{equation}\label{eq:cpmIDs}
\begin{aligned}
    \partial_t \phi (c_{+}) + (V+\beta_{+}) \partial_x  \phi (c_+) + (\partial_x V) \phi (c_+) &= \underbrace{\phi' (c_{+})  c_{-} r (c_{+})  - \phi (c_{+}) r (c_{-})}_{=: \mathfrak{R}_1}\,, \\
    \partial_t \phi (c_{-}) + (V-\beta_{-}) \partial_x  \phi (c_{-}) + (\partial_x V) \phi (c_{-}) &=\underbrace{\phi' (c_{-})  c_{+} r (c_{-})  - \phi (c_{-}) r (c_{+})}_{=:\mathfrak{R}_2}\,, \\
    \frac{d}{dt} \phi (b_{-}) &=  
    f_b({\rm tr}\,c_-,b_-,t) \phi'(b_-)\,.
\end{aligned}
\end{equation}
We claim that, for a.e. $t$ and $x$, 
\begin{equation}
  \mathfrak{R}_1+\mathfrak{R}_2 \le 0\,. 
\end{equation}
To show this, it suffices to consider three cases: (a) $c_{\pm} (t, x) \ge 0$, (b) $c_{\pm} (t, x) \le 0$, (c) one of $c_\pm(x,t)\ge0$ and the other is non-positive. 
In case (a), $\mathfrak{R}_j=0$ for $j=1,2$. In case (b), we have
\begin{equation}
    \mathfrak{R}_1+\mathfrak{R}_2  = |c_{-}| \,r(c_{+}) - |c_{+}| \,r(c_{-}) + |c_{+}|\, r(c_{-})- |c_{-}| \,r(c_{+})=0.
\end{equation}
Finally, in case (c), taking $c_+(x,t)\ge 0$ without loss of generality, we have
\begin{equation}
    \mathfrak{R}_1 (t, x)=0\,, \quad  \mathfrak{R}_2 (t, x) = -c_{+} r (c_{-}) - \phi (c_{-}) r (c_{+}) \le 0
\end{equation}
since $\phi (c_{-}) \ge 0$. The case $c_-(x,t)\ge 0$ follows similarly.

Integrating the identities \eqref{eq:cpmIDs} for $\phi (c_{\pm})$ over $(0, t) \times \R_{+}$ and the identity for $\phi (b_{-})$ over $(0, t)$, upon integrating by parts in $x$, we obtain
\begin{equation}\label{eq:J1J2defs}
\begin{aligned}
\phi (b_{-} (t)) &+ \int_0^{\infty} \bigg(\phi (c_{+} (t, x))+\phi (c_{-} (t, x))\bigg) \, dx \\
&\qquad\qquad\qquad+ \int_0^t \big(J_1 (s) + J_2 (s)\big) \, ds \le 0\,, \\
J_1 &=\big(\beta_{-}-V\big|_{x=0}\big)   \big[\phi (c_{-}\big|_{x=0})  -  c_-\big|_{x=0}\phi'(b_{-})\big]\,,\\
J_2 &= \big(V\big|_{x=0}+\beta_+\big)  \big[g_b(b_-,t) \phi'(b_{-}) - \phi(g_b(b_-,t))\big]\,.
\end{aligned}
\end{equation}
Here in $J_2$ we have used the form \eqref{eq:formoffb} of $f_b$ and the boundary condition for $c_+$.

We claim that for any $t$,
\begin{equation}
    J_1 (t) \ge 0\,, \quad J_2 (t)=0\,.
\end{equation}
First, if $c_{-}\big|_{x=0} (t) \ge 0$, then, using the property \eqref{eq:Vandbetabounds},
\begin{equation}
    J_1 (t) =  \big(\beta_{-}-V\big|_{x=0}\big) \,c_-\big|_{x=0}1_{b_{-} < 0 } \ge 0\,.
\end{equation}
Otherwise, 
\begin{equation}
     J_1 (t) = \big(\beta_{-}-V\big|_{x=0}\big)\left(\left|c_{-}\big|_{x=0} (t)\right| + c_{-}\big|_{x=0} (t)1_{b_{-} < 0 }\right) \ge 0\,.
\end{equation}
Furthermore, if $b_{-} (t) \ge 0$, then, by assumption, $g_b(b_-,t)\ge 0$, and, since $\phi' (b_{-})=0$, we have
\begin{equation}
    J_2 (t) = 0\,.
\end{equation}
Similarly, if $b_{-} (t) < 0$, then $g_b(b_-,t)< 0$, which gives
\begin{equation}
  J_2(t) = \big(V\big|_{x=0}+\beta_+\big)\left( |g_b(b_-,t)|-|g_b(b_-,t)|\right)=0\,.
\end{equation}
Hence, we may drop the integrals containing $J_1$ and $J_2$ from the left-hand side of \eqref{eq:J1J2defs}. Since $\phi$ is non-negative, we conclude that $\phi(c_{\pm}) \equiv 0$ and $\phi(b_-) \equiv 0$, which implies that $c_{\pm}, b_- \ge 0$.
\end{proof}

\subsection{Incorporating nonlinear advection}
\label{sec:quasilinearterms}

Finally, we specialize to the system
\begin{align}
\label{eq:inviscidsystem3}
    \p_t c_+ + \p_x ((V+\beta_+) c_+) &=  f(c_+,c_-)  \\
    \p_t c_- + \p_x ((V-\beta_-) c_-) &= - f(c_+,c_-)
\end{align} 
\begin{gather}
    \dot b_- = \underbrace{-(V\big|_{x=0} + \beta_+) g_b(b_-,V\big|_{x=0}) + (\beta_- - V\big|_{x=0}) c_-\big|_{x=0}}_{=: f_b(c_-,b_-,V)}  \label{eq:bminuseqncase3} \\
    c_+\big|_{x=0} = g_b(b_-,V\big|_{x=0}) \label{eq:cinv_plus_bvalcase2_sec5} \\
    b_{-}(0) = b_-^{\rm in} \geq 0\,, \quad
    c_{\pm}\big|_{t=0} = c_{\pm}^{\rm in} \in W^{2,1}(\R_+) \, , \;\; c_{\pm}^{\rm in} \geq 0
\end{gather}
with the constitutive equation
\begin{align}
    \label{eq:Vlaw3}
     V[\rho,b_-](x) = \int_{\R_+} K(x,y) \rho(y) \, dy + K(x,0) b_{-} \, ,
\end{align}
where we recall $\rho=c_++c_-$.
This system is more general than the one for which we prove inviscid limits, as it combines Case~1 and Case~2.
Here, $f, K$ are assumed to satisfy the assumptions in Section~\ref{sec:introduction}, namely,~\eqref{eq:smoothnessofK}-\eqref{eq:decayonkernel} and~\eqref{eq:fintermsofr}-\eqref{eq:tumblingassumptiononr}, and $g_b : [0,+\infty) \times \R \to [0,+\infty)$
is a non-negative function satisfying 
\begin{equation}
\label{eq:gconditionsforquasilinearity}
 g_b(0,V) \equiv 0\,, \quad   
 |\p_{b_-}^{j+1} \p_V^k g_b| \leq C \, , \quad 
 |\p_V^k g_b| \leq C|b_-| \,,   \quad |j|, |k| \leq 4 \, ,
\end{equation}
where the last condition is a consequence of the first two.
Here $g_b(b_-,V)$ plays the role of $g_b(b_-,t)$ in the previous section except that the $t$ dependence is now through $V\big|_{x=0}$.\footnote{We have in mind that $g_b$ is an extension of a restriction of the function $c_\infty$ in Section~\ref{sec:ODE}, but we can avoid discussing extensions until we prove Proposition~\ref{pro:inviscidsolutionexists} at the end.} For brevity, we refer to the function on the right-hand side of the ODE~\eqref{eq:bminuseqncase3} as $f_b(c_-,b_-,V)$. 

\begin{proposition}[Existence]
    \label{pro:existence}
    Invoke the above assumptions on $f$, $K$, and $g_b$. Define $V^{\rm in} := V[\rho^{\rm in},b_-^{\rm in}]$. Assume the compatibility conditions
    \begin{equation}
        \label{zeroordercomp}
        g_b(b_-^{\rm in},V^{\rm in}(0)) = c_+^{\rm in}(0)\,,
    \end{equation}
    \begin{equation}
        \label{firstordercomp}
        (\p_t c_+)(0,0) = (\p_{b_-} g_b)(b_-^{\rm in},V^{\rm in}(0)) \dot b_-(0) + (\p_V g_b) (b_-^{\rm in},V^{\rm in}(0)) (\p_t V)(0,0) \, ,
    \end{equation}
    where $\p_t c_+$, $\dot b_-$, and $\p_t V$ are interpreted in the sense of the initial conditions via the equation. Suppose that there exists $\delta > 0$ such that
    \begin{equation}
        \label{eq:initialguysatisfies2deltabound}
        V^{\rm in}\big|_{x=0} + \beta_+ \geq 2\delta \, , \quad V^{\rm in}\big|_{x=0} - \beta_- \leq - 2\delta \, .
    \end{equation}
    Then there exists $T > 0$, depending on $|b^{\rm in}|,\| c_{\pm}^{\rm in} \|_{W^{2,1}(\R_+)},\delta,f,K$, and $g_b$ through~\eqref{eq:gconditionsforquasilinearity}, such that there exists a solution
    \begin{equation}
        (c_+,c_-,b_-) \in \mathbb{Y}^{2,1}(T)
    \end{equation}
    to the system~\eqref{eq:inviscidsystem3}-\eqref{eq:Vlaw3}, and
    \begin{equation}
    \label{eq:lowerboundstobepropagated}
        V[\rho,b_-]\big|_{x=0} + \beta_+ \geq \delta \, , \quad V[\rho,b_-]\big|_{x=0} - \beta_- \leq - \delta \, .
    \end{equation}
\end{proposition}

In the following, we allow the implicit constants to depend on $f,K$, and $g_b$ through~\eqref{eq:gconditionsforquasilinearity}.

\begin{lemma}[\emph{A priori} estimates]
    \label{lem:aprioriestsunderdeltacondition}
    Let $T > 0$ and $V_0 \in L^\infty_t W^{3,\infty}_x(\Sigma_T)$.
    Suppose
    \begin{equation}
        V_0\big|_{x=0} + \beta_+ \geq \delta > 0 \, , \quad V_0\big|_{x=0} - \beta_- \leq -\delta \, .
    \end{equation}
    
    Suppose that $(c_+,c_-,b_-)$ is a $\mathbb{Y}^{1,1}(T)$ solution to~\eqref{eq:inviscidsystem3}-\eqref{eq:cinv_plus_bvalcase2_sec5} with $V=V_0$ and initial conditions satisfying the zeroth compatibility condition \eqref{zeroordercomp}. Let
    \begin{equation}
    \label{eq:C0def}
           C_0 := M + \sum_{\pm} \| c_\pm^{\rm in} \|_{L^\infty(\R_+)}  \, , \quad
        C_0' = C_0 + \sum_{\pm} \| c_\pm^{\rm in}\|_{W^{1,1}(\R_+)}\, .
    \end{equation}

    Then the following assertions hold.

 $(i)$ One has
    \begin{equation}
        \label{eq:maximumprincipleest}
        \sum_{\pm} \| c_\pm \|_{L^\infty_t L^\infty_x(\Sigma_{T})} \lesssim C_0 \, .
    \end{equation}
The implicit constant depends only on $T$ and $\| V_0 \|_{L^\infty_t W^{2,\infty}_x(\Sigma_T)}$.
 
 $(ii)$   Let $V = V[\rho,b]$. Then
    \begin{align}
        \label{eq:onetimederivativeest}
     \| \p_t V \|_{L^\infty_t W^{j, \infty}_x(\Sigma_T)} &\lesssim_j M, \\
     \label{eq:twotimederivativeest}
     \| \p_t^2  V \|_{L^\infty_t W^{j, \infty}_x (\Sigma_T)} &\lesssim_j  C_0 \, ,
    \end{align}
 for all $j \in \N$. The implicit constants depend only on $j$, $T$, and $\| V_0 \|_{L^\infty_t W^{1,\infty}_x(\Sigma_T)}$ (for~\eqref{eq:onetimederivativeest}) and $\| V_0 \|_{ W^{1,\infty}_{t,x}(\Sigma_T) }$ (for~\eqref{eq:twotimederivativeest}).
    
$(iii)$ Suppose that $\p_t V_0 \in L^\infty_t W^{2,\infty}_x(\Sigma_T)$. Then
    \begin{equation} \label{eq:y11aprioriest}
        \sum_{\pm} \| c_{\pm} \|_{Y^{1,1}(T)} \lesssim C_0' \, .
    \end{equation}
    The implicit constants may depend on the previous quantities (including $C_0$) and $\| \p_t V_0 \|_{L^\infty_t W^{2,\infty}_x(\Sigma_T)}$.

 $(iv)$ Suppose that $\p_t^2 V_0 \in L^\infty_t W^{2,\infty}_x(\Sigma_T)$ and that $(c_+,c_-,b_-) \in \mathbb{Y}^{2,1}(T)$ with initial conditions also satisfying the second-order compatibility condition \eqref{firstordercomp}.
    Then
    \begin{equation}  \label{eq:y21aprioriest}
        \sum_{\pm} \| c_{\pm} \|_{Y^{2,1}(T)} \lesssim C_0' + \| c_\pm^{\rm in} \|_{W^{2,1}(\R_+)} \, .
    \end{equation}
    The implicit constant may depend on the previous quantities, $\| \p_t^2 V_0 \|_{L^\infty_t W^{2,\infty}_x(\Sigma_T)}$, and $\delta$.
\end{lemma}

\begin{proof}[Proof of Lemma~\ref{lem:aprioriestsunderdeltacondition}]
\textbf{Maximum principle estimate~\eqref{eq:maximumprincipleest}.} 
First, by the non-negativity and conservation of total mass in Lemmas \ref{lem:massconservation}-\ref{lem:nonnegativity}, $\|b_{-}\|_{L^{\infty} (0, T)} \le M$. Next, recall the solution formula~\eqref{IBVPsol} from the method of characteristics. By estimating the ODE along characteristics, we have
\begin{equation}
\begin{aligned}
    \| c_- \|_{L^\infty_t L^\infty_x(\Sigma_t)} &\leq \exp \left( t\| \p_x V_0 \|_{L^\infty (\Sigma_t)} \right) \left( \| c_-^{\rm in} \|_{L^\infty(\R_+)} +\int_0^t \| f \|_{L^\infty(\R_+)} \, ds \right)\\
    \| c_+ \|_{L^\infty_t L^\infty_x(\Sigma_t)} &\leq \exp \left( t\| \p_x V_0 \|_{L^\infty (\Sigma_t)} \right)\times\\
    &\quad \times\left( \| c_+^{\rm in} \|_{L^\infty(\R_+)} + \| g_b \|_{L^\infty(0,t)} + \int_0^t \| f \|_{L^\infty(\R_+)} \, ds \right) \, .
\end{aligned}
\end{equation}
Since
\begin{equation}
    \| f(c_+,c_-) \|_{L^\infty(\R_+)} \lesssim \| c_+ \|_{L^\infty(\R_+)} + \| c_- \|_{L^\infty(\R_+)} \, , \quad \abs{g_b(b_-,V_0\big|_{x=0})} \lesssim M \, ,
\end{equation}
we have that $\sum_{\pm} \| c_\pm \|_{L^\infty(\Sigma_t)}$ satisfies an integral inequality amenable to Gr{\"o}nwall's lemma, which yields
\begin{equation}
    \sum_{\pm} \| c_\pm \|_{L^\infty(\Sigma_t)} + \| b_- \|_{L^\infty(0,t)} \leq C e^{Ct} \left( \sum_{\pm} \| c_\pm^{\rm in} \|_{L^\infty(\R_+)} + M \right) \, .
\end{equation}

\textbf{Estimate~\eqref{eq:onetimederivativeest} on $\p_t V$.} With this bound in hand, we can estimate
\begin{equation}
    \label{eq:ptVexpression}
    \p_t V = \int_{\R_+} K(x,y) \p_t \rho\,dy + K(x,0) \dot b_- \, .
\end{equation}
We recall that 
\begin{equation}
    \p_t \rho = -\p_y[(V_0+\beta_{+})c_{+} + (V_0 -\beta_{-})c_{-}] \, ,
\end{equation}
since the two $f$ terms cancel. Integrating by parts and using \eqref{eq:bminuseqncase3}-\eqref{eq:cinv_plus_bvalcase2_sec5}, we obtain
\begin{equation}
\begin{aligned}
    & \p_t V (t, x) = \int_{\R_+} \p_y K(x,y) [(V_0+\beta_{+})c_{+}+(V_0-\beta_{-})c_{-}] \, dy \\
    &\qquad + K(x,0) \big[ \underbrace{\dot b_- + (V_0|_{x=0}+\beta_{+})c_{+}|_{x=0}+ (V_0|_{x=0}-\beta_{-})c_{-}|_{x=0} }_{=0}  \big] \, \\
    \label{eq:ptV}
    &\quad=  \int_{\R_+} \p_y K(x,y)  [(V_0+\beta_{+})c_{+}+(V_0-\beta_{-})c_{-}] \, dy.
\end{aligned}
\end{equation}
This last identity combined with mass conservation and positivity of $c_{\pm}$ yield \eqref{eq:onetimederivativeest}. To derive the estimate~\eqref{eq:twotimederivativeest} on $\partial_t^2 V$, we use a similar argument: differentiate \eqref{eq:ptV} in $t$, use the equations \eqref{eq:bminuseqncase3}-\eqref{eq:cinv_plus_bvalcase2_sec5}, integrate by parts in $x$, estimate the resulting integral term via the $L^{\infty}_t L^1_x$-norms of $c_{\pm}$, and the remaining boundary terms via the $L^{\infty}$-estimate \eqref{eq:maximumprincipleest}.

\textbf{Propagation of $Y^{1,1}$.} We calculate
\begin{equation} \label{eq:ptgb}
    \p_t \big[ g_b(b_-(t),V_0\big|_{x=0}(t)) \big] = \p_{b_-} g_b \dot b_- + \p_V g_b \p_t V_0\big|_{x=0} \, .
\end{equation}
This yields
\begin{equation}
    \label{eq:linftyongb}
    \| g_b(b_-(t),V_0\big|_{x=0}(t)) \|_{W^{1,\infty}(0,T) } \lesssim M + \| c_- \|_{L^\infty_t L^\infty_x([0,T] \times \R_+)}\, .
\end{equation}
With this in hand, we can apply the linear estimate of Lemma~\ref{lemma 2.3} to propagate control in $Y^{1,1}$. The key point is that, due to the assumptions on $V_0$ and the saturation of the tumbling,\footnote{From a certain perspective, these assumptions ``linearize" the equation, at least in terms of the possible growth.} the norm cannot explode in finite time. Indeed, from Lemma~\ref{lemma 2.3} and the $L^{\infty}_{t, x}$-bound \eqref{eq:maximumprincipleest} on $c_{\pm}$, we have
\begin{equation}
\begin{aligned}
        &\|c_- \|_{ Y^{1,1}(t)} \lesssim \left\|c_-^{\rm in} \right\|_{W^{1,1}(\R_{+})} + \int_0^t \| f\|_{ W^{1,1}(\R_+)}(s) \, ds \\
        &\quad\lesssim \left\|c_-^{\rm in} \right\|_{W^{1,1}(\R_{+})} + \sum_{\pm}  \underbrace{ \| c_\pm \|_{L^\infty_s L^\infty_x([0,t] \times \R_+)} }_{ \lesssim C_0 }   \int_0^t \sum_{\pm} \| c_{\pm} \|_{ Y^{1,1}(s)} \,ds   
\end{aligned}
\end{equation}
and
\begin{align}
\label{eq:cplusboundsendofpaper}
    &\| c_+ \|_{ Y^{1,1}(t)} \lesssim \left\|c_+^{\rm in} \right\|_{W^{1,1}(\R_{+})} + \underbrace{\| g_b \|_{ W^{1,1} (0,t)}}_{\text{bdd. by } \eqref{eq:linftyongb}} + C_0 \int_0^t \sum_{\pm} \| c_{\pm} \|_{ Y^{1,1}(s)} \, ds \, ,
\end{align}
where the implicit constants may depend on $T$ and norms of $V_0$. Gr{\"o}nwall's inequality yields $\| (c_+,c_-,b_-) \|_{\mathbb{Y}^{1,1}(T)} \lesssim C_1$.

\textbf{Propagation of $Y^{2,1}$.} Finally, we propagate the $Y^{2,1}$ regularity. This is done as in the $Y^{1,1}$ propagation, though now we appeal to the linear estimates in Lemma~\ref{lemma 2.8}, which require the $\delta$-transversality bound on $V_0$ (see \eqref{eq:Vandbetabounds}). Below is a sketch of the argument. All the implicit constants may depend on $T$, $\|\partial_t V_0\|_{ L^{\infty}_t  W^{2,\infty}_x(\Sigma_T)}$, $\|V_0\|_{ L^{\infty}_t  W^{3,\infty}_x(\Sigma_T)}$,   $\|c_{\pm}\|_{ W^{2, 1} (\R_{+}) }$, and $b_-^{\rm in}$. First, we estimate $b_{-}$. By the equation \eqref{eq:bminuseqncase3}, the assumptions on $g_b$ \eqref{eq:gconditionsforquasilinearity},  and the control of the $Y^{1, 1}$-norm \eqref{eq:y11aprioriest}, 
\begin{equation}
    \|\dot b_{-}\|_{L^{\infty} (0, T) } \lesssim  \|b_{-}\|_{L^{\infty} (0, T) } + \|c_{-}\big|_{x=0}\|_{L^{\infty} (0, T)} \lesssim 1\,.
\end{equation}
Furthermore, by differentiating \eqref{eq:bminuseqncase3} and using \eqref{eq:linftyongb} and \eqref{eq:naivetrace}, for any $t \in [0, T]$, we obtain
\begin{equation}
\begin{aligned}
  \|\ddot b_{-}\|_{L^{1} (0, t) } &\lesssim (1+ \|V_0\big|_{x=0}\|_{W^{1, \infty} (0, T) }) \times\\
  &\quad\times \big(\|g_b (b_{-}, V_0\big|_{x=0})\|_{ W^{1, 1} (0, T) } +  \|c_{-}\big|_{x=0}\|_{ W^{1, 1} (0, t) }\big) \\
  & \lesssim 1+ \int_0^t \|c_{-}\|_{Y^{1, 1} (s)} \, ds\,.
\end{aligned}
\end{equation}
Next, by  differentiating the identity \eqref{eq:ptgb} and using \eqref{eq:gconditionsforquasilinearity} and the $L^{\infty}$-bound of $\dot b_{-}$, we get 
\begin{equation}
    \norm{\frac{d^2}{dt^2} g_b (b_{-}, V_0\big|_{x=0})}_{L^1 (0, T) } \lesssim 1+  \int_0^t |\ddot{b}_{-}|  \, ds \, .
\end{equation}
Finally, by using \eqref{eq2.8.1} in Lemma \ref{lemma 2.8}, and estimates of $\|c_{\pm}\|_{X}, X=L^{\infty} (\Sigma_T), Y^{1, 1} (T)$, we obtain, for $t \in [0, T]$,
\begin{equation}
\begin{aligned}
   & \sum_{\pm} \||\partial_x^2 c_{\pm}|+|\partial_t \partial_x c_{\pm}|\|_{ L^{\infty}_t L^1_x (\Sigma_t) } \\
   &\quad \lesssim  \sum_{\pm} \big(\|c^{\rm in}_{\pm}\| + \|f\|_{ L^1 (0, t; W^{2, 1}_x (\R_{+})) } + \|\partial_t f\|_{ L^1 (0, t; W^{1, 1}_x (\R_{+})) }\big) \\
   &\quad\qquad + \|g_b (b_{-}, V\big|_{x=0})\|_{W^{2, 1} (0, t)} \\
   &\quad   \lesssim 1+ \sum_{\pm} \int_0^t \|c_{\pm}\|_{ Y^{1, 1} (s) } \, ds\,.
\end{aligned}
\end{equation}
An application of Gronwall's inequality yields the desired estimate \eqref{eq:y21aprioriest}.
\end{proof}

\begin{proof}[Proof of Proposition~\ref{pro:existence} (Local existence)]
First, we define $\bm{c}^0 = \bm{c}^{\rm in}$,  $b^0_{-}=b^{\rm in}_{-}$, $V^0=V[\rho^{\rm in}, b^{\rm in}_-]$. For $n \ge 1$, we set $\bm{c}^n = (c^n_{+}, c^n_{-}, b^n_{-})$ to be the unique solution of the system
\begin{align}
    \label{eq4.1}
   & \p_t c^n_{+} + \p_x ((V^{n-1}+\beta_+) c^n_+) =  f(c^{n}_+,c^{n}_-)\,, \quad c_{+}^n\big|_{t=0}=c^{\rm in}_{+}\,, \\
     \label{eq4.2}
  &  \p_t c^n_{-} + \p_x ((V^{n-1}-\beta_-) c^n_-) = - f(c^{n}_+,c^{n}_-) \,, \quad  c_{-}^n\big|_{t=0}=c^{\rm in}_{-}\,, \\
     \label{eq4.3}
  & V^{n-1} (t, x) = V [\rho^{n-1}, b^{n-1}_{-}]\,, \quad \rho^n:=c^n_{+}+c^n_{-}\,, \\
     \label{eq4.4}
  &  \dot b^n_- = f_b({\rm tr} \, c_-^n,b^n_-,V^{n-1}\big|_{x=0})\,, \quad b_-^n (0) =b^{\rm in}_{-}\,,\\
    \label{eq4.5}
 &   c^n_+\big|_{x=0} = g_b(b^n_-,V^{n-1}\big|_{x=0})
\end{align}
in the class of functions $c^n_{\pm} \in Y^{2, 1} (T)$, $b^n_{-} \in W^{2, 1} (T)$, for all $T < T_n$, where $T_n \le T_{n-1}$ is the maximal time of existence of the $n$th iterate. Lemmas~\ref{lem:nonnegativity} and~\ref{lem:massconservation} guarantee that the iterates are non-negative and conserve the total mass $M$.

From mass conservation, we obtain 
\begin{equation}
    \label{eq:VnboundedbyM}
    \| V^n \|_{L^\infty_t W^{3,\infty}_x(\Sigma_{T_n})} + \|\partial_t V^n \|_{L^\infty_t W^{2,\infty}_x(\Sigma_{T_n})} \lesssim M\, ,
\end{equation}
under the assumptions on the kernel and due to \eqref{eq:onetimederivativeest}. This single estimate begets a string of estimates in Lemma~\ref{lem:aprioriestsunderdeltacondition} provided that the velocities $V\big|_{x=0} \pm \beta_{\pm}$ on the boundary retain the correct signs. An important point will therefore be to propagate these signs for some time. 
\begin{quote}
    \textbf{Claim.} There exist $\bar{T} \in (0,1]$ and $\bar{C} > 0$, depending on norms of the data and $\delta$, such that the iterates $(c_+^n,c_-^n,b_-^n)$ are defined in $\mathbb{Y}^{2,1}(\bar{T})$ and satisfy~\eqref{eq:lowerboundstobepropagated} with $V=V^n$, and
    \begin{equation}
        \| (c_+^n,c_-^n,b_-^n) \|_{\mathbb{Y}^{2,1}(\bar T)} \leq \bar{C} \, .
    \end{equation}
\end{quote}

  For $n=0$, $(c_+^0,c_-^0,b_-^0)$ is constant-in-time and therefore exists globally-in-time, conserves the mass, and satisfies the $\delta$-bound~\eqref{eq:lowerboundstobepropagated} with $V=V^0$. 

  First, we present a preliminary computation, which will determine $\bar{T}$. We consider the $n$th and $(n+1)$st solutions, with maximal times of existence $T_n \geq T_{n+1}$. Up to its maximal time of existence, the $n$th solution satisfies the estimate~\eqref{eq:onetimederivativeest} on $\p_t V^n$:
  \begin{equation}
    \label{eq:timederivvncc0}
      \| \p_t V^n \|_{L^\infty_t L^\infty_x(\Sigma_{T_n})} \leq CC_0 \, ,
  \end{equation}
  where $C_0$ is as in~\eqref{eq:C0def} and $C$ depends on $M$ and the $L^\infty_t W^{3,\infty}_x(\Sigma_{T_n})$ bound~\eqref{eq:VnboundedbyM} on $V^n$. This derivative estimate was granted by Lemma~\ref{lem:aprioriestsunderdeltacondition}. Since $V[\rho^{\rm in},b_-^{\rm in}]$ satisfies the $2\delta$ condition~\eqref{eq:initialguysatisfies2deltabound},~\eqref{eq:timederivvncc0} grants
  \begin{equation} 
      V^n\big|_{x=0} + \beta_+ \geq \delta \, , \quad V^n\big|_{x=0} - \beta_- \leq - \delta \quad \text{ on } [0,\delta (CC_0)^{-1}] \, .
  \end{equation}
  This demonstrates that $T_{n+1} \geq \min(T_n,\delta( CC_0)^{-1},1)$. Since $T_0 = +\infty$, we have
  \begin{equation}
      T_n \geq \min(\delta (CC_0)^{-1},1) =: \bar{T} \, .
  \end{equation}
  From here, an induction based on the \emph{a priori} estimates in Lemma~\ref{lem:aprioriestsunderdeltacondition} is enough to prove the \textbf{Claim}. (In particular, we first prove the claim with
  $\mathbb{Y}^{1,1}$ instead of $\mathbb{Y}^{2,1}$ by employing that \eqref{eq:timederivvncc0} is independent of time. We then use Lemma~\ref{lem:aprioriestsunderdeltacondition} to establish 
  \begin{align*}
      \|\partial_t^2 V^n \|_{L^\infty_t W^{2,\infty}_x(\Sigma_{\bar T})} \lesssim 1 \, ,
  \end{align*}
  where the constant depends only on $T$, $M$, and $\| c^{\rm in}_{\pm} \|_{L^\infty(\R_+)}$.)
  
Finally, the $\mathbb{Y}^{2,1}$ estimate and the equations~\eqref{eq4.1}-\eqref{eq4.2} themselves for $c_{\pm}$ grant uniform $X^{2,1}$ estimates:
\begin{equation}
    \label{eq:X21est}
    \sup_{n \geq 1} \sum_{\pm} \| c_{\pm}^n \|_{X^{2,1}(\bar{T})} < +\infty \, .
\end{equation}

We now use compactness to pass to the limit. By the \emph{a priori} estimates and compact embeddings guaranteed by the $\mathbb{Y}^{2,1}$ and $X^{2,1}$ bounds, there exist a subsequence (not relabeled) and
\begin{equation}
    \label{eq:wherecpmlives}
    c_{\pm} \in L^\infty_t L^1_x(\Sigma_{\bar{T}}) \, ,\;  \p_x c_{\pm} , \p_t c_{\pm} \in L^\infty_t L^1_x(\Sigma_{\bar{T}}) \, , \quad b_- \in W^{1,1}(0,\bar{T}) \, ,
\end{equation}
such that
\begin{equation}
    c_{\pm}^n \to c_{\pm} \text{ in } C([0,\bar{T}] \times [0,R]) \, , \quad \forall R > 0
\end{equation}
\begin{equation}
    b_-^n \to b_- \text{ in } C([0,\bar{T}]) \, .
\end{equation}
We analyze also the convergence of $V^n$. Since
\begin{equation}
    V^n(x,t) = \int_{\R_+} K(x,y) \rho^n(y,t) \, dy + K(x,0) b_-^n(t) \, ,
\end{equation}
we have 
\begin{equation}
\begin{aligned}
       (V^n - V)(x,t) &= \int_{\R_+} K(x,y) (\rho^n - \rho)(y,t) \, dy \\
       &= \int_{B_R(x) \cap \R_+} K(x,y) (\rho^n - \rho)(y,t) \, dy + O(R^{-\alpha}) M \, ,
\end{aligned}
\end{equation}
from which we prove that $V^n \to V$ locally uniformly in $[0,\bar{T}] \times [0,+\infty)$. By weak-$\ast$ convergence, 
\begin{equation}
    V \in L^\infty_t W^{3,\infty}_x(\Sigma_{\bar{T}}) \, , \quad \p_t V \in  L^\infty_t W^{2,\infty}_x(\Sigma_{\bar{T}}) \, .
\end{equation}
Together, the above convergence results imply that $(c_+,c_-,b_-)$ is a solution with the desired data. We did not yet establish that $c_{\pm} \in Y^{1,1}(T)$, since $Y^{1,1}(T)$ entails time continuity, and we only established~\eqref{eq:wherecpmlives}; however, the uniqueness argument in Remark~\ref{rmk:tildeY11rmk}
yields that $(c_+,c_-,b_-)$ agrees with the unique $\mathbb{Y}^{1,1}(\bar{T})$ solution with velocity $V$. Moreover, we have
\begin{equation}
    \label{eq:gotmysecondVderivs}
    V, \p_t V, \p_t^2 V \in L^\infty_t W^{j, \infty}_x (\Sigma_{\bar{T}}) \, , \quad \forall j \in \N_0 \, ,
\end{equation}
by the \emph{a priori} estimates for $\mathbb{Y}^{1,1}$ solutions in Lemma~\ref{lem:aprioriestsunderdeltacondition}.

It remains to prove $\mathbb{Y}^{2,1}(\bar{T})$ regularity on the solution. This is not immediate from weak compactness, since in principle, when taking $n \to \infty$, the second derivatives may become measures. However, \emph{a posteriori}, Proposition~\ref{pro:X2psol} guarantees that there exists a time $T \leq \bar{T}$ on which the limiting solution belongs to $\mathbb{Y}^{2,1}(T)$. (We can invoke Proposition~\ref{pro:X2psol} using the information~\eqref{eq:gotmysecondVderivs} on $V$, which in particular controls two time derivatives of $g_b$ and and one time derivative of $f_b$.) This solution is then propagated to $\bar{T}$ by the \emph{a priori} estimates in Lemma~\ref{lem:aprioriestsunderdeltacondition}.
\end{proof}

We next demonstrate uniqueness.

\begin{proposition}[Uniqueness]
\label{pro:uniqueness} Assume the hypotheses of Proposition~\ref{pro:existence} and let $T > 0$. Suppose that $(c_+^{(k)},c_-^{(k)},b_-^{(k)})$, $k=1,2$, are $\mathbb{Y}^{1,1}$ solutions to the system~\eqref{eq:inviscidsystem3}-\eqref{eq:Vlaw3} on $\Sigma_T$, equal at time $t=0$ and such that $\beta_++V^{(k)}\big|_{x=0}>0$, $-\beta_-+V^{(k)}\big|_{x=0}<0$, $k=1,2$ on $[0, T]$. Then
\begin{equation}
    (c_+^{(1)},c_-^{(1)},b_-^{(1)}) \equiv (c_+^{(2)},c_-^{(2)},b_-^{(2)}) \text{ on } \Sigma_T \, .
\end{equation}
\end{proposition}

\begin{proof}[Proof of Proposition~\ref{pro:uniqueness}]

Defining the differences
\begin{align}
    \bar c_\pm = c_\pm^{(1)}-c_\pm^{(2)}\,, \quad
    \bar b = b_{-}^{(1)}-b_{-}^{(2)}\,, 
\end{align}
we first note that $\bar c_{\pm}$ satisfy 
\begin{equation}
\begin{aligned}
    \p_t \bar c_{+} + \p_x ((V^{(1)}+\beta_+) \bar c_{+}) &=  -\partial_x (c_{+}^{(2)} (V^{(1)}-V^{(2)})) + R (c_{\pm}^{(1)}, c_{\pm}^{(2)}) \\
    \p_t \bar c_{-} + \p_x ((V^{(1)}-\beta_-) \bar c_-) &= -\partial_x (c_{-}^{(2)} (V^{(1)}-V^{(2)})) - R (c_{\pm}^{(1)}, c_{\pm}^{(2)})\,.
\end{aligned}
\end{equation}
Here, using the trace inequality
\begin{equation}
    \|c_{\pm}^{(j)}\|_{  L^{\infty} ((0, T) \times \R_{+}) } \le 2\|c_{\pm}^{(j)}\|_{ Y^{1, 1} (T) }\,,
\end{equation}
we have 
\begin{equation}
\begin{aligned}
  R(c_{\pm}^{(1)}, c_{\pm}^{(2)}) &=c^{(1)}_{-} r (c^{(1)}_{+})-c^{(2)}_{-} r (c^{(2)}_{+}) - c^{(1)}_{+} r (c^{(1)}_{-}) + c^{(2)}_{+} r (c^{(2)}_{-}) \\
  & = \bar c_{-} r (c^{(1)}_{+}) + c^{(2)}_{-}\big(r (c^{(1)}_{+})-r (c^{(2)}_{+})\big) \\
  &\qquad - \bar c_{+}  r (c^{(1)}_{+}) + c^{(2)}_{+} \big(r (c^{(2)}_{-})-r (c^{(1)}_{-})\big)\\
  & = O (|\bar c_{+}|+|\bar c_{-}|)\,.
\end{aligned}
\end{equation}

 Furthermore, $\bar b_-$ satisfies
\begin{equation}
\label{eq:barbeqn}
\begin{aligned}
        &\frac{d\bar b_{-}}{dt}  = - (V^{(1)} + \beta_+) \big(g_b(b^{(1)}_-, V^{(1)})  -g_b(b^{(2)}_-, V^{(2)}) \big) \\
        &\quad + g_b(b^{(2)}_-, V^{(2)}) \big( V^{(2)} - V^{(1)} \big)  + (\beta_- - V^{(1)}) \bar c_- + c^{(2)}_- \big(V^{(2)} - V^{(1)} \big)
\end{aligned}
\end{equation}
with $\bar b_{-}(0) =0$, where we temporarily omit the evaluation notation $\big|_{x=0}$. By \eqref{eq:gconditionsforquasilinearity}, we may bound
\begin{equation} \label{eq:gbdiff}
    \abs{g_b(b^{(1)}_-,V^{(1)}\big|_{x=0}) -g_b(b^{(2)}_-, V^{(2)}\big|_{x=0})} \lesssim \abs{\bar b_-(t)} + \abs{V^{(1)}\big|_{x=0}-V^{(2)}\big|_{x=0}}\,,
\end{equation}
so that, upon integrating \eqref{eq:barbeqn}, we have 
\begin{align}\label{eq:barbbnd}
 \abs{\bar b_{-} (t)} \lesssim  \int_0^t \abs{\bar b_{-}(s)}\, ds + \int_0^t  \|\bar{\bm c}\|_{ L^1_x } \, ds+ \int_0^t \abs{\bar c_-\big|_{x=0} (s)} \, ds.
\end{align}

Next, by the $L^1$ estimates \eqref{eq2.6.1} and \eqref{eq2.6.2}, for $t \in [0, T]$, we have 
\begin{equation}\label{eq:barcplusbd}
\begin{aligned}
    \|\bar c_{+} (t, \cdot)\|_{L^1 (\R_{+})} &\le N \big(\|\abs{\bar c_{+}}+\abs{\bar c_{-}}\|_{ L^1 ((0, T) \times \R_{+}) } + \mathcal{E}_1 (t) \\
    &\qquad + \|(V^{(1)}\big|_{x=0}+\beta_{+}) \abs{\bar c_{+}\big|_{x=0}}\|_{L^1 (0, t)} \big)
\end{aligned}
\end{equation}
as well as
\begin{equation}\label{eq:barcminusbd}
\begin{aligned}
    & \|\bar c_{-} (t, \cdot)\|_{L^1 (\R_{+})} + \|\bar c_{-}\big|_{x=0}\|_{ L^1 (0, t) } \, \min_{[0,T]} \big(\beta_{-} - V^{(2)}\big|_{x=0}\big)  \\
   &\qquad \le N \left(\|\abs{\bar c_{+}}+\abs{\bar c_{-}}\|_{ L^1 ((0, T) \times \R_{+}) } + \mathcal{E}_2 (t) \right) \,,
 \end{aligned}
 \end{equation}
 where $N = N(V, T, \beta_{\pm}, M)$ and
 \begin{equation}
     \mathcal{E}_j (t)=\|\partial_x \big(c_{+}^{(j)} (V^{(1)}-V^{(2)})\big)\|_{  L^1 ((0, t) \times \R_{+})}\,.
 \end{equation}
By using the boundary condition \eqref{eq:cinv_plus_bvalcase2_sec5} for $c_{+}$ and \eqref{eq:gbdiff}, we may bound
\begin{equation}\label{eq:barcplustrace}
\begin{aligned}
    & \|(V^{(1)}\big|_{x=0}+\beta_{+}) \abs{\bar c_{+}\big|_{x=0}}\|_{L^1 (0, t)} \\  
    &\quad  \lesssim  \int_0^t \abs{\bar b_{-}(s)}\, ds +  \int_0^t  \|\bar{\bm c}\|_{ L^1_x } \, ds+ \int_0^t \abs{\bar c_-\big|_{x=0} (s)} \, ds\,.
\end{aligned}
\end{equation}

Gathering \eqref{eq:barbbnd}-\eqref{eq:barcplustrace}, we thus obtain
\begin{equation}\label{eq:barsineq}
\begin{aligned}
  &  \abs{\bar b_{-}(t)} + \|\abs{\bar c_{+}(t, \cdot)}+ \abs{\bar c_{-} (t, \cdot)}\|_{L^1 (\R_{+})}  \\
  &\qquad \le N \left(\|\abs{\bar c_{+}}+ \abs{\bar c_{-}}\|_{ L^1 ((0, T) \times \R_{+}) } + \|\bar b_{-}\|_{L^1 (0, t)} + (\mathcal{E}_1+  \mathcal{E}_2) (t)\right)\,.
\end{aligned}
\end{equation}
By Gronwall's inequality, we may replace the right-hand side of \eqref{eq:barsineq} with  $N (\mathcal{E}_1+  \mathcal{E}_2) (t)$.
Furthermore, if $\bm{c}^{(j)} \in Y^{1, 1} (T)$, then a simple argument gives
\begin{equation}
\begin{aligned}
   \mathcal{E}_j (t) 
   & \le  \|\bm{c}^{(j)}\|_{ L^{\infty} (0, T; W^{1, 1} (\R_{+})) } \|V^{(1)}-V^{(2)}\|_{ L^{1} (0, t; W^{1, \infty} (\R_{+})) } \,. 
\end{aligned}
\end{equation}

We therefore obtain the bound
\begin{equation}\label{eq:bdbyVdiff2_again}
\begin{aligned}
 & |(b^{(1)}_{-}-b^{(2)}_{-}) (t)|   + \|(\bm{c}^{(1)}-\bm{c}^{(2)}) (t, \cdot)\|_{ L^1 (\R_{+}) }  \\
 &\qquad \le N \|V [c^{(1)}_{+}+c^{(1)}_{-}, b^{(1)}_{-}]-V [c^{(2)}_{+}+c^{(2)}_{-}, b^{(2)}_{-}]\|_{ L^{1} (0, t; W^{1, \infty} (\R_{+})) }
\end{aligned}
\end{equation}
for $t \in [0, T]$, where $N$ is independent of $t$.
Using the properties of $V$, we may conclude that the right-hand side of \eqref{eq:bdbyVdiff2_again} is bounded by
\begin{equation}
    N \big( \|\bm{c}^{(1)}-\bm{c}^{(2)}\|_{ L^1 ((0, t) \times \R_{+}) } + \|b^{(1)}_{-}-b^{(2)}_{-}\|_{ L^1 (0, t) }\big)\,.
\end{equation}
The desired assertion thus follows from Gronwall's inequality.
\end{proof}

With uniqueness in hand, we can now prove Proposition~\ref{pro:inviscidsolutionexists} (Local well-posedness).

\begin{proof}[Proof of Proposition~\ref{pro:inviscidsolutionexists}.]

In Case~1, we define
\begin{equation}
    \label{eq:Tstardefcase1}
    T^* := \sup \{ T > 0 : \exists (c_+,c_-,b_-) \in \mathbb{Y}^{2,1}(T) \text{ solution to  \eqref{eq:inviscidequation}-\eqref{eq:V0condition}} \} \, ,
\end{equation}
where ``solution" entails the $\delta$-bounds~\eqref{eq:lowerboundstobepropagated} on $[0,T]$ for some $\delta$. Recall that $c_\infty = r_0 b_-/(\beta_+ + V|_{x=0})$ is explicit. Suppose that the initial condition satisfies the $2\delta$-bound~\eqref{eq:initialguysatisfies2deltabound}. To demonstrate that the set on the right-hand side of~\eqref{eq:Tstardefcase1} is non-empty, we modify the boundary condition $c_+|_{x=0} = r_0b_-/(\beta_+ + V|_{x=0})$ ($= c_\infty$ in Case~1) by choosing
\begin{equation}
    \label{eq:gbmodification}
    g_b(b_-,V) := \frac{r_0 b_-}{\beta_+ + V\big|_{x=0}} \chi(\beta_+ + V) \, ,
\end{equation}
where $\chi(z) \equiv 1$ when $z \geq \delta$ and $\chi(z) \equiv 0$ when $z \leq \delta/2$. This $g_b$ satisfies the assumptions (see~\eqref{eq:gconditionsforquasilinearity}) in Proposition~\ref{pro:existence}, which thereby guarantees a short-time solution to the system~\eqref{eq:inviscidsystem3}-\eqref{eq:Vlaw3} satisfying the $\delta$-bound~\eqref{eq:lowerboundstobepropagated}. Hence, $g_b = c_\infty$ on the support of the solution, so it is a solution of the original system.

By uniqueness, any of the solutions agree on a common time of existence, so we speak of \emph{the} solution on $[0,T^*) \times \R_+$. To complete Case~1, it remains to characterize $T^*$ by demonstrating (for the sake of contradiction) that, if $T^* < +\infty$ but the $\delta$-bounds~\eqref{eq:lowerboundstobepropagated} are satisfied on $[0,T^*)$ for some $\delta > 0$, then $(c_+,c_-,b_-) \in \mathbb{Y}^{2,1}(T^*)$ and thus the solution can be continued via Proposition~\ref{pro:existence} (the desired contradiction). This falls under the \emph{a priori} bounds in Lemma~\ref{lem:aprioriestsunderdeltacondition}.

In Case~2, the $\delta$-bounds~\eqref{eq:lowerboundstobepropagated} are automatically satisfied since $V\big|_{x=0} = 0$. Instead, we deal with the separate issue that $c_\infty$ is not globally defined. Let
\begin{equation}
\begin{aligned}
    T^* := \sup \{ T > 0 : \exists (c_+,c_-,b_-) \in \mathbb{Y}^{2,1}(T) \text{ solution to }&\text{\eqref{eq:inviscidequation}-\eqref{eq:Vinviscid}}, \\ &\text{\eqref{eq:bminuseqncase2}-\eqref{eq:cinv_plus_bvalcase2}} \} \, ,
\end{aligned}
\end{equation}
where ``solution" entails that $q(t) := b_-(t)/\beta_- < q^*(\beta_+,\beta_-)$ on $[0,T]$. The set on the right-hand side is non-empty, since we may apply Proposition~\ref{pro:existence} with
\begin{equation}
    \label{eq:gbmodification2}
    g_b(b_-) = \chi(q^* - b_-) c_\infty(q,\beta_+,\beta_-) \, ,
\end{equation}
where $\chi(z) \equiv 1$ when $z \leq -\delta$ and $\chi(z) \equiv 0$ when $z \geq -\delta/2$ for sufficiently small $\delta$; because $b_-(t)$ is continuous-in-time, we will have that $g_b = c_\infty$ on the support of the solution for sufficiently small time. To complete Case~2, we characterize $T^*$ by demonstrating that, if $T^* < +\infty$ but $q(t) \leq q^*(\beta_+,\beta_-) - \delta$ on $[0,T^*)$ for some $\delta > 0$, then $(c_+,c_-,b_-) \in \mathbb{Y}^{2,1}(T^*)$ and thus the solution can be continued by Proposition~\ref{pro:existence} (contradiction). This also follows from Lemma~\ref{lem:aprioriestsunderdeltacondition}.
\end{proof}

\begin{corollary}[Small data GWP]
Let $M := \sum_{\pm} \int_{\R_+} c_{\pm}^{\rm in} \, dy + b^{\rm in}_-$ be the initial mass. Suppose that Case~1 holds and
\begin{equation}
    \beta_+ + M \inf_y K(0,y) > 0 \quad \text{ and } \quad M \sup_y K(0,y) - \beta_- < 0 \, ,
\end{equation}
or that Case~2 holds and
\begin{equation}
    \frac{M}{\beta_-} < q^*(\beta_+,\beta_-)
\end{equation}
(in particular, this holds when $r' \leq 0$ in Case~2). Then $T^* = +\infty$.
\end{corollary}

\subsubsection*{Acknowledgments}
DA was supported by NSF grant DMS-2406947, the Office of the Vice Chancellor for Research and Graduate Education at UW–Madison with funding from the Wisconsin Alumni Research Foundation, and a Sloan Fellowship. LO acknowledges support from NSF grant DMS-2406003. TY was partially supported by the National Science and Technology Council of Taiwan grant number 114-2115-M-001-011-MY3.

\subsubsection*{AI Statement} Gemini 3 (Pro for Education subscription through UW-Madison) produced Matlab code to generate Figure~\ref{fig:ODEforA}; suggested the odd reflection kernel~\eqref{eq:reflectedkernel}; helped to search for relevant literature; helped DA double-check calculations related to characteristics in Section~\ref{sec:outerproblem}; and aided the authors in proofreading the near-final draft. GPT-5.6 Sol also aided the authors in proofreading the near-final draft.

\subsubsection*{Conflict of Interest Statement} The authors have no conflicts of interest to report.

\bibliographystyle{abbrv}
\bibliography{bibliography.bib}

\end{document}